\documentclass[11pt]{article}

\usepackage[letterpaper,top=2cm,bottom=2cm,left=2cm,right=2cm,marginparwidth=1.75cm]{geometry}

\usepackage[dvipsnames]{xcolor}

\usepackage{amsmath, amsfonts, mathtools}
\usepackage{amsthm} \usepackage{graphicx}

\usepackage{orcidlink}
\usepackage{multirow}

\usepackage{hyperref}
\hypersetup{hidelinks}
\usepackage[noblocks]{authblk}
\usepackage{float}
\usepackage{booktabs}
\usepackage{array}
\newcolumntype{C}[1]{>{\centering\arraybackslash}m{#1}}

\newcommand{\px}{\partial_{x}}
\newcommand{\py}{\partial_{y}}

\newcommand{\pa}{{\partial}}

\newcommand{\ap}{\alpha}

\newcommand{\util}{\Tilde{u}}

\newcommand{\vtil}{\Tilde{v}}

\newcommand{\bbu}{\boldsymbol{u}}
\newcommand{\bbg}{\boldsymbol{g}}

\newcommand{\tU}{\widetilde{U}}
\newcommand{\tA}{\widetilde{A}}

\newcommand{\iRe}{\mathrm{Re}^{-1}}
\newcommand{\mathR}{\mathbb{R}}
\newcommand{\md}{{\,\rm d}}

\newcommand{\fl}[2]{\frac{#1}{#2}}

\newtheorem{definition}{Definition}[section]
\newtheorem{theorem}{Theorem}[section]
\newtheorem{remark}{Remark}[section]
\newtheorem{proposition}{Proposition}[section]

\newcommand{\diag}{\mbox{diag}}
\newcommand{\brac}[1]{\left(#1\right)}

\newtheorem{lem}{Lemma}[section]
\newtheorem{lemma}{Lemma}[section]

\theoremstyle{definition}

\numberwithin{equation}{section}

\title{Two-Dimensional Shallow Water Linearized Moment Equations: Hyperbolicity and Well-Balanced Schemes}

\author[1]{Shiping Zhou \orcidlink{0000-0002-3247-0328}\thanks{Corresponding author: \texttt{zhouship@msu.edu}}}
\author[2]{Juntao Huang \orcidlink{0000-0003-0527-7431}}
\author[1]{Andrew J. Christlieb \orcidlink{0000-0002-5395-5455}}

\affil[1]{\small{Computational Mathematics, Science and Engineering, Michigan State University, East Lansing, MI 48824, USA}}
\affil[2]{\small{Department of Mathematical Sciences, University of Delaware, Newark, DE 19716, USA}}

\date{September 23, 2026}

\begin{document}
\maketitle

\bigskip
\noindent
{\bf Abstract.}
{
We study rotational invariance, hyperbolicity, and well-balanced discretization of shallow water linearized moment equations in two horizontal dimensions.
Although the direct extension has real characteristic speeds, its directional matrices can lack a complete eigenbasis.
We construct a rotationally invariant modification and prove its global hyperbolicity for all states with positive water depth.
For quasi-two-dimensional flows over non-flat topography with Navier-slip friction, we develop first- and second-order path-conservative finite-volume schemes that exactly preserve a discrete moving equilibrium constructed by midpoint collocation.
Numerical tests confirm equilibrium preservation to roundoff, second-order approximation of smooth stationary branches, and near-second-order accuracy of the second-order scheme for the tested perturbations.
Comparisons with three-dimensional two-phase OpenFOAM simulations illustrate the invariant shallow-water limit and show that retaining moments improves the reconstruction of curved rotational velocity profiles, while differences remain in radial motion.
}

{\bf { Keywords}:} {shallow water moment equations, rotational invariance, global hyperbolicity, well-balanced schemes, moving equilibria, path-conservative methods.}

{\bf { Mathematics Subject Classification (2020)}:} { 35L65, 65M08, 65M12, 76M12.}

\section{Introduction}
The shallow water equations (SWE) are classical depth-averaged models for free-surface flows, with applications in open-channel hydraulics and natural-hazard assessment \cite{French1985,Christen2010}.
Their computational economy comes from replacing the three-dimensional velocity field by the water depth and depth-averaged horizontal velocity.
The resulting equations do not resolve vertical variations of the horizontal velocities or the associated contributions to momentum transport.

Multilayer shallow water equations, also called multilayer Saint--Venant models, provide an alternative description of vertical structure through layerwise averaged velocities.
These models introduce interlayer exchange or coupling and pose model-dependent challenges concerning hyperbolicity and the treatment of nonconservative terms at interfaces \cite{Audusse2005Multilayer,Audusse2011Multilayer,BouchutZeitlin2010}.
Shallow water moment equations (SWME) recover part of this information by expanding the horizontal velocity about its depth average in a finite Legendre basis and evolving the expansion coefficients \cite{Kowalski2019}.
A common framework combining layerwise and moment approximations is provided by multilayer-moment models \cite{Garres-Diaz2023}.
Applications of moment models include bedload transport, where near-bottom velocity is relevant \cite{Garres-Diaz2021}, and two-dimensional curved shallow flows with secondary circulation \cite{Steldermann2023}.
This moment representation of the velocity has also been combined with an effective wall closure for no-slip bottoms \cite{Zhou2025a}.

The loss of global hyperbolicity in higher-order SWME motivated the development of hyperbolic shallow water moment equations (HSWME) and shallow water linearized moment equations (SWLME) \cite{Koellermeier2020,Koellermeier2022}.
The SWLME simplify selected nonlinear couplings in the higher-moment equations while retaining the quadratic moment contributions to the mean-momentum flux \cite{Koellermeier2022}.
Hyperbolic regularization in primitive variables provides another approach to balancing hyperbolicity, momentum consistency, and tractable stationary states \cite{Koellermeier2025}.
In two horizontal dimensions, hyperbolicity must hold in every direction, and rotational invariance provides the similarity relation that reduces this directional analysis to an $x$-direction matrix \cite{Bauerle2025}.
Bauerle et al.\ developed rotationally invariant hyperbolic modifications of the HSWME family.

For these moment models, an important numerical challenge is to accurately resolve flows near steady states, where transport terms balance topography and friction sources.
Well-balanced schemes address this challenge by preserving the corresponding discrete equilibria \cite{Bermudez1994,Greenberg1996,LeVeque1998,Vazquez-Cendon1999,Audusse2004}.
For the SWE, methods preserving moving-water equilibria include high-order finite-volume WENO and discontinuous Galerkin schemes \cite{Noelle2007,Xing2014}, and their importance for resolving small perturbations of moving flows has been demonstrated in \cite{XingShuNoelle2011}.
For moment systems, this balance also involves the moment variables and the nonconservative terms in their evolution equations \cite{Koellermeier2022,Pimentel-Garcia2024}.
Path-conservative methods provide a framework for treating these nonconservative products \cite{DalMaso1995,Pares2006b}.
Methods based on reconstructing deviations from a supplied stationary profile can preserve that profile \cite{Klingenberg2019,Berberich2021}, whereas local stationary reconstruction methods seek an equilibrium matching the evolving cell data \cite{Castro2020,GomezBuenoControl2021}.
Collocation-based stationary solvers extend the latter construction to cases without analytic equilibrium formulas \cite{Gomez-Bueno2021a}.

For one-dimensional SWLME, Koellermeier and Pimentel-Garc\'ia characterized frictionless moving equilibria over topography and constructed first- and second-order well-balanced finite-volume schemes \cite{Koellermeier2022}.
Pimentel-Garc\'ia extended fully well-balanced reconstruction to smooth topography with friction \cite{Pimentel-Garcia2024}.
Within this one-dimensional setting, recent methods include a semi-implicit exactly fully well-balanced relaxation method for low-Froude SWLME \cite{Caballero-Cardenas2025}, a second-order flux-globalized central-upwind method for frictional hyperbolic SWLME \cite{Cao2026}, and high-order equilibrium-preserving discontinuous Galerkin methods for SWLME \cite{Fan2026}.
For cases in which analytic steady states are unavailable, Ciallella and Koellermeier developed high-order global-flux WENO methods that approximately preserve general steady states of shallow water moment systems with nonconservative products \cite{Ciallella2026}.
A complementary development is entropy-stable, lake-at-rest well-balanced discontinuous Galerkin discretization of one-dimensional moment equations \cite{Careaga2026}.
Two-dimensional shallow-moment computations with lake-at-rest well-balancing also appear in \cite{Steldermann2023}; these should be distinguished from preservation of frictional moving equilibria.

Extending the SWLME to two horizontal dimensions also raises a structural difficulty, since it introduces additional coupling between the moment variables associated with the two velocity components.
As we show in Section~\ref{sec:sub:hyper-SWLME}, the resulting system can have real characteristic speeds without a complete eigenbasis and therefore fails to be globally hyperbolic.
To address this difficulty, we establish rotational invariance of the two-dimensional SWLME and construct a rotationally invariant, globally hyperbolic modification, denoted by $G\text{-}\mathrm{SWLME}$.
For every retained order $N\geq1$ and positive depth, the modified directional matrices have a complete real eigenbasis, including at the defective states of the direct extension.

For quasi-two-dimensional flows, which vary in one horizontal coordinate while retaining both horizontal velocity components and their moment families, we construct a discrete moving equilibrium by midpoint collocation.
The construction uses prescribed boundary data and the standard projected Navier-slip source on smooth stationary branches for which the stationary coefficient matrix remains invertible.
Building on path-conservative discretization, equilibrium-deviation reconstruction, and collocation \cite{Pares2006b,Klingenberg2019,Gomez-Bueno2021a}, we develop first- and second-order finite-volume schemes that exactly preserve this discrete equilibrium in exact arithmetic.
This preservation property is distinct from the second-order approximation of the continuous stationary branch and is restricted to the prescribed quasi-two-dimensional equilibrium.
Numerical experiments assess moment dynamics, stationary-profile approximation, equilibrium preservation, and perturbation accuracy, while comparisons with three-dimensional two-phase OpenFOAM simulations illustrate model differences in radial collapse.

The remainder of the paper is organized as follows.
{
Section~\ref{sec:SWME} reviews the hydrostatic formulation and the two-dimensional shallow water moment equations with the standard projected Navier-slip source.
Section~\ref{sec:hyperbolicity} introduces the two-dimensional SWLME and establishes rotational invariance and hyperbolicity results.
}
Section~\ref{sec:well-balanced} develops first- and second-order well-balanced schemes for a prescribed discrete equilibrium.
Section~\ref{sec:tests} presents the numerical evidence and discusses its limitations.

\section{Shallow water moment models}
\label{sec:SWME}

{
We review the hydrostatic formulation and moment expansion of \cite{Kowalski2019}, using the two-dimensional matrix formulation of \cite{Bauerle2025}.
}

\subsection{Hydrostatic equations in normalized vertical coordinates}
\label{subsec:normalized-system}

We begin with the three-dimensional incompressible Navier--Stokes equations
\begin{equation*}
\left\{
\begin{aligned}
    \nabla\cdot \bbu &= 0, \\
    \fl{\pa\bbu}{\pa t} + \nabla\cdot(\bbu\bbu) &= -\fl{1}{\rho}\nabla p + \fl{1}{\rho}\nabla\cdot\sigma + \bbg,
\end{aligned}
\right.
\end{equation*}
where $\bbu=(u,v,w)^T$ is the velocity, $p$ is the pressure, $\sigma$ is the deviatoric stress tensor, and the density $\rho$ is constant.
For a Newtonian fluid, $\sigma=\mu(\nabla\bbu+\nabla\bbu^T)$, where $\mu$ is the dynamic viscosity.
The gravitational acceleration is $\bbg=(0,0,-g)^T$.

At the free surface $z=h_s(t,x,y)$, with outward unit normal $\boldsymbol{n}_s$, the material and stress-free conditions are
\begin{equation}\label{bc:free-surface-explicit}
    \fl{\pa h_s}{\pa t}+u\fl{\pa h_s}{\pa x} + v\fl{\pa h_s}{\pa y}=w,
    \qquad
    (-pI_3+\sigma)\boldsymbol{n}_s=\boldsymbol{0},
    \qquad \mbox{on } z=h_s.
\end{equation}

At the smooth stationary bottom $z=h_b(x,y)$, we impose non-penetration and componentwise Navier slip \cite{Kowalski2019}:
\begin{equation*}
    \bbu\cdot\boldsymbol{n}_b=0,
    \qquad
    \left.(\kappa u-\sigma_{xz})\right|_{z=h_b}=0,
    \qquad
    \left.(\kappa v-\sigma_{yz})\right|_{z=h_b}=0.
\end{equation*}
Here $\boldsymbol{n}_b$ is the outward unit normal and $\kappa$ is the friction coefficient.

Let $L$, $H$, and $U$ be the characteristic horizontal length, depth, and horizontal velocity scales.
We set
\begin{equation*}
    \varepsilon=\frac{H}{L}\ll 1.
\end{equation*}
We assume that the bottom varies on the horizontal scale $L$, so that $|\partial_x h_b|+|\partial_y h_b|=\mathcal{O}(\varepsilon)$.
These slip conditions are exact for a flat bottom and approximate the tensorial Navier law to leading order for smooth mild slopes.
Nondimensionalization and neglect of higher-order terms in $\varepsilon$ yield the hydrostatic shallow-flow system \cite{Kowalski2019,Zhou2025a}.
We nondimensionalize the horizontal velocities as $u/U$ and $v/U$ and retain the notation $u$ and $v$.
We then introduce the normalized vertical coordinate
\begin{equation*}
    \zeta=\frac{z-h_b(x,y)}{h(t,x,y)},
    \qquad h=h_s-h_b,
\end{equation*}
which maps the physical interval $h_b\le z\le h_s$ to $0\le\zeta\le 1$.
For a field $\theta=\theta(t,x,y,z)$, we denote its mapped counterpart by
\begin{equation*}
    \widetilde{\theta}(t,x,y,\zeta)
    =\theta(t,x,y,h\zeta+h_b).
\end{equation*}
The tilde changes only the vertical coordinate; $\widetilde u$ and $\widetilde v$ remain nondimensional in this derivation.

The vertical velocity is scaled with $\varepsilon U$.
The mapping induces the contravariant vertical transport velocity $\omega$, defined by
\begin{equation*}
    h\omega
    =\widetilde{w}
    -\fl{\pa(h\zeta+h_b)}{\pa t}
    -\util\fl{\pa(h\zeta+h_b)}{\pa x}
    -\vtil\fl{\pa(h\zeta+h_b)}{\pa y}.
\end{equation*}
The kinematic boundary conditions imply $\omega(t,x,y,0)=\omega(t,x,y,1)=0$.
The resulting vertically normalized system is
\begin{equation}\label{eq:normalized-hydrostatic-system}
\left\{
\begin{aligned}
    \fl{\pa h}{\pa t}
    + \fl{\pa (h u_m)}{\pa x}
    + \fl{\pa (h v_m)}{\pa y}
    &=0, \\
    \fl{\pa (h\util)}{\pa t}
    + \fl{\pa}{\pa x}\left(h\util^2 + \fl{G}{2}h^2\right)
    + \fl{\pa}{\pa y}\left(h\util\vtil\right)
    + \fl{\pa}{\pa\zeta}\left(h\util\omega-\fl{\iRe}{\varepsilon h}\fl{\pa\util}{\pa\zeta}\right)
    &=-Gh\fl{\pa h_b}{\pa x}, \\
    \fl{\pa (h\vtil)}{\pa t}
    + \fl{\pa}{\pa x}\left(h\util\vtil\right)
    + \fl{\pa}{\pa y}\left(h\vtil^2 + \fl{G}{2}h^2\right)
    + \fl{\pa}{\pa\zeta}\left(h\vtil\omega-\fl{\iRe}{\varepsilon h}\fl{\pa\vtil}{\pa\zeta}\right)
    &=-Gh\fl{\pa h_b}{\pa y}.
\end{aligned}
\right.
\end{equation}
The Froude and Reynolds numbers are $\mathrm{Fr}=U/\sqrt{gH}$ and $\mathrm{Re}=\rho UH/\mu$, respectively.
The dimensionless coefficients in the normalized system are
\begin{equation*}
    G=\mathrm{Fr}^{-2}=\fl{gH}{U^2},
    \qquad
    \iRe=\fl{\mu}{\rho UH}.
\end{equation*}

The dimensionless wall-friction coefficient is $\gamma=\kappa/(\rho U)$, with $\gamma=0$ corresponding to perfect slip.
In normalized coordinates, the stress-free and Navier slip conditions become
\begin{equation*}
\begin{aligned}
    \left.\fl{\pa\util}{\pa\zeta}\right|_{\zeta=1}
    =\left.\fl{\pa\vtil}{\pa\zeta}\right|_{\zeta=1}=0, \qquad
    \left.\fl{\iRe}{\varepsilon h}\fl{\pa\util}{\pa\zeta}\right|_{\zeta=0}
    =\fl{\gamma}{\varepsilon}\util(t,x,y,0),
    \qquad
    \left.\fl{\iRe}{\varepsilon h}\fl{\pa\vtil}{\pa\zeta}\right|_{\zeta=0}
    =\fl{\gamma}{\varepsilon}\vtil(t,x,y,0).
\end{aligned}
\end{equation*}

\subsection{Shallow water moment equations}
\label{subsec:MSWME}

Following the notation of \cite{Kowalski2019}, we define the depth-averaged velocities by
\begin{equation*}
    u_m(t,x,y)=\int_0^1 \util(t,x,y,\zeta)\,\md\zeta,
    \qquad
    v_m(t,x,y)=\int_0^1 \vtil(t,x,y,\zeta)\,\md\zeta.
\end{equation*}

We approximate the vertical profiles of the horizontal velocities using Legendre expansions
\begin{align*}
    \util(t,x,y,\zeta)
    &=u_m(t,x,y)+\sum_{j=1}^{N}\alpha_j(t,x,y)\phi_j(\zeta), \\
    \vtil(t,x,y,\zeta)
    &=v_m(t,x,y)+\sum_{j=1}^{N}\beta_j(t,x,y)\phi_j(\zeta).
\end{align*}

Here $\phi_j$ are the scaled Legendre polynomials on $[0,1]$, normalized by $\phi_j(0)=1$, with
\begin{equation*}
    \phi_1(\zeta)=1-2\zeta,
    \qquad
    \phi_2(\zeta)=1-6\zeta+6\zeta^2.
\end{equation*}

They satisfy
\begin{equation*}
    \int_0^1\phi_i\,\md\zeta=0,
    \qquad
    \int_0^1\phi_i\phi_j\,\md\zeta=\fl{\delta_{ij}}{2i+1},
    \qquad i,j=1,\ldots,N.
\end{equation*}

Thus, $\alpha_j$ and $\beta_j$ describe vertical deviations from the depth averages $u_m$ and $v_m$.
Since $\phi_j(0)=1$, the reconstructed bottom velocities are
\begin{equation}\label{eq:reconstructed-bottom-velocities}
    u_b^{(N)}=\util(t,x,y,0)=u_m+\sum_{j=1}^{N}\alpha_j,
    \qquad
    v_b^{(N)}=\vtil(t,x,y,0)=v_m+\sum_{j=1}^{N}\beta_j.
\end{equation}
Integrating the mass equation in the vertically normalized system over $\zeta\in[0,1]$ and using $\omega|_{\zeta=0,1}=0$ gives the depth-averaged mass equation.
Integrating the two momentum equations with unit weight gives the mean-momentum equations, while projection against $(2i+1)\phi_i$ gives the higher-moment equations.
Using the orthogonality relations and the normalized boundary conditions, the shallow water moment equations are
\begin{align*}
    \fl{\pa h}{\pa t}
    + \fl{\pa (hu_m)}{\pa x}
    + \fl{\pa (hv_m)}{\pa y}
    &=0, \\
    \fl{\pa (hu_m)}{\pa t}
    + \fl{\pa}{\pa x}\left(h\left(u_m^2+\sum_{j=1}^{N}\fl{\alpha_j^2}{2j+1}\right)+\fl{G}{2}h^2\right)
    + \fl{\pa}{\pa y}\left(h\left(u_mv_m+\sum_{j=1}^{N}\fl{\alpha_j\beta_j}{2j+1}\right)\right)
    &=-\fl{\gamma}{\varepsilon}u_b^{(N)}-Gh\fl{\pa h_b}{\pa x}, \\
    \fl{\pa (hv_m)}{\pa t}
    + \fl{\pa}{\pa x}\left(h\left(u_mv_m+\sum_{j=1}^{N}\fl{\alpha_j\beta_j}{2j+1}\right)\right)
    + \fl{\pa}{\pa y}\left(h\left(v_m^2+\sum_{j=1}^{N}\fl{\beta_j^2}{2j+1}\right)+\fl{G}{2}h^2\right)
    &=-\fl{\gamma}{\varepsilon}v_b^{(N)}-Gh\fl{\pa h_b}{\pa y},
\end{align*}
with the moment equations, for $i=1,\ldots,N$,
\begin{align*}
\nonumber
    \fl{\pa (h\alpha_i)}{\pa t}
    &+\fl{\pa}{\pa x}\left(h\left(2u_m\alpha_i+\sum_{j,k=1}^{N}A_{ijk}\alpha_j\alpha_k\right)\right)
    +\fl{\pa}{\pa y}\left(h\left(u_m\beta_i+v_m\alpha_i+\sum_{j,k=1}^{N}A_{ijk}\alpha_j\beta_k\right)\right) \\
    &=u_mD_i-\sum_{j,k=1}^{N}B_{ijk}D_j\alpha_k
    -(2i+1)\left(\fl{\gamma}{\varepsilon}u_b^{(N)}+\fl{\iRe}{\varepsilon h}\sum_{j=1}^{N}C_{ij}\alpha_j\right), \\
\nonumber
    \fl{\pa (h\beta_i)}{\pa t}
    &+\fl{\pa}{\pa x}\left(h\left(u_m\beta_i+v_m\alpha_i+\sum_{j,k=1}^{N}A_{ijk}\alpha_j\beta_k\right)\right)
    +\fl{\pa}{\pa y}\left(h\left(2v_m\beta_i+\sum_{j,k=1}^{N}A_{ijk}\beta_j\beta_k\right)\right) \\
    &=v_mD_i-\sum_{j,k=1}^{N}B_{ijk}D_j\beta_k
    -(2i+1)\left(\fl{\gamma}{\varepsilon}v_b^{(N)}+\fl{\iRe}{\varepsilon h}\sum_{j=1}^{N}C_{ij}\beta_j\right).
\end{align*}
The coefficients in the moment system are
\begin{align*}
    A_{ijk}&=(2i+1)\int_0^1\phi_i\phi_j\phi_k\md\zeta,
    && i,j,k=1,\ldots,N, \\
    B_{ijk}&=(2i+1)\int_0^1\phi_i'\left(\int_0^\zeta\phi_j\md\widehat{\zeta}\right)\phi_k\md\zeta,
    && i,j,k=1,\ldots,N, \\
    C_{ij}&=\int_0^1\phi_i'\phi_j'\md\zeta,
    && i,j=1,\ldots,N,
\end{align*}
and
\begin{equation*}
    D_i=\fl{\pa (h\alpha_i)}{\pa x}+\fl{\pa (h\beta_i)}{\pa y}.
\end{equation*}

The system can be written in the matrix-vector form
\begin{equation}\label{eq:SWME}
    \partial_t U+A(U)\partial_x U+B(U)\partial_y U=S(U),
\end{equation}
where
\begin{equation*}
    U=(h,hu_m,hv_m,h\alpha_1,h\beta_1,\ldots,h\alpha_N,h\beta_N)^T,
    \qquad h>0.
\end{equation*}

The case $N=0$ recovers the classical shallow water equations.
Following \cite{Bauerle2025}, we decompose the coefficient matrices into conservative and non-conservative parts:
\begin{equation*}
    A(U)=\partial_UF(U)+P(U),
    \qquad
    B(U)=\partial_UG(U)+Q(U),
\end{equation*}
with the physical flus
\begin{equation*}
\begin{aligned}
F(U)=\Big(&hu_m,
\;h\left(u_m^2+\sum_{j=1}^N\frac{\alpha_j^2}{2j+1}\right)+\frac12 Gh^2,
\;h\left(u_mv_m+\sum_{j=1}^N\frac{\alpha_j\beta_j}{2j+1}\right), \\
&h\left(2u_m\alpha_1+\sum_{j,k=1}^NA_{1jk}\alpha_j\alpha_k\right),
\;h\left(u_m\beta_1+v_m\alpha_1+\sum_{j,k=1}^NA_{1jk}\alpha_j\beta_k\right),
\ldots, \\
&h\left(2u_m\alpha_N+\sum_{j,k=1}^NA_{Njk}\alpha_j\alpha_k\right),
\;h\left(u_m\beta_N+v_m\alpha_N+\sum_{j,k=1}^NA_{Njk}\alpha_j\beta_k\right)\Big)^T,
\end{aligned}
\end{equation*}
and
\begin{equation*}
\begin{aligned}
G(U)=\Big(&hv_m,
\;h\left(u_mv_m+\sum_{j=1}^N\frac{\alpha_j\beta_j}{2j+1}\right),
\;h\left(v_m^2+\sum_{j=1}^N\frac{\beta_j^2}{2j+1}\right)+\frac12 Gh^2, \\
&h\left(u_m\beta_1+v_m\alpha_1+\sum_{j,k=1}^NA_{1jk}\alpha_j\beta_k\right),
\;h\left(2v_m\beta_1+\sum_{j,k=1}^NA_{1jk}\beta_j\beta_k\right),
\ldots, \\
&h\left(u_m\beta_N+v_m\alpha_N+\sum_{j,k=1}^NA_{Njk}\alpha_j\beta_k\right),
\;h\left(2v_m\beta_N+\sum_{j,k=1}^NA_{Njk}\beta_j\beta_k\right)\Big)^T.
\end{aligned}
\end{equation*}

The non-conservative matrices are written as
\begin{equation*}
    P(U)=P_1(U)+P_2(U),
    \qquad
    Q(U)=Q_1(U)+Q_2(U).
\end{equation*}
The first part is
\begin{equation}\label{eq:matrix-non-conservative-part1-x}
    P_1(U)=\diag(0_{3\times3},p(U),\ldots,p(U)),
\end{equation}
\begin{equation*}
    Q_1(U)=\diag(0_{3\times3},q(U),\ldots,q(U)),
\end{equation*}
where
\begin{equation}\label{eq:matrix-small-p-q}
    p(U)=\begin{pmatrix}-u_m&0\\-v_m&0\end{pmatrix},
    \qquad
    q(U)=\begin{pmatrix}0&-u_m\\0&-v_m\end{pmatrix}.
\end{equation}
The second part contains the non-conservative products involving $D_i$:
\begin{equation*}
    P_2(U)=\diag(0_{3\times3},\mathcal{G}(U)),
    \qquad
    Q_2(U)=\diag(0_{3\times3},\mathcal{H}(U)),
\end{equation*}
where $\mathcal{G}(U)=(g_{ij})_{1\le i,j\le N}$ and $\mathcal{H}(U)=(h_{ij})_{1\le i,j\le N}$ are block matrices with $2\times2$ blocks
\begin{equation*}
    g_{ij}(U)=
    \begin{pmatrix}
        \sum_kB_{ijk}\alpha_k&0\\
        \sum_kB_{ijk}\beta_k&0
    \end{pmatrix},
    \qquad
    h_{ij}(U)=
    \begin{pmatrix}
        0&\sum_kB_{ijk}\alpha_k\\
        0&\sum_kB_{ijk}\beta_k
    \end{pmatrix}.
\end{equation*}
Finally, the source term in \eqref{eq:SWME} is
\begin{align*}
S(U)=\Big(&0,
-\fl{\gamma}{\varepsilon}u_b^{(N)}-Gh\fl{\pa h_b}{\pa x},
-\fl{\gamma}{\varepsilon}v_b^{(N)}-Gh\fl{\pa h_b}{\pa y}, \\
&-3\left(\fl{\gamma}{\varepsilon}u_b^{(N)}+\fl{\iRe}{\varepsilon h}\sum_{j=1}^{N}C_{1j}\alpha_j\right),
-3\left(\fl{\gamma}{\varepsilon}v_b^{(N)}+\fl{\iRe}{\varepsilon h}\sum_{j=1}^{N}C_{1j}\beta_j\right),
\ldots, \\
&-(2N+1)\left(\fl{\gamma}{\varepsilon}u_b^{(N)}+\fl{\iRe}{\varepsilon h}\sum_{j=1}^{N}C_{Nj}\alpha_j\right),
-(2N+1)\left(\fl{\gamma}{\varepsilon}v_b^{(N)}+\fl{\iRe}{\varepsilon h}\sum_{j=1}^{N}C_{Nj}\beta_j\right)\Big)^T.
\end{align*}

{
\section{Two-dimensional SWLME and hyperbolicity}
\label{sec:hyperbolicity}
}

{
\subsection{Shallow water linearized moment equations}
\label{subsec:SWLME}

For shallow water moment equations of order $N>1$, nonlinear moment couplings in the conservative fluxes and nonconservative terms complicate the characterization of steady states, even for hyperbolic regularizations.
Based on the insight from the SWME with $N=1$, Koellermeier and Pimentel-Garc\'ia \cite{Koellermeier2022} proposed the one-dimensional SWLME.
Assuming small deviations from a constant profile, i.e., $\alpha_k=O(\varepsilon)$, allows us to neglect moment couplings like $\alpha_i\alpha_j=O(\varepsilon^2)$ in the higher-moment equations.
This leads to a hyperbolic system that admits an implicit description of smooth frictionless moving equilibria over non-flat topography.
However, these steady states cannot be extended directly to the case with bottom friction; here, we extend the setting to the case of two horizontal dimensions with a Navier-slip bottom condition (i.e., frictional) and discuss how to find a (discrete) steady state later in Section~\ref{sec:well-balanced}.

For the two-dimensional extension, we apply the same approximation to both moment families $\alpha_k$ and $\beta_k$, retaining the mass and mean-momentum equations and simplifying the higher-moment equations.
Specifically, we omit the quadratic higher-mode transport terms involving $A_{ijk}$ and the corresponding non-conservative couplings involving $B_{ijk}$, but retain the complete quadratic moment tensor in the mean-momentum fluxes.
The retained terms $\alpha_j^2$, $\alpha_j\beta_j$, and $\beta_j^2$ arise from the exact vertical averages of $\widetilde{u}^2$, $\widetilde{u}\widetilde{v}$, and $\widetilde{v}^2$, respectively, and together form a rotationally covariant tensor.
Thus, ``linearized'' refers to selected higher-moment couplings; the quadratic mean-momentum fluxes remain.

The conservative fluxes of the linearized moment system are
\begin{equation*}
\begin{aligned}
F_L(U)=\Big(&hu_m,
\;h\left(u_m^2+\sum_{j=1}^N\frac{\alpha_j^2}{2j+1}\right)+\frac12 Gh^2,
\;h\left(u_mv_m+\sum_{j=1}^N\frac{\alpha_j\beta_j}{2j+1}\right), \\
&h(2u_m\alpha_1),
\;h(u_m\beta_1+v_m\alpha_1),
\ldots,
\;h(2u_m\alpha_N),
\;h(u_m\beta_N+v_m\alpha_N)\Big)^T,
\end{aligned}
\end{equation*}
and
\begin{equation*}
\begin{aligned}
G_L(U)=\Big(&hv_m,
\;h\left(u_mv_m+\sum_{j=1}^N\frac{\alpha_j\beta_j}{2j+1}\right),
\;h\left(v_m^2+\sum_{j=1}^N\frac{\beta_j^2}{2j+1}\right)+\frac12 Gh^2, \\
&h(u_m\beta_1+v_m\alpha_1),
\;h(2v_m\beta_1),
\ldots,
\;h(u_m\beta_N+v_m\alpha_N),
\;h(2v_m\beta_N)\Big)^T.
\end{aligned}
\end{equation*}
The remaining non-conservative terms are $P_{1}(U)$ and $Q_{1}(U)$ from \eqref{eq:matrix-non-conservative-part1-x}--\eqref{eq:matrix-small-p-q}.
Hence, the coefficient matrices are
\begin{equation}\label{eq:AL-BL-def}
    A_L(U) = \partial_{U} F_{L}(U) + P_{1}(U), \qquad B_{L}(U) = \partial_{U} G_{L}(U) + Q_{1}(U).
\end{equation}
We refer to the resulting system as the shallow water linearized moment equations (SWLME):
\begin{equation}\label{eq:SWLME}
    \partial_tU+A_L(U)\partial_xU+B_L(U)\partial_yU=S(U).
\end{equation}
The source $S(U)$ is unchanged by the linearization.
For $\gamma=0$, zero initial and boundary moments remain zero, reducing the system to the classical two-dimensional shallow water equations.
}

\subsection{Rotational invariance}
We consider the state space $\mathcal{A}_N=\{U\in\mathbb{R}^{2N+3}:h>0\}$.
Following \cite{Bauerle2025}, we use rotational invariance to reduce the hyperbolicity analysis to one coordinate direction.
For an angle $\theta\in[0,2\pi)$, define the rotation matrix
\begin{equation*}
    T_2(\theta)=
    \begin{pmatrix}
        \cos\theta & \sin\theta \\
        -\sin\theta & \cos\theta
    \end{pmatrix}
    \in\mathbb{R}^{2\times2}.
\end{equation*}
For the state vector
\begin{equation*}
    U=(h,hu_m,hv_m,h\alpha_1,h\beta_1,\ldots,h\alpha_N,h\beta_N)^T
    \in\mathbb{R}^{2N+3},
\end{equation*}
we introduce the block-diagonal rotation matrix
\begin{equation*}
    T(\theta)=\diag\big(1,T_2(\theta),T_2(\theta),\ldots,T_2(\theta)\big)
    \in\mathbb{R}^{(2N+3)\times(2N+3)}.
\end{equation*}
Thus, $T(\theta)$ leaves the scalar variable $h$ unchanged and rotates each horizontal vector pair $(hu_m,hv_m)^T$ and $(h\alpha_j,h\beta_j)^T$ in the same way.

\begin{definition}[rotational invariance of the coefficient matrices]
Consider a two-dimensional first-order system
\begin{equation}\label{eq:general-rot-system}
    \partial_t U+A(U)\partial_xU+B(U)\partial_yU=S(U,x,y).
\end{equation}
We say that the coefficient matrices $A(U)$ and $B(U)$ satisfy the rotational-invariance property if, for every $\theta\in[0,2\pi)$ and every admissible state $U$,
\begin{equation}\label{eq:rot-invariance-matrix}
    \cos\theta\,A(U)+\sin\theta\,B(U)
    =T(\theta)^{-1}A(T(\theta)U)T(\theta).
\end{equation}
\end{definition}

\begin{remark}\label{rem:source-rotational-invariance}
The full system \eqref{eq:general-rot-system} is rotationally invariant only if the source term also transforms consistently, namely
\begin{equation*}
    T(\theta)S(U,X)=S(T(\theta)U,T_2(\theta)X),
    \qquad X=(x,y)^T.
\end{equation*}
Isotropic sources, including bottom friction, satisfy this condition because their velocity and moment pairs rotate as horizontal vectors.
The topography source transforms consistently when the bottom is rotated as a scalar field.
Hyperbolicity depends only on the principal part, so we use the matrix identity \eqref{eq:rot-invariance-matrix}.
\end{remark}

For any unit direction $n=(\cos\theta,\sin\theta)^T$, the directional coefficient matrix is
\begin{equation*}
    A_n(U)=\cos\theta\,A(U)+\sin\theta\,B(U).
\end{equation*}
By \eqref{eq:rot-invariance-matrix}, $A_n(U)$ is similar to the $x$-direction matrix $A(T(\theta)U)$.
Hence, $A_n(U)$ and $A(T(\theta)U)$ have the same eigenvalues and the same diagonalizability properties.
Consequently, once rotational invariance is known, the two-dimensional hyperbolicity analysis reduces to the analysis of the $x$-direction coefficient matrix.

\begin{definition}[global hyperbolicity]\label{def:global-hyperbolicity}
The principal part is globally hyperbolic on $\mathcal{A}_N$ if $A_n(U)$ is diagonalizable over $\mathbb{R}$ for every $U\in\mathcal{A}_N$ and $n\in\mathbb{S}^1$.
\end{definition}

This pointwise definition does not require a uniformly bounded diagonalizer or uniformly positive symmetrizer as $h\to0^+$ or when vanishing directional moments change the eigenvalue multiplicities.

\begin{theorem}[rotational invariance of the SWLME]
\label{thm:rot-SWLME}
The coefficient matrices $A_L(U)$ and $B_L(U)$ of the SWLME \eqref{eq:SWLME} satisfy
\begin{equation}\label{eq:rot-SWLME}
    \cos\theta\,A_L(U)+\sin\theta\,B_L(U)
    =T(\theta)^{-1}A_L(T(\theta)U)T(\theta),
\end{equation}
for all $\theta\in[0,2\pi)$ and all admissible states $U$.
\end{theorem}

\begin{proof}
The proof follows the conservative/non-conservative splitting argument used for the shallow water moment equations in \cite{Bauerle2025}.
The conservative fluxes $F_L(U)$ and $G_L(U)$ satisfy the rotational-invariance identity, and the non-conservative matrices $P_1(U)$ and $Q_1(U)$ satisfy the corresponding block rotational-invariance identity.
Combining these two identities with
\begin{equation*}
    A_L(U)=\partial_UF_L(U)+P_1(U),
    \qquad
    B_L(U)=\partial_UG_L(U)+Q_1(U),
\end{equation*}
gives \eqref{eq:rot-SWLME}.
The detailed proof is provided in \hyperref[theorem:proof]{Appendix A}.
\end{proof}

The globally hyperbolic model in Section~\ref{sec:sub:newmodel} is constructed by modifying both directional coefficient matrices while preserving rotational invariance.
{We recall the block characterization of Bauerle et al.~\cite{Bauerle2025}, which specifies the required compatibility.}

\begin{definition}[rotationally consistent $2\times2$ blocks]
Let $V=(p,q)^T\in\mathbb{R}^2$, and let $A_b(V),B_b(V)\in\mathbb{R}^{2\times2}$.
We say that the pair $(A_b,B_b)$ is rotationally consistent if
\begin{equation}\label{eq:rot-block-definition}
    T_2(\theta)^{-1}A_b(T_2(\theta)V)T_2(\theta)
    =\cos\theta\,A_b(V)+\sin\theta\,B_b(V)
\end{equation}
holds for all $\theta\in[0,2\pi)$ and all $V\in\mathbb{R}^2$.
\end{definition}

\begin{theorem}[{Bauerle et al.~\cite{Bauerle2025}}]\label{thm:rot-2}
Assume that $A_b(V)$ and $B_b(V)$ are homogeneous linear functions of $V=(p,q)^T$ and satisfy the block rotational-invariance relation \eqref{eq:rot-block-definition}.
Then there exist constants $c_i\in\mathbb{R}$, $i=1,\ldots,6$, such that
\begin{equation}\label{eq:block-A-form}
\begin{aligned}
A_b(V)={}&
    c_1\begin{pmatrix}p&0\\q&0\end{pmatrix}
    +c_2\begin{pmatrix}q&0\\-p&0\end{pmatrix}
    +c_3\begin{pmatrix}p&0\\0&p\end{pmatrix}
    +c_4\begin{pmatrix}q&0\\0&q\end{pmatrix} \\
&\quad
    +c_5\begin{pmatrix}0&p\\0&q\end{pmatrix}
    +c_6\begin{pmatrix}0&q\\0&-p\end{pmatrix},
\end{aligned}
\end{equation}
and
\begin{equation}\label{eq:block-B-form}
\begin{aligned}
B_b(V)={}&
    c_1\begin{pmatrix}0&p\\0&q\end{pmatrix}
    +c_2\begin{pmatrix}0&q\\0&-p\end{pmatrix}
    +c_3\begin{pmatrix}q&0\\0&q\end{pmatrix}
    -c_4\begin{pmatrix}p&0\\0&p\end{pmatrix} \\
&\quad
    -c_5\begin{pmatrix}p&0\\q&0\end{pmatrix}
    +c_6\begin{pmatrix}-q&0\\p&0\end{pmatrix}.
\end{aligned}
\end{equation}
Conversely, any pair of matrices of the form \eqref{eq:block-A-form}--\eqref{eq:block-B-form} satisfies \eqref{eq:rot-block-definition}.
\end{theorem}

\begin{theorem}[block closure preserving rotational invariance]
\label{thm:block-closure-rot}
Let $A_L(U)$ and $B_L(U)$ be the coefficient matrices of the SWLME.
Suppose that additional matrices $\Delta A(U)$ and $\Delta B(U)$ have zero entries in the first row and first column associated with the scalar variable $h$, and that each $2\times2$ block acting on the rotated vector pairs $(hu_m,hv_m)^T,(h\alpha_1,h\beta_1)^T,\ldots,(h\alpha_N,h\beta_N)^T$ is rotationally consistent in the sense of \eqref{eq:rot-block-definition}.
Then the modified matrices
\begin{equation*}
    A_G(U)=A_L(U)+\Delta A(U),
    \qquad
    B_G(U)=B_L(U)+\Delta B(U)
\end{equation*}
also satisfy
\begin{equation}\label{eq:rot-AG-BG}
    \cos\theta\,A_G(U)+\sin\theta\,B_G(U)
    =T(\theta)^{-1}A_G(T(\theta)U)T(\theta).
\end{equation}
\end{theorem}

\begin{proof}
The result follows from linearity.
The original matrices $A_L(U)$ and $B_L(U)$ satisfy \eqref{eq:rot-SWLME}.
Each added $2\times2$ block satisfies the corresponding block relation \eqref{eq:rot-block-definition}.
Since $T(\theta)$ is block diagonal with the same $T_2(\theta)$ acting on each vector pair, assembling these block identities gives the rotational-invariance identity for $\Delta A(U)$ and $\Delta B(U)$.
Adding this identity to \eqref{eq:rot-SWLME} gives \eqref{eq:rot-AG-BG}.
\end{proof}

By Theorem~\ref{thm:rot-SWLME}, it suffices to analyze the $x$-direction matrix at every admissible state.
We show that the direct extension is not globally hyperbolic and construct a rotationally invariant modification that restores this property.

\subsection{Hyperbolicity of the unmodified SWLME}
\label{sec:sub:hyper-SWLME}

Reorder the state variables as
\begin{equation}\label{def:tildeU}
    \tU = (h, hu_m, h\ap_1, h\ap_2,\cdots, h\ap_N, hv_m, h\beta_1, h\beta_2, \cdots, h\beta_N)^{T}.
\end{equation}

This reordering preserves rotational invariance and hyperbolicity.
In the ordering \eqref{def:tildeU}, the $x$-direction coefficient matrix $A_L$ in \eqref{eq:SWLME} has the block lower triangular form
\begin{equation}\label{eq:AL-block}
    \tA_L(\tU) = \begin{pmatrix}
        \tA_{11}(\tU) &0 \\
        \tA_{21}(\tU) & \tA_{22}(\tU)
    \end{pmatrix},
\end{equation}
where the block matrices $\tA_{11}\in\mathR^{(N+2)\times(N+2)}$, $\tA_{21}\in\mathR^{(N+1)\times(N+2)}$, and $\tA_{22}\in\mathR^{(N+1)\times(N+1)}$ are given by

\begin{equation*}
\tA_{11} =
\begin{pmatrix}
    0 & 1 & 0 & 0 & \cdots & 0 \\
    Gh-u_m^2-\sum_{j=1}^N\frac{\alpha_j^2}{2j+1} & 2u_m & \frac{2\alpha_1}{3} & \frac{2\alpha_2}{5} & \cdots & \frac{2\alpha_N}{2N+1} \\
    -2u_m\alpha_1 & 2\alpha_1 &u_m & 0 & \cdots & 0 \\
    -2u_m\alpha_2 & 2\alpha_2 & 0 & u_m & 0 & \cdots \\
    \vdots & \vdots & \vdots & \vdots & \ddots & \vdots \\
    -2u_m\alpha_N & 2\alpha_N & 0 & 0 & 0 & u_m
\end{pmatrix},
\end{equation*}

\begin{equation*}
\tA_{21} =
\begin{pmatrix}
    -u_m v_m-\sum_{j=1}^{N}\fl{\ap_j\beta_j}{2j+1} & v_m & \frac{\beta_1}{3} & \frac{\beta_2}{5} & \cdots & \frac{\beta_N}{2N+1} \\
    -(u_m\beta_1 + v_m\alpha_1) & \beta_1 & 0 & 0 & \cdots & 0 \\
    -(u_m\beta_2 + v_m\alpha_2) & \beta_2 & 0 & 0 & \cdots & 0 \\
    \vdots & \vdots & \vdots & \vdots & \ddots & \vdots \\
    -(u_m\beta_N + v_m\alpha_N) & \beta_N & 0 & 0 & 0 &0
\end{pmatrix},
\end{equation*}
and
\begin{equation*}
\tA_{22} =
\begin{pmatrix}
    u_m & \frac{\alpha_1}{3} & \frac{\alpha_2}{5} & \cdots & \frac{\alpha_N}{2N+1} \\
    \alpha_1 & u_m & 0 & \cdots & 0 \\
    \alpha_2 & 0 & u_m & \cdots & 0 \\
    \vdots & \vdots & \vdots & \ddots & \vdots \\
    \alpha_N & 0 & 0 & \cdots & u_m
\end{pmatrix}.
\end{equation*}
The block $\tA_{11}$ is the coefficient matrix of the one-dimensional SWLME, while $\tA_{21}$ contains the coupling from the $x$-moment variables to the $y$-moment variables.

Since \eqref{eq:AL-block} is block lower triangular,
\begin{equation*}
    \det(\tA_L - \lambda I) = \det(\tA_{11} - \lambda I) \det(\tA_{22} - \lambda I).
\end{equation*}
Thus, we only need to compute the characteristic polynomial of $\tA_{11}$ and $\tA_{22}$.
\begin{lem}
    Given a matrix $A\in\mathbb{R}^{(n+1)\times(n+1)}$ with only the first row, the first column, and the diagonal entries being non-zero,
    \begin{equation*}
        A =
    \begin{pmatrix}
        d_0 & a_1 & a_2 & \cdots & a_n \\
        b_1 & d_1 & 0 & \cdots & 0 \\
        b_2 & 0 & d_2 & \cdots & 0 \\
        \vdots & \vdots & \vdots & \ddots & \vdots \\
        b_n & 0 & 0 & \cdots & d_n
    \end{pmatrix},
    \end{equation*}
    Then the determinant of $A$ is given by
    \begin{equation*}
        \det(A) = d_0\prod_{i=1}^{n}d_i-\sum_{j=1}^{n}a_jb_j\prod_{\substack{i=1\\i\neq j}}^{n}d_i.
    \end{equation*}
\end{lem}
\begin{proof}
    The determinant can be computed using the Laplace expansion along the first row.
\end{proof}

We first compute the characteristic polynomial of $\tA_{11}$:
\begin{equation*}
\begin{aligned}
    \det(\tA_{11} - \lambda I) &=
    \det
    \begin{pmatrix}
        -\lambda & 1 & 0 & 0 & \cdots & 0 \\
        Gh-u_m^2-\sum_{j=1}^N\frac{\alpha_j^2}{2j+1} & 2u_m - \lambda & \frac{2\alpha_1}{3} & \frac{2\alpha_2}{5} & \cdots & \frac{2\alpha_N}{2N+1} \\
        -2u_m\alpha_1 & 2\alpha_1 & u_m - \lambda & 0 & \cdots & 0 \\
        -2u_m\alpha_2 & 2\alpha_2 & 0 & u_m - \lambda & 0 & \cdots \\
        \vdots & \vdots & \vdots & \vdots & \ddots & \vdots \\
        -2u_m\alpha_N & 2\alpha_N & 0 & 0 & 0 & u_m - \lambda
    \end{pmatrix} \\
    &= (u_m-\lambda)^N \brac{(\lambda-u_m)^2 - Gh - \sum_{j=1}^{N}\frac{3\alpha_j^2}{(2j+1)}},
\end{aligned}
\end{equation*}
Therefore, the eigenvalue of $\tA_{11}$ is given by
\begin{equation}\label{eig:a11}
\lambda_{1,2} = u_m \pm \sqrt{Gh + \sum_{j=1}^{N}\frac{3\alpha_j^2}{(2j+1)}}, \quad \lambda_{i} = u_m, \, i=3, \cdots, N+2.
\end{equation}

The characteristic polynomial of $\tA_{22}$ is
\begin{equation*}
\begin{aligned}
\det(\tA_{22} - \lambda I) ={}& \det
\begin{pmatrix}
    u_m - \lambda & \frac{\alpha_1}{3} & \frac{\alpha_2}{5} & \cdots & \frac{\alpha_N}{2N+1} \\
    \alpha_1 & u_m - \lambda & 0 & \cdots & 0 \\
    \alpha_2 & 0 & u_m - \lambda & \cdots & 0 \\
    \vdots & \vdots & \vdots & \ddots & \vdots \\
    \alpha_N & 0 & 0 & \cdots & u_m - \lambda
\end{pmatrix} \\
={}& (u_m - \lambda)^{N-1} \brac{(u_m - \lambda)^2 - \sum_{j=1}^{N} \frac{\alpha_j^2}{2j+1}},
\end{aligned}
\end{equation*}
Therefore, the eigenvalues of $\tA_{22}$ are
\begin{equation}\label{eig:a22}
\lambda_{i} = u_m, \, i = 1, \cdots, N-1, \quad \lambda_{N,N+1} = u_m \pm \sqrt{\sum_{j=1}^{N} \frac{\alpha_j^2}{2j+1}},
\end{equation}

Although all eigenvalues are real, the direct SWLME is not globally hyperbolic because the repeated eigenvalue $u_m$ may not have a complete eigenspace.
To identify a defective state, consider
\begin{equation*}
    \ap_1=\cdots=\ap_N=0.
\end{equation*}
At such a state, $u_m$ has algebraic multiplicity $N$ in $\tA_{11}$ and $N+1$ in $\tA_{22}$, hence algebraic multiplicity $2N+1$ in $\tA_L$.
The corresponding matrix is
\begin{equation*}
\tA_L =
\left(
\begin{array}{ccccccccccc}
0 & 1 & 0 & 0 & \cdots & 0 & 0 & 0 & \cdots & 0 & 0 \\
Gh-u_m^2 & 2u_m & 0 & 0 & \cdots & 0 & 0 & 0 & \cdots & 0 & 0 \\
0 & 0 & u_m & 0 & \cdots & 0 & 0 & 0 & \cdots & 0 & 0 \\
0 & 0 & 0 & u_m & 0 & \cdots & 0 & 0 & \cdots & 0 & 0 \\
\vdots & \vdots & \vdots & \vdots & \ddots & \vdots & \vdots & \vdots & \ddots & \vdots & 0 \\
0 & 0 & 0 & 0 & 0 & u_m & 0 & 0 & \cdots & 0 & 0 \\
-u_m v_m & v_m & \frac{\beta_1}{3} & \frac{\beta_2}{5} & \cdots & \frac{\beta_N}{2N+1} & u_m & 0 & 0 & \cdots & 0 \\
-u_m\beta_1 & \beta_1 & 0 & 0 & \cdots & 0 & 0 & u_m & 0 & \cdots & 0 \\
-u_m\beta_2 & \beta_2 & 0 & 0 & \cdots & 0 & 0 & 0 & u_m & \cdots & 0 \\
\vdots & \vdots & \vdots & \vdots & \ddots & \vdots & \vdots & \vdots & \ddots & \vdots & 0 \\
-u_m\beta_N & \beta_N & 0 & 0 & \cdots & 0 & 0 & 0 & 0 & \cdots & u_m
\end{array}
\right).
\end{equation*}
We determine the eigenspace associated with $u_m$ by solving
\begin{equation*}
    (\tA_L-u_m I)X=0,
\end{equation*}
where
\begin{equation*}
    X=(x_0,p_0,p_1,\ldots,p_N,q_0,q_1,\ldots,q_N)^T.
\end{equation*}
The first two equations are
\begin{equation*}
\begin{aligned}
    -u_mx_0+p_0&=0, \\
    (Gh-u_m^2)x_0+u_mp_0&=0.
\end{aligned}
\end{equation*}
Since $G>0$ and $h>0$, they imply $x_0=p_0=0$.
The next $N$ equations and the final $N$ transverse-moment equations are then identities.
The transverse mean-momentum equation reduces to the single constraint
\begin{equation*}
    \sum_{j=1}^{N}\frac{\beta_j}{2j+1}p_j=0.
\end{equation*}
If $(\beta_1,\ldots,\beta_N)\neq0$, this constraint leaves $N-1$ independent components among $p_1,\ldots,p_N$, while $q_0,q_1,\ldots,q_N$ are all free.
Consequently, the eigenspace has dimension $(N-1)+(N+1)=2N<2N+1$.
If $\beta_1=\cdots=\beta_N=0$, the constraint vanishes and this particular state is not defective.

\begin{theorem}[failure of global hyperbolicity for the direct SWLME]
\label{thm:direct-SWLME-not-global}
Let $N\geq1$ and $G>0$.
At every admissible state satisfying
\begin{equation*}
    \alpha_1=\cdots=\alpha_N=0,
    \qquad
    (\beta_1,\ldots,\beta_N)\neq0,
\end{equation*}
the eigenvalue $u_m$ of the $x$-direction coefficient matrix has algebraic multiplicity $2N+1$ and geometric multiplicity $2N$.
Therefore, the direct two-dimensional SWLME is not globally hyperbolic.
\end{theorem}

\begin{proof}
Equations~\eqref{eig:a11}--\eqref{eig:a22} give algebraic multiplicity $2N+1$ for $u_m$ when $\alpha_1=\cdots=\alpha_N=0$.
The eigenspace calculation above gives geometric multiplicity $2N$ whenever $(\beta_1,\ldots,\beta_N)\neq0$.
Thus, $\tA_L$ is defective at these admissible states.
\end{proof}
The globally hyperbolic construction below removes the offending lower-left coupling and modifies the transverse block in a rotationally consistent way.

\subsection{A globally hyperbolic modification of the SWLME}
\label{sec:sub:newmodel}

The defective states identified in Theorem~\ref{thm:direct-SWLME-not-global} arise from the coupling between the longitudinal moment variables and the transverse mean-momentum equation.
We therefore construct the hyperbolic modification directly from the reordered matrix $\tA_L$ in \eqref{eq:AL-block}.
Let
\begin{equation*}
    \boldsymbol{\alpha}=(\alpha_1,\ldots,\alpha_N)^T,
    \qquad
    \boldsymbol{\beta}=(\beta_1,\ldots,\beta_N)^T,
    \qquad
    D=\diag(3,5,\ldots,2N+1),
\end{equation*}
and define
\begin{equation*}
    \mathcal{S}_{\alpha}
    =\boldsymbol{\alpha}^TD^{-1}\boldsymbol{\alpha},
    \qquad
    \mathcal{S}_{\alpha\beta}
    =\boldsymbol{\alpha}^TD^{-1}\boldsymbol{\beta}.
\end{equation*}
In the ordering $\tU=(h,hu_m,h\boldsymbol{\alpha}^T,hv_m,h\boldsymbol{\beta}^T)^T$, introduce the correction
\begin{equation}\label{eq:delta-tilde-A}
\Delta\tA=
\begin{pmatrix}
    0&0&0_{1\times N}&0&0_{1\times N}\\
    0&0&0_{1\times N}&0&0_{1\times N}\\
    0_{N\times1}&0_{N\times1}&0_{N\times N}&0_{N\times1}&0_{N\times N}\\
    0&0&-\boldsymbol{\beta}^TD^{-1}&0&\boldsymbol{\alpha}^TD^{-1}\\
    0_{N\times1}&0_{N\times1}&0_{N\times N}&0_{N\times1}&0_{N\times N}
\end{pmatrix}.
\end{equation}
Thus, only the transverse mean-momentum row is changed.
For each $j=1,\ldots,N$, the entries in the $(h\alpha_j,h\beta_j)$ columns are replaced according to
\begin{equation*}
    \left(\fl{\beta_j}{2j+1},\fl{\alpha_j}{2j+1}\right)
    \longmapsto
    \left(0,\fl{2\alpha_j}{2j+1}\right).
\end{equation*}
The first change removes the defective coupling at $\boldsymbol{\alpha}=0$; the second enforces rotational consistency.
We define
\begin{equation}\label{eq:tilde-AG-definition}
    \tA_G=\tA_L+\Delta\tA
    =\begin{pmatrix}
        \tA_{11}&0\\
        \tA_{21}^{G}&\tA_{22}^{G}
    \end{pmatrix},
\end{equation}
where $\tA_{11}$ is unchanged and
\begin{equation}\label{eq:tilde-AG-blocks}
    \tA_{21}^{G}
    =\begin{pmatrix}
        -u_mv_m-\mathcal{S}_{\alpha\beta}&v_m&0_{1\times N}\\
        -(u_m\boldsymbol{\beta}+v_m\boldsymbol{\alpha})&\boldsymbol{\beta}&0_{N\times N}
    \end{pmatrix},
    \qquad
    \tA_{22}^{G}
    =\begin{pmatrix}
        u_m&2\boldsymbol{\alpha}^TD^{-1}\\
        \boldsymbol{\alpha}&u_mI_N
    \end{pmatrix}.
\end{equation}

Because $\tA_G$ is block lower triangular, its eigenvalues are the union of those of $\tA_{11}$ and $\tA_{22}^{G}$.
The longitudinal block $\tA_{11}$ has the eigenvalues in \eqref{eig:a11}.
The characteristic polynomial of the modified transverse block is
\begin{equation}\label{eq:char-tilde-A22G}
    \det(\tA_{22}^{G}-\lambda I)
    =(u_m-\lambda)^{N-1}\left((u_m-\lambda)^2-2\mathcal{S}_{\alpha}\right).
\end{equation}
Hence, if $\mathcal{S}_{\alpha}>0$, the transverse eigenvalues are $u_m$ with multiplicity $N-1$ and $u_m\pm\sqrt{2\mathcal{S}_{\alpha}}$.
If $\mathcal{S}_{\alpha}=0$, then $\boldsymbol{\alpha}=0$ and all $N+1$ transverse eigenvalues equal $u_m$.

It remains to verify that the repeated eigenvalue $u_m$ has a complete eigenspace.
Write an eigenvector in the reordered variables as
\begin{equation*}
    X=(x_0,p_0,\boldsymbol{p}^T,q_0,\boldsymbol{q}^T)^T,
    \qquad
    \boldsymbol{p},\boldsymbol{q}\in\mathbb{R}^N.
\end{equation*}
First suppose that $\mathcal{S}_{\alpha}=0$.
Then $\boldsymbol{\alpha}=0$, and the first two equations of $(\tA_G-u_mI)X=0$ imply $x_0=p_0=0$ because $Gh>0$.
All components of $\boldsymbol{p}$, $q_0$, and $\boldsymbol{q}$ are free, so the eigenspace has dimension $2N+1$, equal to the algebraic multiplicity of $u_m$.

Now suppose that $\mathcal{S}_{\alpha}>0$.
The equation $(\tA_G-u_mI)X=0$ reduces to the four independent scalar relations
\begin{equation}\label{eq:tilde-AG-um-eigenspace}
\begin{aligned}
    p_0-u_mx_0&=0,\\
    q_0-v_mx_0&=0,\\
    (Gh-\mathcal{S}_{\alpha})x_0
    +2\boldsymbol{\alpha}^TD^{-1}\boldsymbol{p}&=0,\\
    -\mathcal{S}_{\alpha\beta}x_0
    +2\boldsymbol{\alpha}^TD^{-1}\boldsymbol{q}&=0.
\end{aligned}
\end{equation}
These relations are independent because $\boldsymbol{\alpha}\neq0$.
They leave an eigenspace of dimension $(2N+3)-4=2N-1$, equal to the algebraic multiplicity of $u_m$ obtained from \eqref{eig:a11} and \eqref{eq:char-tilde-A22G}.
Moreover, the four remaining eigenvalues
\begin{equation*}
    u_m\pm\sqrt{Gh+3\mathcal{S}_{\alpha}},
    \qquad
    u_m\pm\sqrt{2\mathcal{S}_{\alpha}},
\end{equation*}
are real and mutually distinct because $Gh+\mathcal{S}_{\alpha}>0$.
Consequently, $\tA_G$ is real diagonalizable for every $h>0$ and $G>0$.

We return to the original vector-pair ordering to construct the rotational companion of the correction.
Let $\Pi$ be the permutation matrix such that $\tU=\Pi U$.
The $x$-direction matrix in the original ordering is
\begin{equation}\label{def:AG}
    A_G(U)
    =\Pi^{-1}\tA_G(\Pi U)\Pi
    =A_L(U)+\Delta A(U).
\end{equation}
Index the horizontal vector pairs by $0,1,\ldots,N$, where pair $0$ is $(hu_m,hv_m)^T$ and pair $j$ is $(h\alpha_j,h\beta_j)^T$.
The only nonzero $2\times2$ vector-pair blocks of $\Delta A$ are
\begin{equation}\label{eq:delta-A-blocks}
    (\Delta A)_{0j}
    =\fl{1}{2j+1}
    \begin{pmatrix}
        0&0\\
        -\beta_j&\alpha_j
    \end{pmatrix},
    \qquad j=1,\ldots,N.
\end{equation}
All blocks involving the scalar variable $h$ and all other vector-pair blocks vanish.

For $V_j=(\alpha_j,\beta_j)^T$, the block in \eqref{eq:delta-A-blocks} has the form in Theorem~\ref{thm:rot-2} with
\begin{equation*}
    c_1=-\fl{1}{2j+1},
    \qquad
    c_3=\fl{1}{2j+1},
    \qquad
    c_2=c_4=c_5=c_6=0.
\end{equation*}
The corresponding rotational companion is therefore
\begin{equation}\label{eq:delta-B-blocks}
    (\Delta B)_{0j}
    =\fl{1}{2j+1}
    \begin{pmatrix}
        \beta_j&-\alpha_j\\
        0&0
    \end{pmatrix},
    \qquad j=1,\ldots,N,
\end{equation}
with all other blocks equal to zero.
We define
\begin{equation}\label{def:BG}
    B_G(U)=B_L(U)+\Delta B(U).
\end{equation}
The resulting blocks in the mean-momentum rows and moment columns are
\begin{equation*}
    (A_G)_{0j}=\fl{2\alpha_j}{2j+1}I_2,
    \qquad
    (B_G)_{0j}=\fl{2\beta_j}{2j+1}I_2.
\end{equation*}
The block pairs \eqref{eq:delta-A-blocks}--\eqref{eq:delta-B-blocks} satisfy the rotational-consistency relation \eqref{eq:rot-block-definition}.
Therefore, Theorem~\ref{thm:block-closure-rot} gives
\begin{equation}\label{eq:rot-AG-BG-revised}
    \cos\theta\,A_G(U)+\sin\theta\,B_G(U)
    =T(\theta)^{-1}A_G(T(\theta)U)T(\theta).
\end{equation}

\begin{theorem}[global hyperbolicity of the $G\text{-}\mathrm{SWLME}$]
\label{thm:global-hyperbolicity-AG}
Let $N\geq1$ and $G>0$.
For each $U\in\mathcal{A}_N$, define $A_G(U)$ and $B_G(U)$ by \eqref{def:AG} and \eqref{def:BG}, with the correction blocks \eqref{eq:delta-A-blocks}--\eqref{eq:delta-B-blocks}.
Then, for every unit direction $n=(\cos\theta,\sin\theta)^T$, the directional coefficient matrix
\begin{equation*}
    A_n(U)=\cos\theta\,A_G(U)+\sin\theta\,B_G(U)
\end{equation*}
is real diagonalizable.
Consequently, the two-dimensional $G\text{-}\mathrm{SWLME}$ is globally hyperbolic.
\end{theorem}

\begin{proof}
The permutation relation \eqref{def:AG} shows that $A_G(U)$ is similar to $\tA_G(\Pi U)$.
If $\mathcal{S}_{\alpha}=0$, the eigenvalue $u_m$ has a $2N+1$-dimensional eigenspace, and the two remaining eigenvalues $u_m\pm\sqrt{Gh}$ are real and simple.
If $\mathcal{S}_{\alpha}>0$, the eigenvalue $u_m$ has a $2N-1$-dimensional eigenspace, and the four remaining eigenvalues are real, simple, and mutually distinct.
Thus, $A_G(U)$ is real diagonalizable for every admissible state.
Finally, \eqref{eq:rot-AG-BG-revised} shows that $A_n(U)$ is similar to $A_G(T(\theta)U)$, which is real diagonalizable by the preceding argument.
\end{proof}
Thus, Theorem~\ref{thm:global-hyperbolicity-AG} establishes global hyperbolicity precisely in the pointwise sense of Definition~\ref{def:global-hyperbolicity}, including at states with $\boldsymbol{\alpha}=0$, where the repeated eigenvalue $u_m$ has a complete eigenspace.

Table~\ref{tab:model-summary-short} summarizes the model hierarchy and the stationary classes considered in well-balanced schemes.

\begin{table}[H]
\centering
{
\setlength{\tabcolsep}{3pt}
\renewcommand{\arraystretch}{1.25}
\begin{tabular}{@{}C{0.10\textwidth} C{0.20\textwidth} C{0.14\textwidth} C{0.25\textwidth} C{0.25\textwidth}@{}}
\toprule
Dimension & Model & Rotational invariance & Hyperbolicity of the principal part & Stationary class used in well-balanced work \\
\midrule
\multirow{3}{*}{1D}
& SWME \cite{Kowalski2019} & N.A. & globally hyperbolic only for $N=0,1$ & not used here \\
\cmidrule(lr){2-5}
& HSWME \cite{Koellermeier2020} & N.A. & globally hyperbolic & lake at rest \\
\cmidrule(lr){2-5}
& SWLME \cite{Koellermeier2022,Pimentel-Garcia2024} & N.A. & globally hyperbolic & moving states: invariants without friction, steady ODE with friction \\
\midrule
\multirow{3}{*}{2D}
& HSWME and globally hyperbolic variants \cite{Bauerle2025} & yes & direct extension is not globally hyperbolic; modified variants are globally hyperbolic & lake at rest \\
\cmidrule(lr){2-5}
& SWLME (direct, this work) & yes & not globally hyperbolic by Theorem~\ref{thm:direct-SWLME-not-global} & not developed \\
\cmidrule(lr){2-5}
& $G\text{-}\mathrm{SWLME}$ (this work) & yes & globally hyperbolic by Theorem~\ref{thm:global-hyperbolicity-AG} & quasi-two-dimensional frictional moving states from a steady ODE \\
\bottomrule
\end{tabular}
}
\caption{Model hierarchy relevant to the present analysis.
Here N.A. means that planar rotational invariance is not applicable to a one-dimensional system.}
\label{tab:model-summary-short}
\end{table}

\section{Well-balanced schemes for quasi-two-dimensional flows}
\label{sec:well-balanced}

We construct first- and second-order finite-volume methods for the quasi-two-dimensional $G\text{-}\mathrm{SWLME}$ that exactly preserve one prescribed discrete moving equilibrium.
We construct and store the equilibrium from boundary data, then reconstruct deviations from it.
The construction combines the path-conservative framework of \cite{Pares2006b,Castro2017}, reconstruction of deviations from a known equilibrium \cite{Klingenberg2019}, midpoint collocation \cite{Gomez-Bueno2021a}, and the frictional SWLME stationary formulation of \cite{Pimentel-Garcia2024}.

The full two-dimensional system can be written as
\begin{equation}\label{eq:general-2d}
    U_t + A_{G}(U) \px U + B_{G}(U)\py U + S_{x}(U)\px h_b + S_{y}(U)\py h_b + R(U) = 0,
\end{equation}
where
\begin{align*}
    S_{x}(U) &= \left(0, Gh, 0, \cdots, 0\right)^{T}, \quad
    S_{y}(U) = \left(0,0, Gh, 0, \cdots, 0\right)^{T}, \\
    R(U) &= \left(0, R^{u}(U), R^{v}(U), R^{\ap}_{1}(U), R^{\beta}_{1}(U),\cdots, R^{\ap}_{N}(U), R^{\beta}_{N}(U)\right)^{T},
\end{align*}
Using the bottom velocities defined in~\eqref{eq:reconstructed-bottom-velocities}, the mean-momentum components of the standard projected Navier-slip source are
\begin{align*}
    R^{u}(U) = \frac{\gamma}{\varepsilon}u_b^{(N)}, \qquad
    R^{v}(U) = \frac{\gamma}{\varepsilon}v_b^{(N)}.
\end{align*}
For $i=1,\ldots,N$, the moment-friction and viscous-relaxation components are
\begin{align*}
    R^{\ap}_{i} =  (2i+1)\left(\frac{\gamma}{\varepsilon}u_b^{(N)}+\fl{\iRe}{\varepsilon h}\sum_{j=1}^{N}C_{ij}\ap_j\right), \quad
    R^{\beta}_{i} = (2i+1)\left(\frac{\gamma}{\varepsilon}v_b^{(N)}+\fl{\iRe}{\varepsilon h}\sum_{j=1}^{N}C_{ij}\beta_j\right).
\end{align*}

Constructing general two-dimensional moving equilibria requires solving a coupled system of stationary PDEs.
We therefore focus on quasi-two-dimensional flows, for which the stationary equations reduce to an ODE system.
Specifically, we impose
\begin{equation*}
    \py(\cdot) = 0, \quad \py h_b=0,
\end{equation*}
under which \eqref{eq:general-2d} becomes
\begin{equation}\label{eq:quasi-2d}
    U_t + A_{G}(U) \px U + S_{x}(U)\px h_b + R(U) = 0.
\end{equation}

We assume that the bottom is continuous across cell interfaces.
The HLL fluctuations approximate the path integral of $A_G(U)\,\md U$, while the topography and friction contributions enter the equilibrium-subtracted volume term.

\subsection{Construction of a prescribed discrete equilibrium}
\label{subsec:stationary-wb}

A stationary solution of \eqref{eq:quasi-2d} satisfies
\begin{equation*}
    A_G(U)U_x+S_x(U)\partial_xh_b+R(U)=0.
\end{equation*}
When $A_G(U)$ is invertible, the stationary equation becomes
\begin{equation*}
    U_x=-A_G(U)^{-1}\left(S_x(U)\partial_xh_b+R(U)\right).
\end{equation*}
The eigenvalues in Section~\ref{sec:hyperbolicity} show that $A_G(U)$ is singular if and only if
\begin{equation*}
    u_m=0,
    \quad\text{or}\quad
    |u_m|=\sqrt{Gh+3\sum_{j=1}^{N}\frac{\alpha_j^2}{2j+1}},
    \quad\text{or}\quad
    |u_m|=\sqrt{2\sum_{j=1}^{N}\frac{\alpha_j^2}{2j+1}}.
\end{equation*}
We assume that the prescribed moving branch avoids these singular states.
Lake-at-rest profiles are prescribed analytically.

Let $I_i=[x_{i-\frac12},x_{i+\frac12}]$ be a uniform cell of width $\Delta x$, and prescribe the left equilibrium interface value $U_{\frac12}^*$.
The discrete equilibrium is propagated from left to right by the midpoint-collocation relations
\begin{equation}\label{eq:stored-equilibrium}
\begin{aligned}
    K_i&=-A_G(U_i^*)^{-1}\left(S_x(U_i^*)\partial_xh_b(x_i)+R(U_i^*)\right),\\
    U_i^*-\frac{\Delta x}{2}K_i&=U_{i-\frac12}^*,\\
    U_{i+\frac12}^*&=U_i^*+\frac{\Delta x}{2}K_i.
\end{aligned}
\end{equation}
The nonlinear equation for $U_i^*$ is solved by damped Newton iteration, and $U_{i+\frac12}^*$ is then used as the left state in the next cell.
The stored cellwise profile is
\begin{equation*}
    U_i^*(x)=U_i^*+(x-x_i)K_i.
\end{equation*}
Adjacent profiles share the stored interface state $U_{i+\frac12}^*$, and $U_i^*$ is exactly the cell average of the stored linear profile:
\begin{equation}\label{eq:stored-equilibrium-average}
    \frac{1}{\Delta x}\int_{I_i}U_i^*(x)\,\md x=U_i^*.
\end{equation}
For a smooth continuous equilibrium $U^*(x)$, the cell average instead satisfies
\begin{equation}\label{eq:continuous-equilibrium-average}
    \overline{U}_i^*
    :=\frac{1}{\Delta x}\int_{I_i}U^*(x)\,\md x
    =U^*(x_i)+\frac{\Delta x^2}{24}\partial_{xx}U^*(x_i)+\mathcal O(\Delta x^4).
\end{equation}
Thus, if $U_i^*=U^*(x_i)+\mathcal O(\Delta x^2)$, then $U_i^*=\overline{U}_i^*+\mathcal O(\Delta x^2)$.

\subsection{Equilibrium-deviation reconstruction and HLL fluctuations}
\label{subsec:first-order-wb}

Let $U_i(t)$ denote the cell average
\begin{equation}\label{eq:finite-volume-cell-average}
    U_i(t)=\frac{1}{\Delta x}\int_{I_i}U(x,t)\,\md x,
\end{equation}
and define the deviation from the stored equilibrium average by
\begin{equation*}
    V_i=U_i-U_i^*.
\end{equation*}
The first-order method uses $\sigma_i=0$, whereas the second-order method uses the componentwise limited slope
\begin{equation}\label{eq:deviation-slope}
    \sigma_i=\operatorname{minmod}\left(
    \frac{V_i-V_{i-1}}{\Delta x},
    \frac{V_{i+1}-V_i}{\Delta x}
    \right),
\end{equation}
where
\begin{equation*}
    \operatorname{minmod}(a,b)=
    \begin{cases}
        \operatorname{sign}(a)\min(|a|,|b|), & ab>0,\\
        0, & ab\leq0.
    \end{cases}
\end{equation*}
The boundary extensions are $V_0=V_{N_x+1}=0$.
Both spatial orders use the equilibrium-deviation reconstruction
\begin{equation*}
    P_i(x)=U_i^*(x)+V_i+(x-x_i)\sigma_i.
\end{equation*}
This linear profile has cell average $U_i$, midpoint value $P_i(x_i)=U_i$, and slope $\partial_xP_i=K_i+\sigma_i$.
The traces at an interior interface are therefore
\begin{equation}\label{eq:wb-interface-traces}
\begin{aligned}
    U_{i+\frac12}^-&=U_{i+\frac12}^*+V_i+\frac{\Delta x}{2}\sigma_i,\\
    U_{i+\frac12}^+&=U_{i+\frac12}^*+V_{i+1}-\frac{\Delta x}{2}\sigma_{i+1}.
\end{aligned}
\end{equation}

For two traces $U_L$ and $U_R$, define the path jump
\begin{equation}\label{eq:path-jump-hll}
    \mathcal Q(U_L,U_R)=\int_0^1 A_G\big(U_L+s(U_R-U_L)\big)(U_R-U_L)\,\md s.
\end{equation}
The path integral in \eqref{eq:path-jump-hll} is evaluated by four-point Gauss--Legendre quadrature.
If the state derivatives of $A_G$ through order eight are uniformly bounded along the path and $\|U_R-U_L\|=\mathcal O(\Delta x)$, the quadrature error is $\mathcal O(\|U_R-U_L\|^9)=\mathcal O(\Delta x^9)$ and therefore does not limit second-order consistency.
Let
\begin{equation*}
\begin{aligned}
    c(U)&=\sqrt{Gh+3\sum_{j=1}^{N}\frac{\alpha_j^2}{2j+1}},\quad
    \lambda_{\min}(U)&=u_m-c(U),\quad
    \lambda_{\max}(U)&=u_m+c(U).
\end{aligned}
\end{equation*}
These are the extremal eigenvalues of $A_G(U)$ because the transverse and repeated characteristic speeds lie between them.
The HLL speed estimates are
\begin{equation*}
    s_L=\min\big(0,\lambda_{\min}(U_L),\lambda_{\min}(U_R)\big),
    \qquad
    s_R=\max\big(0,\lambda_{\max}(U_L),\lambda_{\max}(U_R)\big).
\end{equation*}
\begin{lemma}[pathwise HLL bounds]\label{lem:pathwise-hll-bounds}
Let $U_L,U_R\in\mathcal{A}_N$ have depths $h_L,h_R$, and let $\Phi(s)=U_L+s(U_R-U_L)$, $s\in[0,1]$, be the straight path in conservative variables.
Then the following bounds hold for every $s\in[0,1]$:
\begin{equation}\label{eq:pathwise-extremal-bounds}
\begin{aligned}
    \lambda_{\min}(\Phi(s))
    &\geq\min\big(\lambda_{\min}(U_L),\lambda_{\min}(U_R)\big),\\
    \lambda_{\max}(\Phi(s))
    &\leq\max\big(\lambda_{\max}(U_L),\lambda_{\max}(U_R)\big)
\end{aligned}
\end{equation}
Consequently, $s_L$ and $s_R$ bound every characteristic speed of $A_G(\Phi(s))$ along the path.
\end{lemma}

\begin{proof}
Write $q=hu_m$ and define $\boldsymbol z\in\mathR^N$ by
\begin{equation*}
    z_j=\sqrt{\frac{3}{2j+1}}\,h\alpha_j,
    \qquad j=1,\ldots,N.
\end{equation*}
Let
\begin{equation*}
    \Psi(h,\boldsymbol z)=\sqrt{Gh^3+\|\boldsymbol z\|_2^2}.
\end{equation*}
Since $h\,c(U)=\Psi(h,\boldsymbol z)$, the extremal eigenvalues satisfy
\begin{equation}\label{eq:extremal-conservative-form}
    \lambda_{\min}(U)=\frac{q-\Psi(h,\boldsymbol z)}{h},
    \qquad
    \lambda_{\max}(U)=\frac{q+\Psi(h,\boldsymbol z)}{h}.
\end{equation}
The function $\Psi$ is convex for $h>0$, since
\begin{equation*}
    \Psi(h,\boldsymbol z)
    =\left\|\big(\sqrt{G}\,h^{3/2},\boldsymbol z\big)\right\|_2,
\end{equation*}
where $h\mapsto h^{3/2}$ is convex and nonnegative, and the Euclidean norm is convex and nondecreasing in its nonnegative first argument.
Along the conservative straight path, $h_s$, $q_s$, and $\boldsymbol z_s$ are affine in $s$, and $h_s=(1-s)h_L+sh_R>0$.
Hence
\begin{equation}\label{eq:Psi-path-convexity}
    \Psi(h_s,\boldsymbol z_s)
    \leq(1-s)\Psi(h_L,\boldsymbol z_L)+s\Psi(h_R,\boldsymbol z_R).
\end{equation}
Set
\begin{equation*}
    a=\max\big(\lambda_{\max}(U_L),\lambda_{\max}(U_R)\big),
    \qquad
    b=\min\big(\lambda_{\min}(U_L),\lambda_{\min}(U_R)\big).
\end{equation*}
For $K=L,R$, \eqref{eq:extremal-conservative-form} gives
\begin{equation*}
    q_K+\Psi(h_K,\boldsymbol z_K)\leq ah_K,
    \qquad
    q_K-\Psi(h_K,\boldsymbol z_K)\geq bh_K.
\end{equation*}
Convexity and the affine dependence of $q_s$ and $h_s$ imply
\begin{align*}
    q_s+\Psi(h_s,\boldsymbol z_s)
    &\leq(1-s)\big(q_L+\Psi(h_L,\boldsymbol z_L)\big)
    +s\big(q_R+\Psi(h_R,\boldsymbol z_R)\big)
    \leq ah_s,\\
    q_s-\Psi(h_s,\boldsymbol z_s)
    &\geq(1-s)\big(q_L-\Psi(h_L,\boldsymbol z_L)\big)
    +s\big(q_R-\Psi(h_R,\boldsymbol z_R)\big)
    \geq bh_s.
\end{align*}
Division by $h_s>0$ proves \eqref{eq:pathwise-extremal-bounds}.
The remaining eigenvalues $u_m$ and $u_m\pm\sqrt{2\sum_{j=1}^N\alpha_j^2/(2j+1)}$ lie between $\lambda_{\min}$ and $\lambda_{\max}$, which completes the proof.
\end{proof}

When $s_L<0<s_R$, the path-conservative HLL state and fluctuations are
\begin{equation}\label{eq:path-conservative-hll}
\begin{aligned}
    U^{\mathrm{HLL}}&=\frac{s_RU_R-s_LU_L-\mathcal Q(U_L,U_R)}{s_R-s_L},\\
    D^-(U_L,U_R)&=s_L\big(U^{\mathrm{HLL}}-U_L\big),\quad
    D^+(U_L,U_R)&=s_R\big(U_R-U^{\mathrm{HLL}}\big).
\end{aligned}
\end{equation}
For one-sided wave configurations, the fluctuations are
\begin{equation}\label{eq:path-conservative-hll-one-sided}
    (D^-,D^+)=
    \begin{cases}
        (0,\mathcal Q), & s_L\geq0,\\
        (\mathcal Q,0), & s_R\leq0.
    \end{cases}
\end{equation}
In every case, $D^-+D^+=\mathcal Q$, which gives path consistency in the sense of \cite{Pares2006b,Castro2017}.

At the left boundary, the HLL pair is $(U_{\frac12}^*,U_{\frac12}^+)$, and at the right boundary it is $(U_{N_x+\frac12}^-,U_{N_x+\frac12}^*)$.
These prescribed exterior states are used at every time-integration stage.
We assume that all cell averages, reconstructed traces, path states, and intermediate stages remain in $\mathcal{A}_N$.
The method includes no positivity-preserving limiter or wetting--drying treatment.
\subsection{Volume correction, time integration, and preservation}
\label{subsec:second-order-wb}

For a differentiable cellwise profile $W$, define the interior differential contribution
\begin{equation}\label{eq:cell-interior-contribution}
    \mathcal G[W](x)
    =A_G(W(x))\partial_xW(x)+S_x(W(x))\partial_xh_b(x)+R(W(x)).
\end{equation}
The exact equilibrium-subtracted contribution associated with the reconstructed profile $P_i$ is
\begin{equation}\label{eq:exact-equilibrium-subtracted-volume}
    \widehat{\mathcal C}_i
    =\int_{I_i}\left(\mathcal G[P_i](x)-\mathcal G[U_i^*](x)\right)\,\md x,
\end{equation}
where $\mathcal G[U_i^*]$ acts on the stored cellwise profile $U_i^*(x)$.
The midpoint identities $P_i(x_i)=U_i$, $\partial_xP_i=K_i+\sigma_i$, $U_i^*(x_i)=U_i^*$, and $\partial_xU_i^*=K_i$ therefore give the discrete midpoint contribution
\begin{equation}\label{eq:second-order-volume-correction}
\begin{aligned}
    \mathcal C_i=\Delta x\Big[&A_G(U_i)(K_i+\sigma_i)-A_G(U_i^*)K_i\\
    &+\big(S_x(U_i)-S_x(U_i^*)\big)\partial_xh_b(x_i)
    +R(U_i)-R(U_i^*)\Big].
\end{aligned}
\end{equation}
This correction is required at both spatial orders because adding a constant deviation to the stored equilibrium profile does not generally yield a stationary profile.
Thus, even when $\sigma_i=0$ in the first-order method, $\mathcal C_i$ generally remains nonzero away from the prescribed equilibrium.
The stationary slope definition in \eqref{eq:stored-equilibrium} gives
\begin{equation*}
    \mathcal G[U_i^*](x_i)
    =A_G(U_i^*)K_i+S_x(U_i^*)\partial_xh_b(x_i)+R(U_i^*)=0.
\end{equation*}
If the reconstructed and stored profiles, $h_b$, $A_G$, $S_x$, and $R$ are sufficiently smooth with uniformly bounded derivatives in the fully wet admissible set, midpoint quadrature yields
\begin{equation*}
    \widehat{\mathcal C}_i=\mathcal C_i+\mathcal O(\Delta x^3),
    \qquad
    \int_{I_i}\mathcal G[U_i^*](x)\,\md x=\mathcal O(\Delta x^3).
\end{equation*}
It follows that
\begin{equation}\label{eq:volume-consistency}
    \mathcal C_i=\int_{I_i}\mathcal G[P_i](x)\,\md x+\mathcal O(\Delta x^3),
\end{equation}
so division by $\Delta x$ in the semi-discrete operator produces an $\mathcal O(\Delta x^2)$ consistency error from the equilibrium subtraction and midpoint volume quadrature.
In smooth regions where the limiter does not reduce the reconstruction order, the second-order interface reconstruction, the path-consistent HLL fluctuations, the four-point path quadrature, and \eqref{eq:volume-consistency} give a formal $\mathcal O(\Delta x^2)$ spatial truncation error.
SSP-RK2 has local temporal truncation error $\mathcal O(\Delta t^3)$, so under a stable step with $\Delta t=\mathcal O(\Delta x)$ the fully discrete method is formally second order for smooth solutions away from limiter extrema.
This formal consistency argument does not establish stability or positivity, and excludes limiter extrema and discontinuities.
The resulting semi-discrete operator is
\begin{equation*}
    \mathcal L_i(U)=-\frac{1}{\Delta x}\left(
    D_{i+\frac12}^{-}+D_{i-\frac12}^{+}+\mathcal C_i
    \right),
\end{equation*}
where each fluctuation is evaluated from the corresponding states in \eqref{eq:wb-interface-traces} and from the boundary pairs stated above.

The first-order method combines $\sigma_i=0$ with forward Euler,
\begin{equation}\label{eq:first-order-wb-euler}
    U_i^{n+1}=U_i^n+\Delta t\,\mathcal L_i(U^n).
\end{equation}
The second-order method combines the limited slope \eqref{eq:deviation-slope} with SSP-RK2,
\begin{equation}\label{eq:second-order-wb-ssprk2}
\begin{aligned}
    U_i^{(1)}&=U_i^n+\Delta t\,\mathcal L_i(U^n),\\
    U_i^{n+1}&=\frac12U_i^n+\frac12\left(
    U_i^{(1)}+\Delta t\,\mathcal L_i(U^{(1)})
    \right).
\end{aligned}
\end{equation}
In the numerical experiments, the hyperbolic time step is computed from characteristic speeds evaluated at the current cell averages:
\begin{equation*}
    \Delta t_{\mathrm{hyp}}=\mathrm{CFL}\,\frac{\Delta x}{
    \displaystyle\max_i\big(|u_m(U_i^n)|+c(U_i^n)\big)}.
\end{equation*}

Explicit treatment of wall friction and viscous relaxation also requires a source time-step restriction.
For the longitudinal family, define the conservative momentum--moment vector
\begin{equation*}
    \boldsymbol w=(hu_m,h\alpha_1,\ldots,h\alpha_N)^T.
\end{equation*}
The transverse family has the analogous vector $(hv_m,h\beta_1,\ldots,h\beta_N)^T$ and the same relaxation matrix below.
The source-only subsystem has constant depth and takes the linear form
\begin{equation}\label{eq:source-relaxation-matrix}
    \boldsymbol w_t=-\mathsf M(h)\boldsymbol w,
    \qquad
    \mathsf M(h)=\mathsf D_N\left[
    r_w(h)\mathbf 1\mathbf 1^T+r_\nu(h)\hat{\mathsf C}
    \right],
\end{equation}
where
\begin{equation*}
\begin{aligned}
    \mathsf D_N&=\operatorname{diag}(1,3,5,\ldots,2N+1),
    &\mathbf 1&=(1,\ldots,1)^T,\\
    \hat{\mathsf C}&=\begin{pmatrix}0&0\\0&\mathsf C\end{pmatrix},
    &\mathsf C&=(C_{ij})_{i,j=1}^N,\\
    r_w(h)&=\frac{\gamma}{\varepsilon h},
    &r_\nu(h)&=\frac{\iRe}{\varepsilon h^2}.
\end{aligned}
\end{equation*}
Let $h_{\min}^n=\min_i h_i^n$, $r_w^n=r_w(h_{\min}^n)$, and $r_\nu^n=r_\nu(h_{\min}^n)$.
A uniform bound on the maximum absolute row sum is
\begin{equation}\label{eq:source-rate-bound}
    \kappa_{\mathrm{src}}^n
    =\max\left\{
    (N+1)r_w^n,
    \max_{1\leq i\leq N}(2i+1)\left[
    (N+1)r_w^n+r_\nu^n\sum_{j=1}^N|C_{ij}|
    \right]
    \right\}.
\end{equation}
Thus, for every cell, the momentum--moment block of the source Jacobian satisfies $\rho(\mathsf M(h_i^n))\leq\|\mathsf M(h_i^n)\|_\infty\leq\kappa_{\mathrm{src}}^n$.
The frictional tests use
\begin{equation}\label{eq:source-time-step}
    \Delta t_{\mathrm{src}}
    =\frac{\mathrm{CFL}_{\mathrm{src}}}{\kappa_{\mathrm{src}}^n},
    \qquad
    \mathrm{CFL}_{\mathrm{src}}=0.5,
\end{equation}
with $\Delta t_{\mathrm{src}}=+\infty$ when $\kappa_{\mathrm{src}}^n=0$.
The actual explicit step is
\begin{equation*}
    \Delta t=\min\big(\Delta t_{\mathrm{hyp}},\Delta t_{\mathrm{src}}\big).
\end{equation*}
This choice ensures $\Delta t\,\|\mathsf M(h_i^n)\|_\infty\leq0.5$ but does not establish nonlinear stability or positivity of the coupled update.

\begin{proposition}[exact preservation of the prescribed discrete equilibrium]\label{prop:prescribed-equilibrium}
Suppose that the stored data $\{U_i^*,U_{i+\frac12}^*,K_i\}$ satisfy \eqref{eq:stored-equilibrium} exactly and that the same functions $A_G$, $S_x$, and $R$, the same bottom-slope samples $\partial_xh_b(x_i)$, and the same stored interface states are used in the branch construction, reconstruction, volume correction, and boundary states.
Then both the first-order method \eqref{eq:first-order-wb-euler} and the second-order method \eqref{eq:second-order-wb-ssprk2} preserve $U_i^*$ exactly in exact arithmetic.
\end{proposition}

\begin{proof}
Let $U^n=U^*$.
For either spatial order,
\begin{equation*}
    V_i=U_i^n-U_i^*=0,\qquad \sigma_i=0,\qquad P_i(x)=U_i^*(x).
\end{equation*}
At every interface, including the prescribed boundaries,
\begin{equation*}
\begin{aligned}
    U_{i+\frac12}^-&=U_{i+\frac12}^+=U_{i+\frac12}^*,\\
    \mathcal Q(U_{i+\frac12}^*,U_{i+\frac12}^*)&=\int_0^1 A_G(U_{i+\frac12}^*)\,0\,\md s=0,\\
    D_{i+\frac12}^-&=D_{i+\frac12}^+=0.
\end{aligned}
\end{equation*}
Moreover, \eqref{eq:second-order-volume-correction} gives
\begin{equation*}
\begin{aligned}
    \mathcal C_i(U^*)=\Delta x\Big[&A_G(U_i^*)K_i-A_G(U_i^*)K_i\\
    &+\big(S_x(U_i^*)-S_x(U_i^*)\big)\partial_xh_b(x_i)
    +R(U_i^*)-R(U_i^*)\Big]=0,
\end{aligned}
\end{equation*}
and hence
\begin{equation*}
    \mathcal L_i(U^*)=-\frac{1}{\Delta x}(0+0+0)=0.
\end{equation*}
The forward Euler update is therefore
\begin{equation*}
    U_i^{n+1}=U_i^*+\Delta t\,\mathcal L_i(U^*)=U_i^*,
\end{equation*}
while the SSP-RK2 stages satisfy
\begin{equation*}
\begin{aligned}
    U_i^{(1)}&=U_i^*+\Delta t\,\mathcal L_i(U^*)=U_i^*,\\
    U_i^{n+1}&=\frac12U_i^*+\frac12\left(U_i^{(1)}+\Delta t\,\mathcal L_i(U^{(1)})\right)
    =\frac12U_i^*+\frac12U_i^*=U_i^*.
\end{aligned}
\end{equation*}
\end{proof}
With identical stored interface states and consistent equilibrium subtraction, the residual at the stored branch is limited by roundoff.
The nonlinear-solver tolerance affects the accuracy of the stored branch, which must be assessed separately from its preservation.

\section{Numerical experiments}
\label{sec:tests}

{
In this section, several numerical tests are considered to validate the $G$-SWLME and the first- and second-order well-balanced schemes.
We first examine two-dimensional moment dynamics through mesh refinement and comparison with a vertically resolved hydrostatic reference.
We then test the preservation of lake-at-rest and prescribed quasi-two-dimensional moving equilibria, followed by the evolution of finite-amplitude and weak perturbations.
Subsequently, we use radial-collapse problems, starting either from rest or with initial vertical shear, to compare the moment-model solutions with three-dimensional OpenFOAM simulations.
}
\subsection{Two-dimensional moment dynamics}
\label{subsec:moment-dynamics-test}

We consider two numerical studies of smooth, two-dimensional flows with nonzero moments.
The first assesses spatial discretization error for each fixed order $N=0,1,2$ of the $G\text{-}\mathrm{SWLME}$ by mesh refinement.
The second compares these models with a vertically resolved numerical solution of the hydrostatic system \eqref{eq:normalized-hydrostatic-system}, specialized to a flat bed with $\iRe=\gamma=0$.
Following the reference-solution approach of \cite{Kowalski2019}, we discretize the vertical coordinate directly; the present reference calculation retains both horizontal coordinates.

The dimensionless domain is $(x,y)\in[0,1]^2$, both horizontal directions are periodic, the bottom is flat, $G=1$, and the final time is $T=0.1$.
The initial depth and mean velocities are
\begin{equation}\label{eq:moment-dynamics-initial-means}
\begin{aligned}
h(x,y,0)&=1+0.05\cos(2\pi x)\cos(2\pi y),\\
u_m(x,y,0)&=0.2+0.02\sin(2\pi y),\\
v_m(x,y,0)&=0.1+0.02\sin(2\pi x).
\end{aligned}
\end{equation}
The nonzero initial moment coefficients are
\begin{equation}\label{eq:moment-dynamics-initial-moments}
\begin{aligned}
\alpha_1(x,y,0)&=0.03\left[1+0.2\cos(2\pi x)\right],\\
\beta_1(x,y,0)&=0.02\left[1+0.2\sin(2\pi y)\right],\\
\alpha_2(x,y,0)&=0.01\sin\left(2\pi(x+y)\right),\\
\beta_2(x,y,0)&=0.015\cos\left(2\pi(x-y)\right).
\end{aligned}
\end{equation}
With $\phi_1(\zeta)=1-2\zeta$ and $\phi_2(\zeta)=1-6\zeta+6\zeta^2$, the vertically resolved initial velocities are
\begin{equation}\label{eq:moment-dynamics-initial-profiles}
\begin{aligned}
\widetilde u(x,y,\zeta,0)&=u_m(x,y,0)+\alpha_1(x,y,0)\phi_1(\zeta)+\alpha_2(x,y,0)\phi_2(\zeta),\\
\widetilde v(x,y,\zeta,0)&=v_m(x,y,0)+\beta_1(x,y,0)\phi_1(\zeta)+\beta_2(x,y,0)\phi_2(\zeta).
\end{aligned}
\end{equation}
All higher moments are initially zero, and models with $N<2$ use the corresponding truncation of \eqref{eq:moment-dynamics-initial-profiles}.
The choice $\iRe=\gamma=0$ eliminates vertical viscous diffusion and bottom friction.
The mapped vertical velocity satisfies $\omega=0$ at $\zeta=0,1$.

For the moment calculations, monotonized-central MUSCL reconstruction is combined with a local Lax--Friedrichs flux for the conservative part, centered differences for the smooth nonconservative products, and third-order SSP Runge--Kutta time stepping.
This discretization is used only for the smooth flows considered here.
Table~\ref{tab:moment-dynamics-setup} lists the refinement grids and CFL numbers, where $n=N_x=N_y$ and $N_\zeta$ is the number of vertical layers in the reference calculation.

\begin{table}[H]
\centering
\small
\begin{tabular}{lccc}
\toprule
Calculation & $n$ & $N_\zeta$ & CFL \\
\midrule
Moment spatial refinement & $32,64,128,256,512,1024$ & -- & $0.3$ \\
Moment temporal refinement & $256$ & -- & $0.3,0.15,0.075$ \\
Reference horizontal refinement & $32,64,128,256$ & $64$ & $0.2$ \\
Reference vertical refinement & $128$ & $16,32,64,128$ & $0.2$ \\
Reference temporal refinement & $128$ & $64$ & $0.2,0.1,0.05$ \\
\bottomrule
\end{tabular}
\caption{Refinement parameters for the two-dimensional moment studies.
All moment calculations use $N=0,1,2$; the horizontal grid has $n^2$ cells.}
\label{tab:moment-dynamics-setup}
\end{table}

We estimate spatial discretization error by comparing each moment system on successive grids.
For each conservative component $q$, let $q_n$ denote its numerical approximation on the $n^2$ grid at $t=0.1$, and define
\begin{equation*}
E_n(q)=\left\|q_n-\mathcal R_n q_{2n}\right\|_{L^2},
\qquad
p_n(q)=\log_2\!\left(\frac{E_{n/2}(q)}{E_n(q)}\right),
\end{equation*}
where $\mathcal R_n$ averages each block of four fine-grid cell values onto the corresponding coarse cell, and the norm is the area-weighted discrete $L^2$ norm.
The observed order $p_n(q)$ is reported for $n\geq64$.
Table~\ref{tab:moment-dynamics-convergence} shows decreasing grid differences and approximately second-order spatial convergence for all conservative components of the $N=2$ model.
The results for $N=0,1$ exhibit similar convergence rates and are omitted for brevity.

\begin{table}[H]
\centering
\setlength{\tabcolsep}{2.4pt}
\renewcommand{\arraystretch}{1.10}
\resizebox{0.85\linewidth}{!}{\begin{tabular}{cccc*{4}{!{\hspace{7pt}}cc}}
\toprule
\multirow{2}{*}{$N$}
& \multirow{2}{*}{Moment}
& \multicolumn{2}{c}{$n=32$}
& \multicolumn{2}{c}{$n=64$}
& \multicolumn{2}{c}{$n=128$}
& \multicolumn{2}{c}{$n=256$}
& \multicolumn{2}{c}{$n=512$} \\
\cmidrule(lr){3-4}\cmidrule(lr){5-6}\cmidrule(lr){7-8}\cmidrule(lr){9-10}\cmidrule(lr){11-12}
&
& Error & Rate
& Error & Rate
& Error & Rate
& Error & Rate
& Error & Rate \\
\midrule
\multirow{7}{*}{2}
& $h$          & $7.748\mathrm{E}{-5}$ & -- & $1.526\mathrm{E}{-5}$ & 2.344 & $3.189\mathrm{E}{-6}$ & 2.259 & $6.420\mathrm{E}{-7}$ & 2.312 & $1.263\mathrm{E}{-7}$ & 2.346 \\
& $hu_m$       & $7.548\mathrm{E}{-5}$ & -- & $2.090\mathrm{E}{-5}$ & 1.853 & $5.013\mathrm{E}{-6}$ & 2.060 & $1.216\mathrm{E}{-6}$ & 2.044 & $2.834\mathrm{E}{-7}$ & 2.101 \\
& $hv_m$       & $7.611\mathrm{E}{-5}$ & -- & $1.851\mathrm{E}{-5}$ & 2.040 & $4.533\mathrm{E}{-6}$ & 2.030 & $1.057\mathrm{E}{-6}$ & 2.101 & $2.551\mathrm{E}{-7}$ & 2.050 \\
& $h\alpha_1$ & $2.783\mathrm{E}{-5}$ & -- & $8.438\mathrm{E}{-6}$ & 1.722 & $2.448\mathrm{E}{-6}$ & 1.785 & $6.855\mathrm{E}{-7}$ & 1.836 & $1.875\mathrm{E}{-7}$ & 1.870 \\
& $h\beta_1$  & $1.574\mathrm{E}{-5}$ & -- & $4.753\mathrm{E}{-6}$ & 1.727 & $1.312\mathrm{E}{-6}$ & 1.857 & $3.391\mathrm{E}{-7}$ & 1.952 & $8.347\mathrm{E}{-8}$ & 2.022 \\
& $h\alpha_2$ & $6.455\mathrm{E}{-5}$ & -- & $1.970\mathrm{E}{-5}$ & 1.712 & $4.868\mathrm{E}{-6}$ & 2.017 & $1.056\mathrm{E}{-6}$ & 2.205 & $2.181\mathrm{E}{-7}$ & 2.275 \\
& $h\beta_2$  & $9.198\mathrm{E}{-5}$ & -- & $2.818\mathrm{E}{-5}$ & 1.707 & $6.884\mathrm{E}{-6}$ & 2.033 & $1.442\mathrm{E}{-6}$ & 2.255 & $2.933\mathrm{E}{-7}$ & 2.298 \\
\bottomrule
\end{tabular}
}
\caption{Spatial refinement of the $G\text{-}\mathrm{SWLME}$ with $N=2$.
The Error and Rate columns report $E_n(q)$ and $p_n(q)$, respectively.}
\label{tab:moment-dynamics-convergence}
\end{table}

In the second study, the $N=0,1,2$ moment solutions are compared with a vertically resolved approximation of \eqref{eq:normalized-hydrostatic-system} on a uniform grid in $\zeta$.
The reference solver uses monotonized-central MUSCL reconstruction, a common-speed Rusanov flux in each horizontal direction, and third-order SSP Runge--Kutta time stepping.
At each Runge--Kutta stage, the vertical mass flux $W=h\omega$ is recovered from discrete continuity, satisfying the zero-flux conditions at the bottom and free surface to roundoff.
Vertical advection uses limited upwind reconstruction, with the time step constrained by both horizontal and vertical advective CFL conditions.
Depth averages and moment coefficients are computed by midpoint quadrature in $\zeta$ with the same Legendre normalization as the moment models.
Table~\ref{tab:moment-dynamics-setup} specifies the independent horizontal, vertical, and temporal refinement checks.
Figures~\ref{fig:moment-dynamics-fields}--\ref{fig:moment-dynamics-profiles} compare the $256^2$ moment solutions with the $256^2\times128$ reference at $t=0.1$.
The moment and reference solutions use CFL numbers $0.3$ and $0.1$, respectively.

\begin{figure}[H]
\centering
\includegraphics[width=0.85\linewidth]{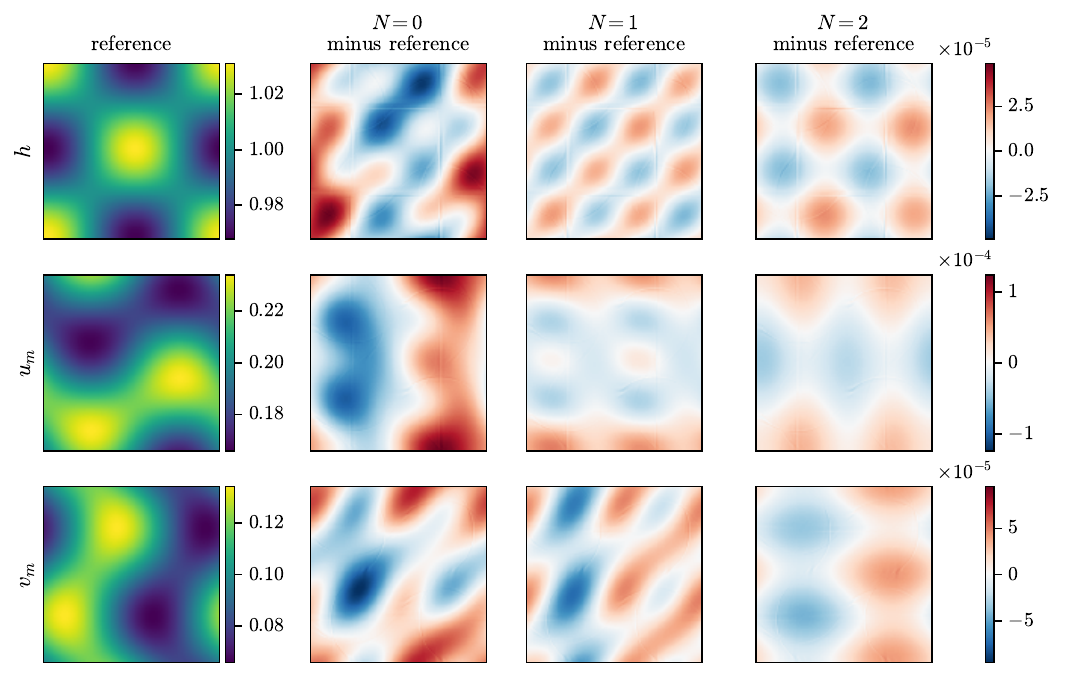}
\caption{Reference fields $h$, $u_m$, and $v_m$ (first column) and the corresponding moment-model differences from the reference for $N=0,1,2$ at $t=0.1$.}
\label{fig:moment-dynamics-fields}
\end{figure}

Figure~\ref{fig:moment-dynamics-fields} shows small differences in depth and depth-averaged velocities relative to the variations of the reference fields.
The $N=2$ model generally gives closer agreement than $N=0$, although the improvement varies across components and locations.

\begin{figure}[H]
\centering
\includegraphics[width=0.45\linewidth]{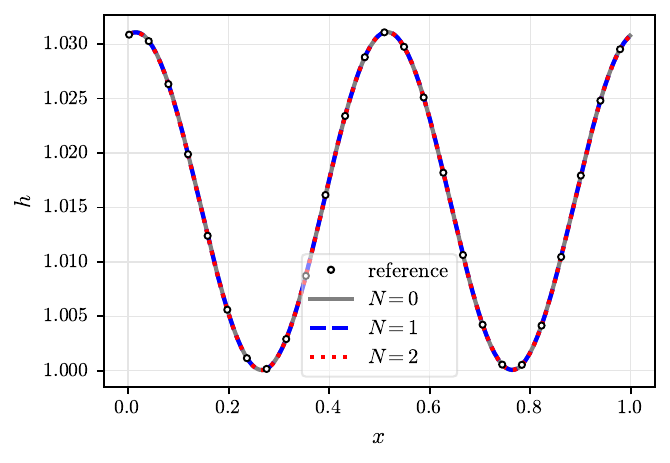}
\caption{Water depth at $t=0.1$ on the diagonal transect $y=x$.
Open black circles denote the reference, while gray solid, blue dashed, and red dotted curves denote $N=0,1,2$, respectively.}
\label{fig:moment-dynamics-height}
\end{figure}

\begin{figure}[H]
\centering
\includegraphics[width=0.65\linewidth]{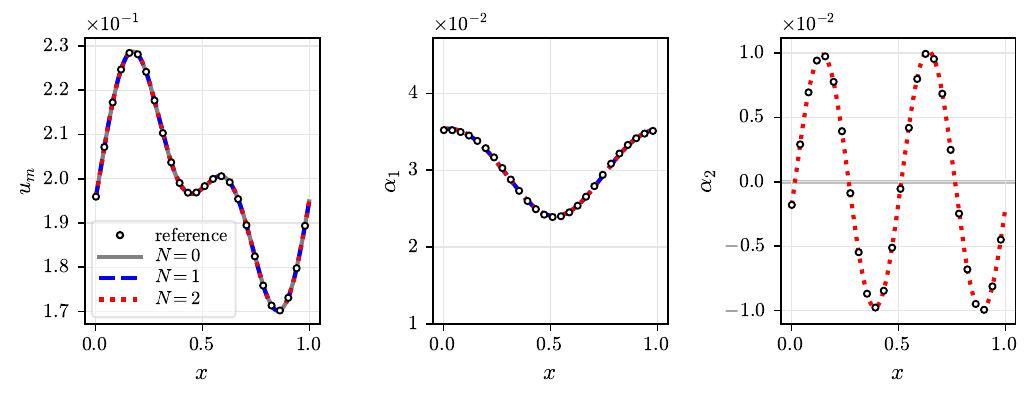}
\caption{Comparison of $u_m$, $\alpha_1$, and $\alpha_2$ from left to right at $t=0.1$ on the diagonal transect $y=x$.
The symbols and line styles are those of Figure~\ref{fig:moment-dynamics-height}, and a moment curve appears only when that coefficient is retained by the corresponding model order.}
\label{fig:moment-dynamics-u}
\end{figure}

\begin{figure}[H]
\centering
\includegraphics[width=0.65\linewidth]{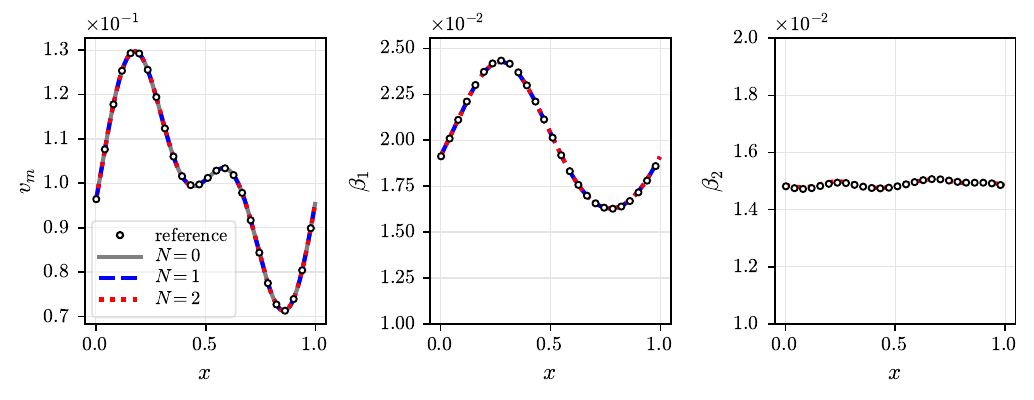}
\caption{Comparison of $v_m$, $\beta_1$, and $\beta_2$ from left to right at $t=0.1$ on the diagonal transect $y=x$.
The symbols and line styles are those of Figure~\ref{fig:moment-dynamics-height}, and a moment curve appears only when that coefficient is retained by the corresponding model order.}
\label{fig:moment-dynamics-v}
\end{figure}

The diagonal profiles in Figures~\ref{fig:moment-dynamics-height}--\ref{fig:moment-dynamics-v} show close agreement in depth and depth-averaged velocities for all three moment orders.
The $N=1,2$ models also reproduce the first moments, while the $N=2$ model captures the second moments of the reference solution.

\begin{figure}[H]
\centering
\includegraphics[width=0.65\linewidth]{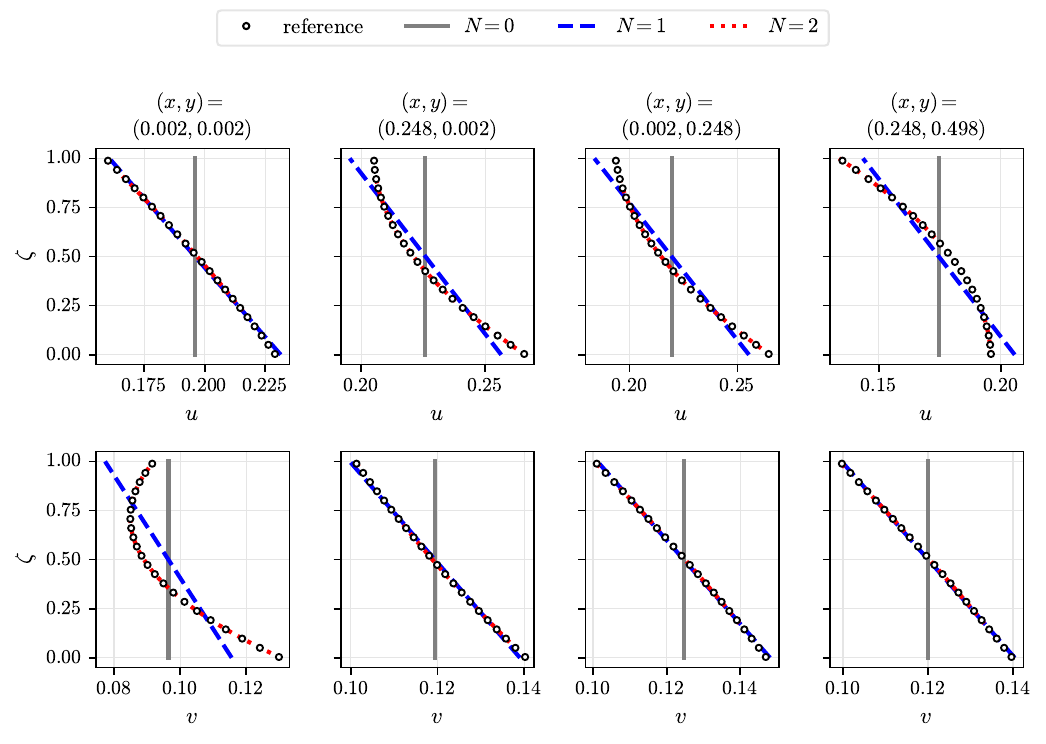}
\caption{Reference and reconstructed $u(\zeta)$ and $v(\zeta)$ profiles at $t=0.1$ in four representative water columns.
The cell-center coordinates are stated above the upper panels, and the symbols and line styles are those of Figure~\ref{fig:moment-dynamics-height}.}
\label{fig:moment-dynamics-profiles}
\end{figure}

Figure~\ref{fig:moment-dynamics-profiles} illustrates the effect of moment order on the reconstructed vertical profiles of the horizontal velocities.
The $N=0$ model gives depth-uniform velocities, whereas $N=1$ represents linear variation and $N=2$ additionally captures curvature.
The quadratic reconstructions closely follow the reference profiles at the selected locations, demonstrating improved representation of vertical structure for this smooth test.
Reference-grid sensitivity limits a quantitative interpretation of the smallest componentwise differences.

\subsection{Lake-at-rest preservation}
\label{subsec:lake-at-rest-test}
\label{ex:lake-at-rest}

We test lake-at-rest preservation over a non-flat bottom for the quasi-two-dimensional system \eqref{eq:quasi-2d} with $N=2$, $x\in[0,1]$, $G=1$, and $\gamma=\iRe=0$.
The state vector is
\begin{equation}\label{eq:lake-rest-state}
U=(h,hu_m,hv_m,h\alpha_1,h\beta_1,h\alpha_2,h\beta_2)^T.
\end{equation}
The prescribed equilibrium satisfies
\begin{equation}\label{eq:lake-rest-equilibrium}
\eta^*=h^*+h_b=\eta_0=1,\qquad
u_m^*=v_m^*=\alpha_1^*=\beta_1^*=\alpha_2^*=\beta_2^*=0,
\end{equation}
with bottom topography
\begin{equation}\label{eq:lake-rest-bottom}
h_b(x)
=0.15+0.20\exp\!\left[-100(x-0.32)^2\right]
+0.10\exp\!\left[-200(x-0.70)^2\right]
+0.03\sin(2\pi x).
\end{equation}

The bottom is represented by its continuous piecewise-linear interpolant through the cell-face values, with $h_{b,i}=(h_{b,i-\frac12}+h_{b,i+\frac12})/2$ and $(\partial_xh_b)_i=(h_{b,i+\frac12}-h_{b,i-\frac12})/\Delta x$.
The equilibrium depths $h_i^*=\eta_0-h_{b,i}$ and $h_{i\pm\frac12}^*=\eta_0-h_{b,i\pm\frac12}$ satisfy the discrete stationary balance, and the corresponding equilibrium states are imposed at both boundaries.
Since $A_G(U^*)$ is singular at rest, these states are prescribed analytically.

We compare the standard first-order path-conservative HLL method with the first- and second-order well-balanced methods of Section~\ref{sec:well-balanced}, denoted by WB1 and WB2.
All three methods use the same equilibrium and boundary data, HLL speed estimates, and four-point Gauss--Legendre path quadrature.
The HLL baseline uses a midpoint bed-slope source and forward Euler time stepping; WB1 and WB2 use the reconstruction, volume correction, and time integration defined in Section~\ref{sec:well-balanced}.
Table~\ref{tab:lake-rest-setup} summarizes the remaining numerical parameters.

\begin{table}[H]
\centering
\begin{tabular}{ll}
\toprule
Quantity & Value \\
\midrule
Mesh sequence & $N_x=50,100,200,400$ \\
CFL number and final time & $\mathrm{CFL}=0.25$ and $T=10$ \\
History output & $N_x=200$ at $t_k=k/4$, $k=0,\ldots,40$ \\
\bottomrule
\end{tabular}
\caption{Numerical parameters for the lake-at-rest preservation test.}
\label{tab:lake-rest-setup}
\end{table}

For $\mathsf{m}\in\{\mathrm{HLL},\mathrm{WB1},\mathrm{WB2}\}$, let $\mathcal L_h^{\mathsf{m}}$ denote the spatial operator and define
\begin{equation}\label{eq:lake-rest-diagnostics}
\begin{aligned}
R_\infty^{\mathsf{m}}
&=\max_i\left\|\mathcal L_{h,i}^{\mathsf{m}}(U^*)\right\|_\infty,\\
E_\eta^{\mathsf{m}}(t)
&=\left\|h^{\mathsf{m}}(t)+h_b-\eta_0\right\|_\infty,\\
E_u^{\mathsf{m}}(t)
&=\left\|u_m^{\mathsf{m}}(t)\right\|_\infty.
\end{aligned}
\end{equation}
Table~\ref{tab:lake-rest-results} reports the initial residuals and the errors at $t=10$ for each method and mesh.
The initial residual measures imbalance in all state components; $E_\eta(10)$ measures the final free-surface error.

\begin{figure}[H]
\centering
\includegraphics[width=0.5\linewidth]{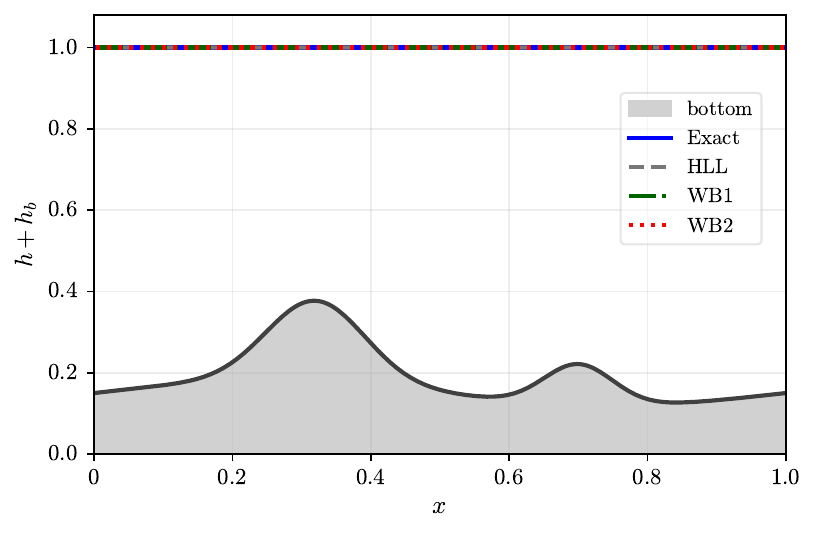}
\caption{Bottom topography and free-surface elevation $\eta=h+h_b$ at $t=10$ on the $N_x=200$ mesh.
The shaded region shows the bottom, and the blue solid line (Exact) denotes the exact free surface $\eta_0=1$.
Gray dashed, dark green dash-dotted, and red dotted lines denote HLL, WB1, and WB2, respectively.
The surface errors are shown in Figure~\ref{fig:lake-surface-error}.}
\label{fig:lake-surface-profile}
\end{figure}

\begin{figure}[H]
\centering
\includegraphics[width=0.5\linewidth]{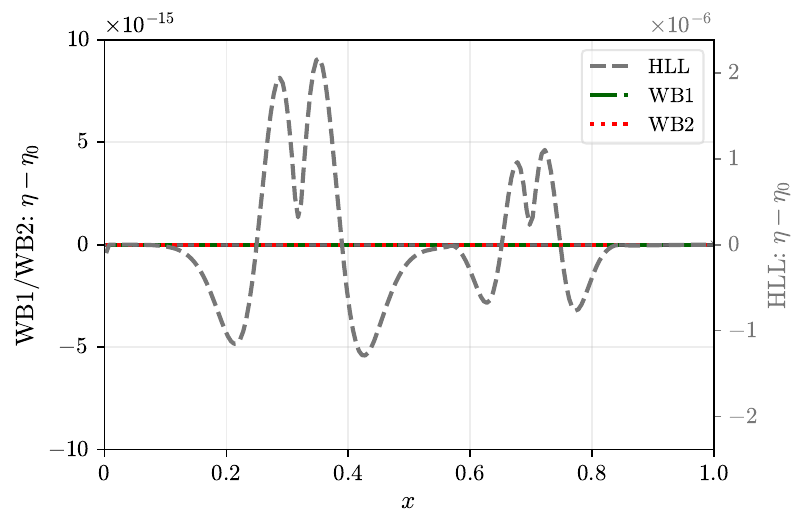}
\caption{Free-surface error $h+h_b-\eta_0$ at $t=10$ on the $N_x=200$ mesh.
The left axis shows the WB1 and WB2 errors, and the gray right axis shows the HLL error.
The WB1 and WB2 errors are zero in the saved double-precision data, whereas the maximum absolute HLL error is $10^{-6}$.}
\label{fig:lake-surface-error}
\end{figure}

Figure~\ref{fig:lake-surface-profile} compares the computed free surfaces with the exact level $\eta_0=1$ over the non-flat bottom at $t=10$.
All three numerical curves appear to coincide with the exact surface at this plotting scale.
Figure~\ref{fig:lake-surface-error} resolves the difference between the methods: WB1 and WB2 have zero computed surface error, whereas HLL produces spatially varying errors of order $10^{-6}$ near the bottom variations.

\begin{figure}[H]
\centering
\includegraphics[width=0.45\linewidth]{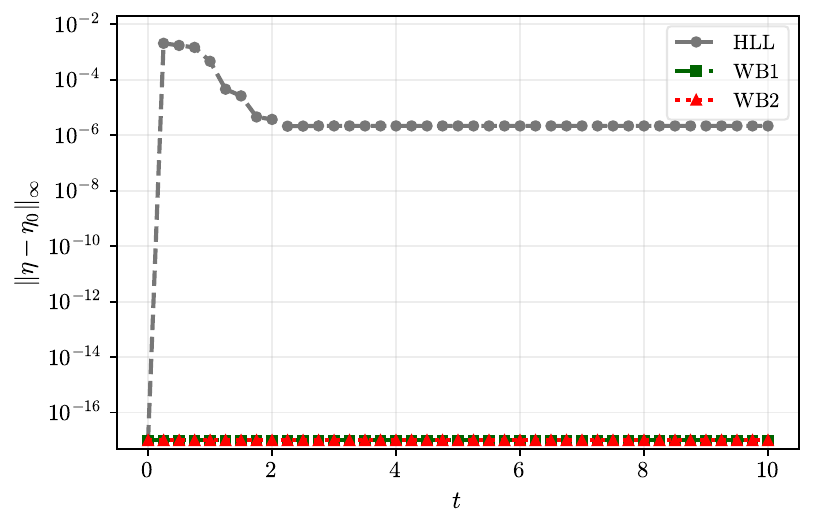}\hfill
\includegraphics[width=0.45\linewidth]{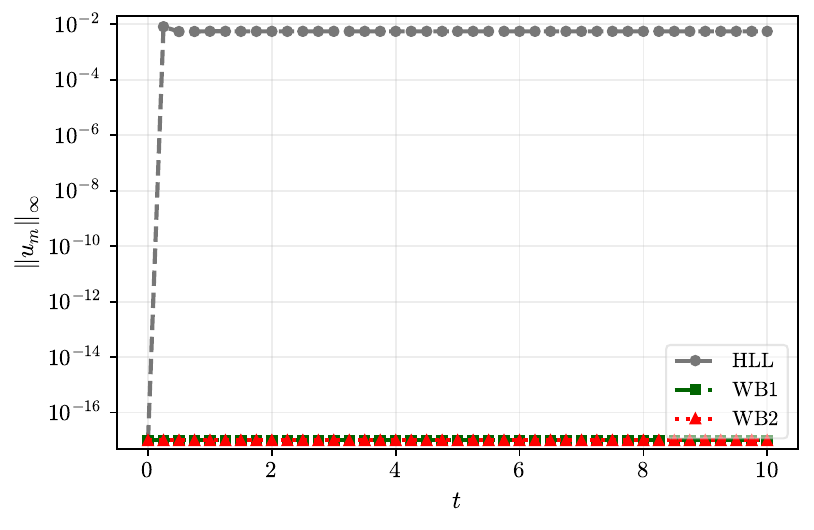}
\caption{Temporal histories of the free-surface error $E_\eta(t)$ (left) and velocity error $E_u(t)$ (right) on the $N_x=200$ mesh.
The method colors and line styles are those of Figure~\ref{fig:lake-surface-profile}.
Zero errors for WB1 and WB2 are displayed at $10^{-17}$ on the logarithmic axes.}
\label{fig:lake-error-histories}
\end{figure}

Figure~\ref{fig:lake-error-histories} shows that WB1 and WB2 maintain zero computed free-surface and mean-velocity errors at all recorded times.
The standard HLL method develops spurious errors from the stationary initial data; after an initial transient, the surface error decreases to a nonzero level while the velocity error persists.
Thus, both well-balanced methods maintain the rest state throughout the simulation without generating spurious motion.

\begin{table}[H]
\centering
\begin{tabular}{llcccc}
\toprule
Method & $N_x$ & $R_\infty(t=0)$ & $E_\eta$ & $E_u$ & $\min_i h_i$ \\
\midrule
\multirow[c]{4}{*}{HLL}
& 50  & $0.305$ & $1.114\mathrm{E}{-4}$ & $2.220\mathrm{E}{-2}$ & $0.626$ \\
& 100 & $0.165$ & $1.628\mathrm{E}{-5}$ & $1.109\mathrm{E}{-2}$ & $0.623$ \\
& 200 & $8.497\mathrm{E}{-2}$ & $2.167\mathrm{E}{-6}$ & $5.545\mathrm{E}{-3}$ & $0.623$ \\
& 400 & $4.355\mathrm{E}{-2}$ & $2.756\mathrm{E}{-7}$ & $2.769\mathrm{E}{-3}$ & $0.623$ \\
\midrule
\multirow[c]{4}{*}{WB2}
& 50  & $5.161\mathrm{E}{-15}$ & $0$ & $0$ & $0.626$ \\
& 100 & $2.065\mathrm{E}{-14}$ & $0$ & $0$ & $0.623$ \\
& 200 & $3.964\mathrm{E}{-14}$ & $0$ & $0$ & $0.623$ \\
& 400 & $8.115\mathrm{E}{-14}$ & $0$ & $0$ & $0.623$ \\
\bottomrule
\end{tabular}\caption{Lake-at-rest calculation: initial spatial residual $R_\infty(t=0)$, together with errors $E_\eta$, $E_u$ and minimum depth $\min_i h_i$ evaluated at $t=10$.
The residual and errors are defined in~\eqref{eq:lake-rest-diagnostics}.
Entries reported as $0$ denote exactly zero computed errors in double-precision arithmetic; no tolerance-based truncation was applied.}
\label{tab:lake-rest-results}
\end{table}

Table~\ref{tab:lake-rest-results} confirms that WB2 preserves the lake-at-rest equilibrium to machine precision on all tested meshes.
WB1 gives identical values for every reported quantity and is omitted for brevity.
All moment components remain identically zero for every method and mesh, so their errors are omitted from the table.
The initial WB residuals are at roundoff level, and the computed free-surface and velocity errors at $t=10$ are zero in double precision.
These results verify the lake-at-rest well-balanced property, including its preservation under the second-order reconstruction and SSP-RK2 time integration of WB2.

The HLL errors decrease under mesh refinement, but the equilibrium is not preserved on any tested mesh.
Since HLL and WB1 are both first-order methods, their comparison demonstrates the benefit of maintaining the discrete balance between the flux and bottom-slope contributions.
This test concerns the special equilibrium with vanishing velocities and moments; Section~\ref{subsec:moving-equilibrium-test} examines moving-equilibrium preservation.

\subsection{Moving-equilibrium preservation}
\label{subsec:moving-equilibrium-test}
\label{ex:moving-equilibrium}

We test the construction and preservation of two quasi-two-dimensional moving equilibria: a frictionless flow with vanishing transverse components and a dissipative flow with nonzero transverse velocity and both moment families.
Both cases use $N=2$, $G=1$, $x\in[0,1]$, and the bottom
\begin{equation}\label{eq:moving-equilibrium-bottom}
h_b(x)=0.1\exp\!\left[-\left(\frac{x-0.5}{0.15}\right)^2\right].
\end{equation}
We use $N_x=100,200,400,800$, CFL number $0.25$, and final time $T=10$.
The HLL, WB1, and WB2 methods are those of Example~\ref{ex:lake-at-rest}, with common initial data and fixed equilibrium boundary states.
The discrete equilibrium $U_h^*$ is constructed by the midpoint-collocation scheme~\eqref{eq:stored-equilibrium}, using damped Newton iteration and a left-to-right march.
Its midpoint values initialize the simulations and are compared with an independently computed stationary reference $U^*$ to assess construction accuracy.

For a primitive component $q\in\{h,u_m,v_m,\alpha_1,\beta_1,\alpha_2,\beta_2\}$, define
\begin{equation}\label{eq:moving-equilibrium-diagnostics}
E_q
=\left\|q(U_h^*)-q(U^*)\right\|_{L_h^1},
\qquad
D_q(t)
=\left\|q(U_h(t))-q(U_h^*)\right\|_{L_h^1},
\qquad
D_{\max}(t)=\max_q D_q(t).
\end{equation}
Both $U^*$ and $U_h^*$ are evaluated at cell midpoints, so $E_q$ measures construction error, whereas $D_q(t)$ measures drift from the stored discrete equilibrium.

\subsubsection{Frictionless moving equilibrium}
\label{subsec:frictionless-equilibrium}

The first case is frictionless, with $\gamma=\iRe=0$, and uses the left primitive state
\begin{equation}\label{eq:example3a-left-state}
(h,u_m,v_m,\alpha_1,\beta_1,\alpha_2,\beta_2)(0)
=(1,0.2,0,0.05,0,0.02,0).
\end{equation}
On the invariant subspace $v_m=\beta_1=\beta_2=0$, an independent continuous reference is obtained from
\begin{equation}\label{eq:example3a-invariants}
hu_m=Q,
\qquad
\frac{u_m^2}{2}+G(h+h_b)+\frac{3}{2}\sum_{j=1}^{2}\frac{\alpha_j^2}{2j+1}=E,
\qquad
\frac{\alpha_j}{h}=C_j,
\end{equation}
where $Q$, $E$, and $C_j$ are fixed by the prescribed left state.
Substituting $u_m=Q/h$ and $\alpha_j=C_jh$ into the energy relation gives a scalar equation for $h$, whose positive deep root is computed by Brent's method.
The remaining components follow from the invariants, independently of the discrete equilibrium construction.

\begin{figure}[H]
\centering
\includegraphics[width=0.48\linewidth]{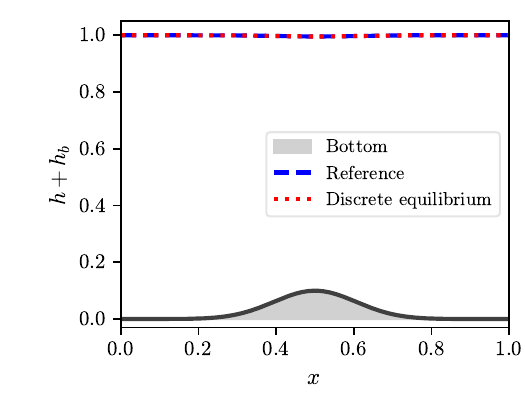}\hfill
\includegraphics[width=0.48\linewidth]{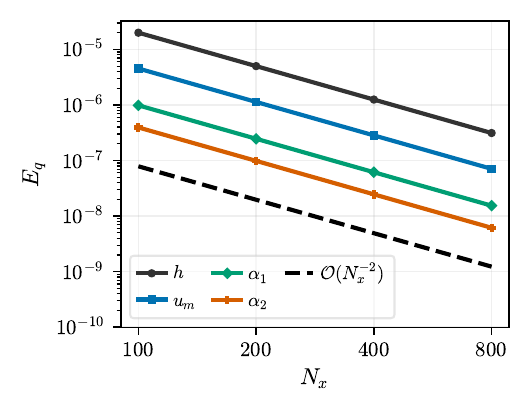}
\caption{Frictionless moving equilibrium: stationary free-surface profiles on the $N_x=800$ mesh (left) and componentwise equilibrium-construction errors $E_q$ under mesh refinement (right).
The gray shaded region denotes the bottom, the blue dashed curve (Reference) is the independently computed surface $h^*(x_i)+h_b(x_i)$, and the red dotted curve (Discrete equilibrium) is the stored midpoint-collocation surface $h_{h,i}^*+h_b(x_i)$.
The error panel shows $q=h,u_m,\alpha_1,\alpha_2$; the transverse component errors are identically zero and are omitted from the logarithmic axis.
The black dashed line is a vertically offset $\mathcal O(N_x^{-2})$ guide.}
\label{fig:frictionless-equilibrium}
\end{figure}

Figure~\ref{fig:frictionless-equilibrium} shows second-order convergence of the discrete equilibrium in all nonzero components; the transverse components remain identically zero.
The reference and discrete free surfaces are visually indistinguishable on the finest mesh.

\begin{figure}[H]
\centering
\includegraphics[width=0.5\linewidth]{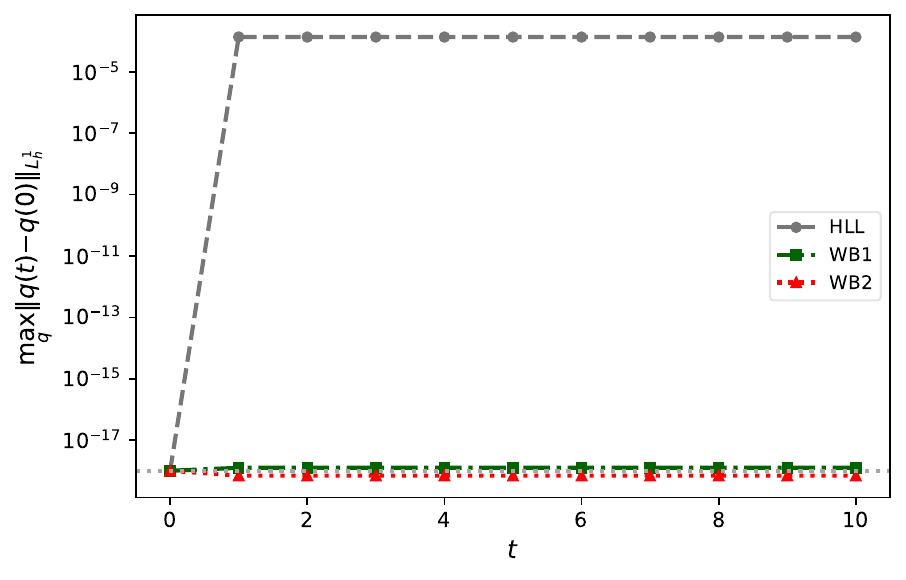}
\caption{Frictionless moving equilibrium: temporal history of the maximum componentwise drift $D_{\max}(t)$ on the $N_x=800$ mesh.
Exact zeros are displayed at $10^{-18}$ only to make them visible on the logarithmic axis.}
\label{fig:frictionless-drift}
\end{figure}

\begin{table}[H]
\centering
\begin{tabular}{rccc}
\toprule
$N_x$ & HLL & WB1 & WB2 \\
\midrule
100 & $1.097\mathrm{E}{-3}$ & $1.110\mathrm{E}{-18}$ & $0$ \\
200 & $5.478\mathrm{E}{-4}$ & $1.665\mathrm{E}{-18}$ & $1.110\mathrm{E}{-18}$ \\
400 & $2.738\mathrm{E}{-4}$ & $1.110\mathrm{E}{-18}$ & $2.776\mathrm{E}{-19}$ \\
800 & $1.369\mathrm{E}{-4}$ & $1.249\mathrm{E}{-18}$ & $6.939\mathrm{E}{-19}$ \\
\bottomrule
\end{tabular}
\caption{Frictionless moving equilibrium: final maximum componentwise $L_h^1$ drift $D_{\max}(10)$.}
\label{tab:frictionless-equilibrium-drift}
\end{table}

Figure~\ref{fig:frictionless-drift} and Table~\ref{tab:frictionless-equilibrium-drift} show that WB1 and WB2 preserve the stored frictionless equilibrium to machine precision throughout the simulation.
In contrast, HLL generates nonzero drift that decreases at approximately first order under mesh refinement.

\subsubsection{Dissipative moving equilibrium}
\label{subsec:dissipative-equilibrium}

The dissipative case has nonzero transverse velocity and both moment families, with left state
\begin{equation}\label{eq:example3b-left-state-numerical}
(h,u_m,v_m,\alpha_1,\beta_1,\alpha_2,\beta_2)(0)
=(1,0.4,0.15,0.08,-0.04,-0.03,0.05).
\end{equation}
The physical parameters are
\begin{equation}\label{eq:example3b-parameters}
\varepsilon=0.1,
\qquad
\gamma=0.002,
\qquad
\iRe=0.0005.
\end{equation}
The source is the projected Navier-slip term $R(U)$ in~\eqref{eq:quasi-2d}.
The independent reference is computed by integrating the stationary ODE from Section~\ref{subsec:stationary-wb} using the adaptive eighth-order Dormand--Prince method (DOP853), with relative and absolute tolerances $10^{-12}$ and $10^{-14}$, respectively.

\begin{figure}[H]
\centering
\includegraphics[width=0.48\linewidth]{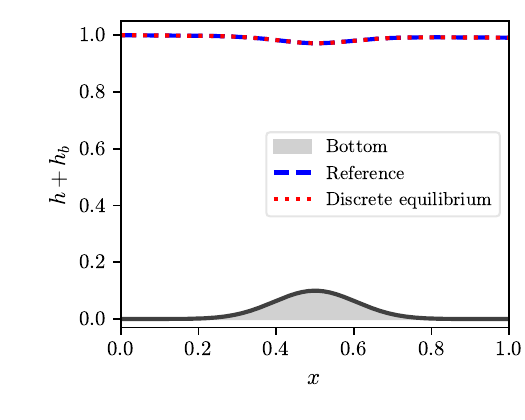}\hfill
\includegraphics[width=0.48\linewidth]{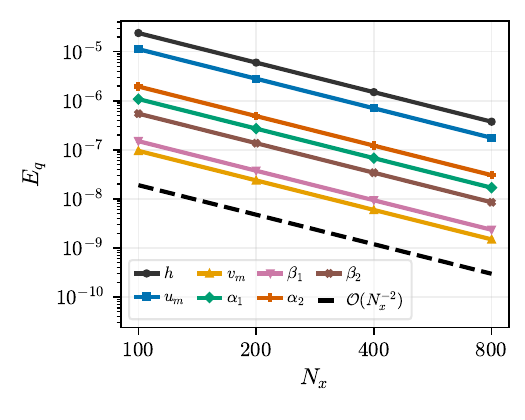}
\caption{Dissipative moving equilibrium: stationary free-surface profiles on the $N_x=800$ mesh (left) and componentwise equilibrium-construction errors $E_q$ under mesh refinement (right).
The bottom, reference surface, and discrete-equilibrium surface use the same conventions as Figure~\ref{fig:frictionless-equilibrium}.
The error panel shows all seven primitive components.
The black dashed line is a vertically offset $\mathcal O(N_x^{-2})$ guide.}
\label{fig:dissipative-equilibrium}
\end{figure}

Figure~\ref{fig:dissipative-equilibrium} shows second-order convergence in all seven primitive components.
The reference remains fully wet and noncritical, and a repeat with relaxed solver tolerances changes it by substantially less than the reported construction errors.

\begin{figure}[H]
\centering
\includegraphics[width=0.5\linewidth]{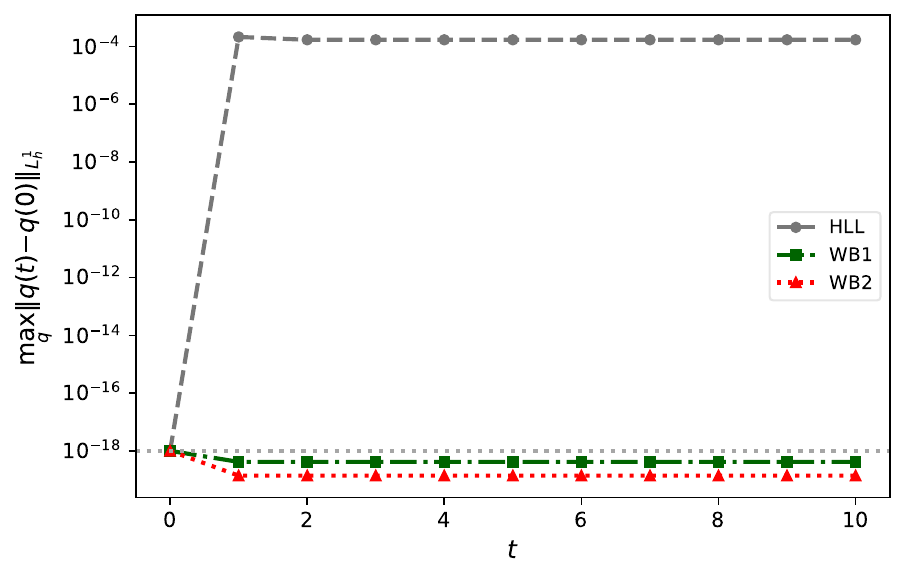}
\caption{Dissipative moving equilibrium: temporal history of the maximum componentwise drift $D_{\max}(t)$ on the $N_x=800$ mesh.
Exact zeros are displayed at $10^{-18}$ only to make them visible on the logarithmic axis.}
\label{fig:dissipative-drift}
\end{figure}

\begin{table}[H]
\centering
\begin{tabular}{rccc}
\toprule
$N_x$ & HLL & WB1 & WB2 \\
\midrule
100 & $1.377\mathrm{E}{-3}$ & $0$ & $0$ \\
200 & $6.871\mathrm{E}{-4}$ & $0$ & $0$ \\
400 & $3.431\mathrm{E}{-4}$ & $2.776\mathrm{E}{-19}$ & $0$ \\
800 & $1.714\mathrm{E}{-4}$ & $4.163\mathrm{E}{-19}$ & $1.388\mathrm{E}{-19}$ \\
\bottomrule
\end{tabular}
\caption{Dissipative moving equilibrium: final maximum componentwise $L_h^1$ drift $D_{\max}(10)$.}
\label{tab:dissipative-equilibrium-drift}
\end{table}

Figure~\ref{fig:dissipative-drift} and Table~\ref{tab:dissipative-equilibrium-drift} confirm that WB1 and WB2 also preserve the dissipative equilibrium to machine precision, while the HLL drift decreases at approximately first order under mesh refinement.
Together, the two cases verify second-order construction and machine-precision preservation of the discrete quasi-two-dimensional moving equilibria.

\subsection{Perturbation accuracy}
\label{subsec:perturbation-accuracy-test}
\label{ex:perturbation-accuracy}

We perturb the dissipative equilibrium of Section~\ref{subsec:dissipative-equilibrium}, retaining its domain, bottom topography, and physical parameters.
The initial depth is
\begin{equation}\label{eq:perturbation-accuracy-initial-data}
h(x,0)=h^*(x)+A\exp\!\left[-\frac{(x-0.5)^2}{2(0.025)^2}\right],
\end{equation}
with $A=0.05$ and $0.005$, while $u_m$, $v_m$, $\alpha_j$, and $\beta_j$ retain their equilibrium values.
The two amplitudes test accuracy away from equilibrium and the resolution of weak perturbations, respectively.

The continuous background $U^*$ is computed by DOP853 with the tolerances specified in Section~\ref{subsec:dissipative-equilibrium}.
Initial conservative cell averages are obtained by four-point Gauss--Legendre quadrature on $N_x=6400$ cells and block-averaged onto $N_x=100,200,400,800$, giving common initial data for all methods on each mesh.
The stored discrete equilibrium supplies the deviation reconstruction and fixed boundary states.

We use HLL, WB1, and WB2 with CFL number $0.25$ and final time $T=1$.
The HLL--WB1 comparison isolates the effect of well balancing at first order; WB2 additionally uses second-order reconstruction and time integration.
A WB2 solution on $N_x=6400$ cells serves as the numerical reference, with a separate $N_x=3200$ calculation and a half-CFL repeat on $N_x=800$ used to assess reference-grid and time-step sensitivity, respectively.

Let $U_h$ denote the numerical solution on the comparison mesh.
For each primitive component $q\in\{h,u_m,v_m,\alpha_1,\beta_1,\alpha_2,\beta_2\}$, define at $t=1$
\begin{equation}\label{eq:perturbation-accuracy-diagnostics}
\begin{aligned}
E_q
&=\left\|q\!\left(U_h(1)\right)-q\!\left(\mathcal{R}_h U_{\mathrm{ref}}(1)\right)\right\|_{L_h^1},
\\
S_q
&=\left\|q\!\left(\mathcal{R}_h U_{\mathrm{ref}}(1)\right)-q\!\left(\mathcal{R}_h U^*\right)\right\|_{L_h^1},
\qquad
\rho_q=\frac{E_q}{S_q}
\end{aligned}
\end{equation}
Here $\mathcal R_h$ denotes conservative block averaging onto the comparison mesh, $U_{\mathrm{ref}}$ is the numerical reference, and $S_q$ measures its departure from the continuous background.
The ratio $\rho_q$ is defined for $S_q\ne0$.
Reference-grid sensitivity is non-negligible for some moment components and is reported below together with time-step sensitivity.

\subsubsection{Finite-amplitude perturbation}
\label{subsec:finite-amplitude-perturbation}

We first set $A=0.05$ to examine the accuracy of the three methods away from equilibrium.
Figure~\ref{fig:perturbation-accuracy-a0p05} shows the final water depth perturbation and its mesh-refinement errors.

\begin{figure}[H]
\centering
\includegraphics[width=\linewidth]{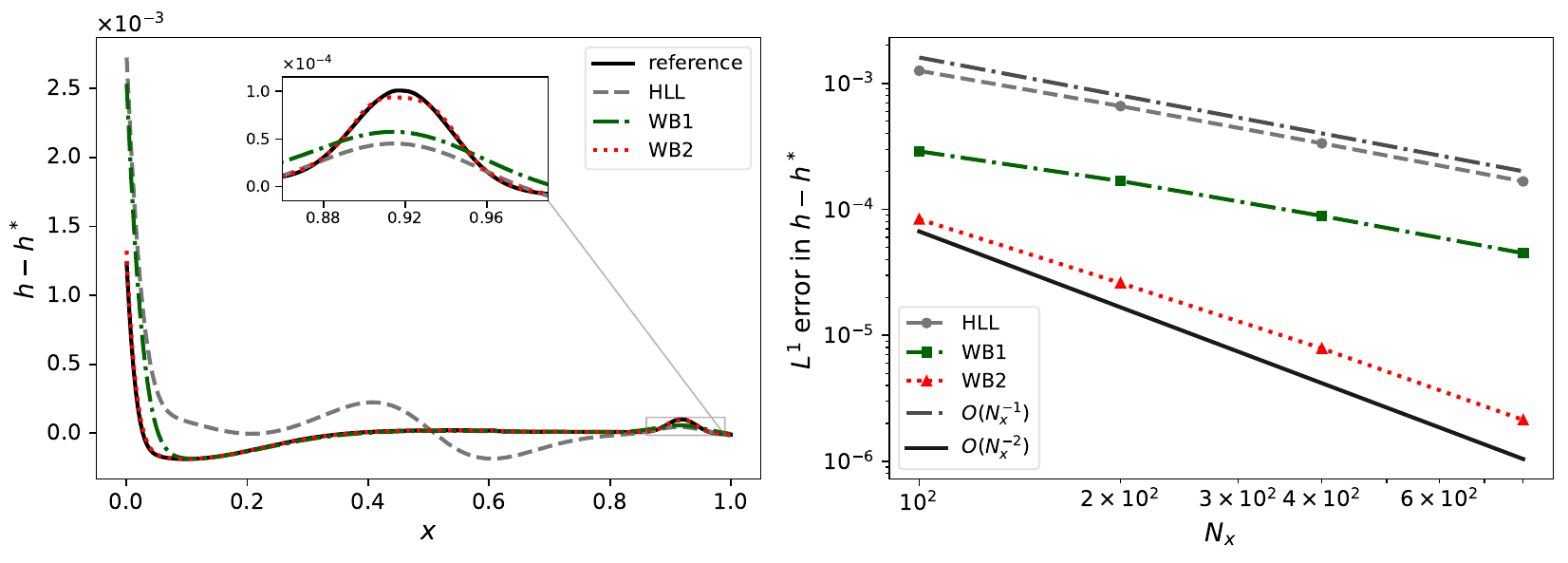}
\caption{Finite-amplitude perturbation, $A=0.05$: final depth perturbation on the $N_x=800$ mesh (left) and depth-component $L_h^1$ errors under mesh refinement (right).
The inset magnifies the depth perturbation on $0.86\le x\le0.99$.
The gray dash-dot and black solid lines are first- and second-order guides, respectively.
All errors use the restricted $N_x=6400$ WB2 solution as the numerical reference.}
\label{fig:perturbation-accuracy-a0p05}
\end{figure}

Figure~\ref{fig:perturbation-accuracy-a0p05} shows that WB2 closely reproduces the reference depth perturbation, whereas HLL introduces a pronounced spurious disturbance.
The inset reveals that HLL and WB1 broaden the bump near $x=0.92$ and underestimate its peak, while WB2 better captures its amplitude and shape.
Under mesh refinement, HLL and WB1 approach first-order convergence, with substantially smaller errors for WB1.
WB2 yields the smallest depth errors and approaches second-order convergence.

\subsubsection{Weak perturbation}
\label{subsec:weak-perturbation}

We next set $A=0.005$ to assess the resolution of a weak perturbation under the same numerical settings.

\begin{figure}[H]
\centering
\includegraphics[width=\linewidth]{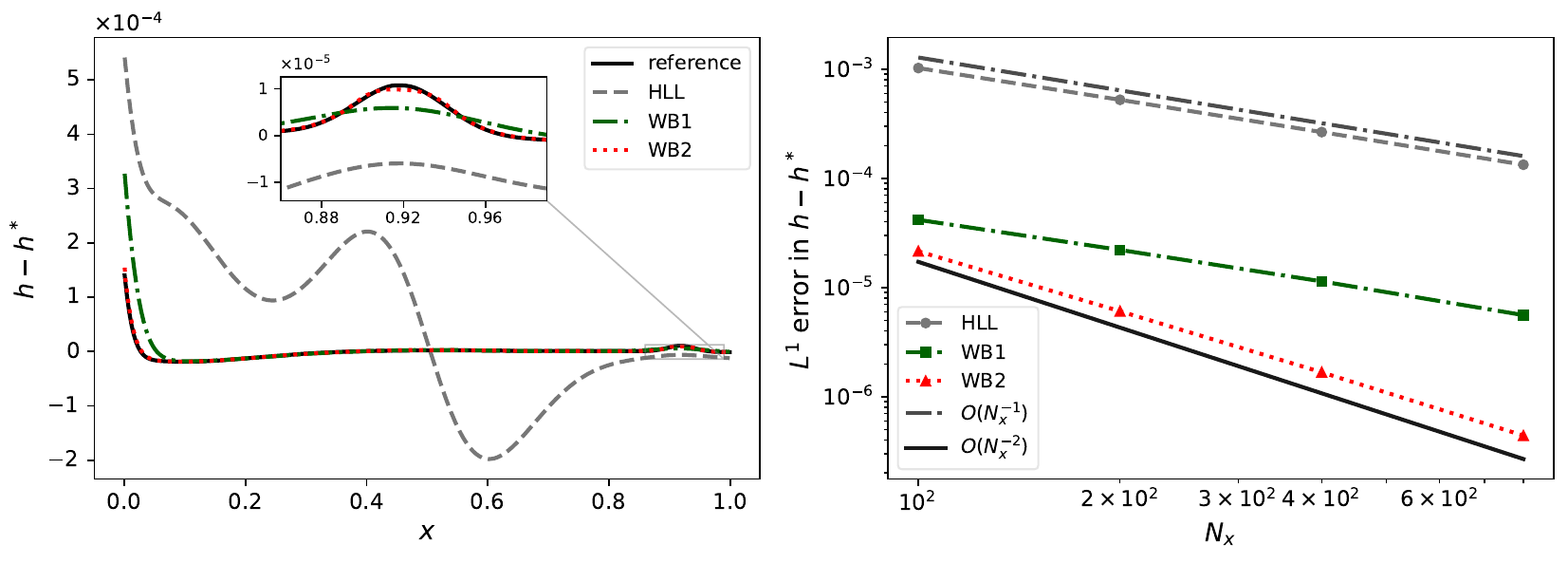}
\caption{Weak perturbation, $A=0.005$: final depth perturbation on the $N_x=800$ mesh (left) and depth-component $L_h^1$ errors under mesh refinement (right).
The inset magnifies the depth perturbation on $0.86\le x\le0.99$.
The reference and convergence-guide conventions are those of Figure~\ref{fig:perturbation-accuracy-a0p05}.
The HLL error remains comparable with the finite-amplitude case and dominates the weaker physical signal.}
\label{fig:perturbation-accuracy-a0p005}
\end{figure}

Figure~\ref{fig:perturbation-accuracy-a0p005} shows the greater importance of well balancing as the perturbation amplitude decreases.
The HLL disturbance dominates the weak physical signal, whereas WB1 and WB2 resolve the perturbation.
The inset shows that WB1 underestimates the bump's peak, while WB2 remains close to the reference.
HLL and WB1 again approach first-order convergence, but the error reduction from HLL to WB1 is more pronounced than in the finite-amplitude case.
WB2 gives the smallest depth errors and approaches second-order convergence.

\subsection{Radial collapse over a perfect-slip bed}
\label{ex:openfoam-collapse}

We compare frictionless moment-model solutions with three-dimensional volume-of-fluid OpenFOAM simulations for collapse from rest and collapse with initial vertical shear.
The tests examine the invariant SWE limit and the representation of vertical shear, respectively; differences between the solutions reflect both modeling assumptions and numerical errors.
Lengths, times, and velocities are expressed in metres, seconds, and metres per second, respectively; $\widetilde u$ and $\widetilde v$ also denote dimensional velocities.

\subsubsection{Collapse from rest: the invariant SWE limit}
\label{ex:openfoam-swe-limit}

We consider collapse from rest on $[0,100]^2\,\mathrm{m}^2$ over a flat bed at $z=0$, with an atmospheric boundary at $z=2\,\mathrm{m}$ and gravity $g=9.81\,\mathrm{m/s^2}$.
Writing $r=[(x-50)^2+(y-50)^2]^{1/2}$, the initial water depth is
\begin{equation}
\label{eq:openfoam-collapse-initial-depth}
h(x,y,0)=
\begin{cases}
1.5\,\mathrm{m}, & r\leq15\,\mathrm{m},\\
1.0\,\mathrm{m}, & r>15\,\mathrm{m},
\end{cases}
\end{equation}
All velocities and moment coefficients initially vanish, and solutions are compared at $t=1,2,3\,\mathrm{s}$.
The initially zero moments remain zero in the frictionless $G\text{-}\mathrm{SWLME}$, so orders $N=0,1,2$ share the SWE solution.

OpenFOAM uses a slip bed, and no-slip lateral walls.
The volume-of-fluid (VOF) method represents the water--air interface by transporting the water volume fraction $\alpha_{\mathrm{water}}$, defined as the fraction of each computational cell occupied by water.
Values $\alpha_{\mathrm{water}}=1$ and $0$ denote cells containing only water and only air, respectively, while intermediate values indicate cells containing both phases.
The horizontal mesh is $400^2$, with $N=200,400,800$ vertical cells for the mesh-sensitivity study; comparisons with the reduced models use $N=800$.
Momentum and phase-fraction advection use linear-upwind and van Leer schemes, respectively, with unit interface compression.
Adaptive first-order Euler steps start from $\Delta t=10^{-3}\,\mathrm{s}$ and limit both flow and phase-fraction Courant numbers to one.

The reduced models use a $400^2$ finite-volume mesh, monotonized-central MUSCL reconstruction, a local Lax--Friedrichs flux, SSP-RK2, CFL number $0.4$, and impermeable free-slip lateral boundaries.

With $\widehat{\alpha}_{\mathrm{water}}=\min\{1,\max\{0,\alpha_{\mathrm{water}}\}\}$, the OpenFOAM water depth and depth-averaged velocities are
\begin{equation}
\label{eq:openfoam-projected-observables}
\begin{aligned}
h_{\alpha}
&=\int_0^{2\,\mathrm{m}}\widehat{\alpha}_{\mathrm{water}}\,\md z,\\
u_m^{\mathrm{OF}}
&=\frac{q_x}{h_{\alpha}},
\qquad q_x=\int_0^{2\,\mathrm{m}}\widehat{\alpha}_{\mathrm{water}}u^{\mathrm{OF}}\,\md z,\\
v_m^{\mathrm{OF}}
&=\frac{q_y}{h_{\alpha}},
\qquad q_y=\int_0^{2\,\mathrm{m}}\widehat{\alpha}_{\mathrm{water}}v^{\mathrm{OF}}\,\md z.
\end{aligned}
\end{equation}
Here $u^{\mathrm{OF}}$ and $v^{\mathrm{OF}}$ are the horizontal velocity components, and $h_\alpha$ is the water volume per unit horizontal area.
Unclipped projections for $N=800$ give only small changes in these quantities.
The horizontal profiles are evaluated on the common cell-center line $y=49.875\,\mathrm{m}$.

\begin{figure}[H]
\centering
\includegraphics[width=0.96\linewidth]{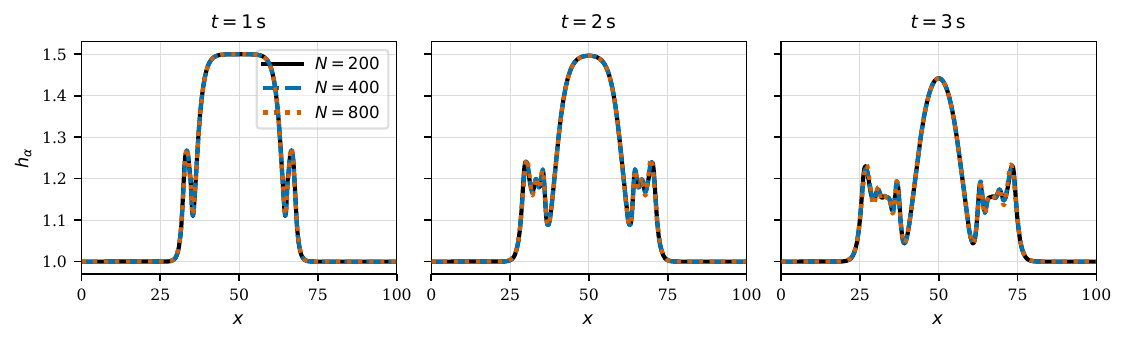}
\caption{OpenFOAM physical-vertical resolution study for the projected water depth on the horizontal line $y=49.875\,\mathrm{m}$.
The panels show $h_\alpha$ for $N=200,400,800$ at $t=1,2,3\,\mathrm{s}$.}
\label{fig:openfoam-height-resolution}
\end{figure}

Figure~\ref{fig:openfoam-height-resolution} shows nearly coincident water depth profiles and principal wave locations for all three OpenFOAM meshes, while local oscillation amplitudes retain some vertical-mesh sensitivity.
At the initial water depth ratio $h_{\mathrm{out}}/h_{\mathrm{in}}=2/3$, the observed wave trains are consistent with undular-bore structures associated with nonhydrostatic and dispersive effects in wet-bed dam-break flows \cite{Kim2011,Cantero-Chinchilla2020}.
The oscillations may also contain numerical contributions: studies of volume-of-fluid solvers report free-surface oscillations whose amplitudes depend on spatial and temporal resolution and interface treatment \cite{Larsen2019,Ferro2022}.

\begin{figure}[H]
\centering
\includegraphics[width=0.96\linewidth]{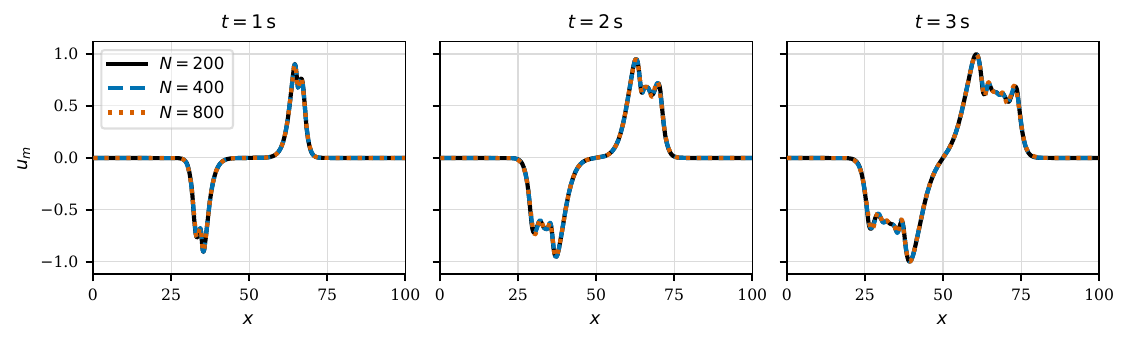}
\caption{OpenFOAM physical-vertical resolution study for the projected $x$-directed depth-averaged velocity on the horizontal line $y=49.875\,\mathrm{m}$.
The panels show $u_m$ for $N=200,400,800$ at $t=1,2,3\,\mathrm{s}$.}
\label{fig:openfoam-um-resolution}
\end{figure}

Figure~\ref{fig:openfoam-um-resolution} likewise shows close agreement in the depth-averaged velocity pattern, with the largest differences near the local extrema.

\begin{figure}[H]
\centering
\includegraphics[width=0.96\linewidth]{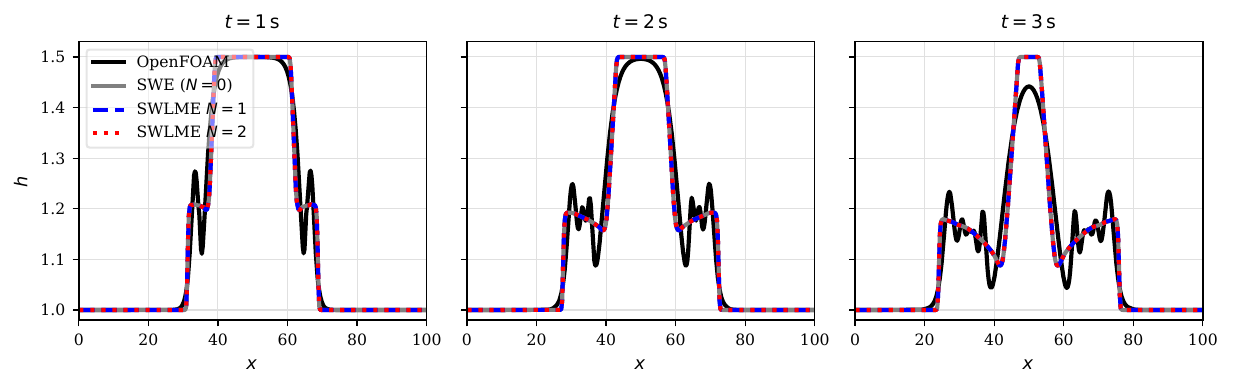}
\includegraphics[width=0.96\linewidth]{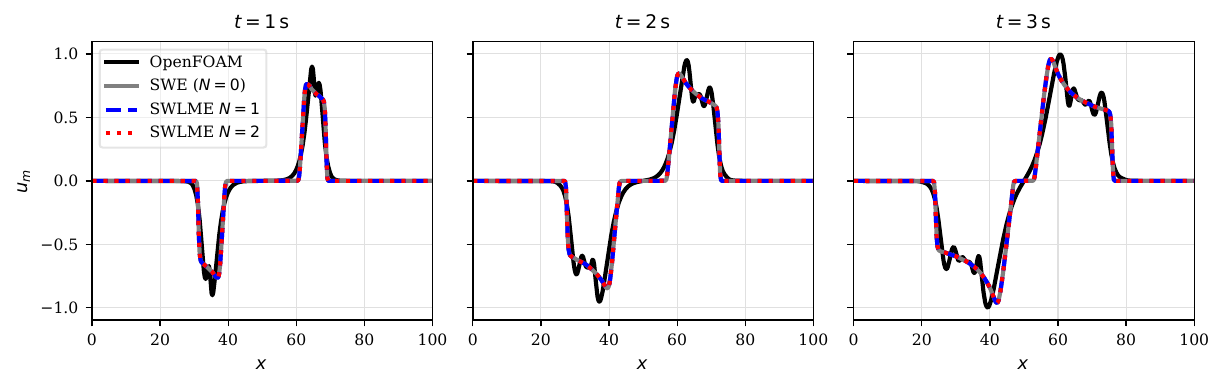}
\caption{Water depth (top) and $x$-directed depth-averaged velocity (bottom) on the horizontal line $y=49.875\,\mathrm{m}$.
The OpenFOAM $N_{\mathrm{vert}}=800$ projection is compared with the matched $400^2$ SWE and SWLME $N=1,2$ calculations at $t=1,2,3\,\mathrm{s}$.
The three reduced-model curves coincide because every moment remains zero.}
\label{fig:openfoam-collapse-slices}
\end{figure}

Figure~\ref{fig:openfoam-collapse-slices} shows that the reduced equations capture the principal inward and outward wave locations and the large-scale depth-averaged velocity pattern of the OpenFOAM calculation.
The three reduced-model curves coincide because the initially vanishing moments remain zero.

Vertical profiles of the horizontal velocity are compared at $(x,y)=(64.875,49.875)\,\mathrm{m}$, near the initial water depth discontinuity.
To exclude mixed interface cells and numerical overshoots, we retain OpenFOAM samples whose original, unclipped water fraction equals one exactly in the stored six-significant-digit data.
This restriction gives a comparison within the water interior, excluding the pointwise free-surface velocity.
The reduced profiles use $\util=u_m+\sum_{j=1}^N\alpha_j\phi_j$.

\begin{figure}[H]
\centering
\includegraphics[width=0.96\linewidth]{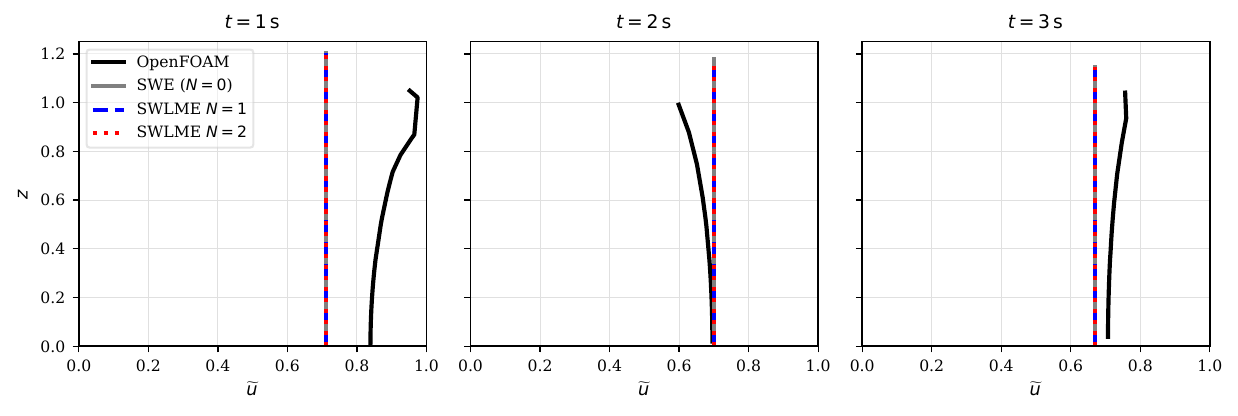}
\caption{Vertical profiles of the horizontal velocity $\util(z)$ at $(x,y)=(64.875,49.875)\,\mathrm{m}$.
The OpenFOAM curve connects cell centers with original, unclipped stored $\alpha_{\mathrm{water}}=1$ exactly for $N_{\mathrm{vert}}=800$; its endpoints are not interpreted as pointwise free-surface velocities.
The SWE and SWLME $N=1,2$ curves are reconstructed from their matched $400^2$ states.}
\label{fig:openfoam-collapse-vertical}
\end{figure}

Figure~\ref{fig:openfoam-collapse-vertical} shows vertical variation in the selected OpenFOAM velocity samples, whereas all three reduced profiles are plug profiles because the moments vanish.
At $t=1,2,3\,\mathrm{s}$, the common reduced-model values of $u_m$ are $0.712$, $0.701$, and $0.671\,\mathrm{m/s}$, while the corresponding projected OpenFOAM depth-averaged velocities are $0.884$, $0.624$, and $0.718\,\mathrm{m/s}$.
The perfect-slip collapse verifies the invariant SWE limit of the $G\text{-}\mathrm{SWLME}$ and provides a bulk-flow comparison.

\subsubsection{Collapse with initial vertical shear}
\label{ex:openfoam-initial-shear}

We retain the domain, initial water depth, gravity, fluid properties, and boundary conditions of Section~\ref{ex:openfoam-swe-limit}, and impose a localized rotational flow with nonzero moments to assess the representation of vertical shear as the moment order increases.
Let $X=x-50\,\mathrm{m}$, $Y=y-50\,\mathrm{m}$, and $r=(X^2+Y^2)^{1/2}$.
The initial depth-averaged velocity is
\begin{equation}
\label{eq:openfoam-shear-mean}
\begin{aligned}
 B(r)&=\left[\max\left\{1-\frac{r^2}{R^2},0\right\}\right]^3,
 & R&=12\,\mathrm{m},\\
 (u_{m,0},v_{m,0})&=\Omega B(r)(-Y,X),
 & \Omega&=\frac{U_{\max}\sqrt{7}}{R}\left(\frac76\right)^3,
 \qquad U_{\max}=0.4\,\mathrm{m/s}.
\end{aligned}
\end{equation}
This counterclockwise flow has maximum speed $U_{\max}$ and vanishes for $r\ge R$.
For $h_0=h(x,y,0)$ and $\zeta=z/h_0$, the initial water velocity is
\begin{equation}
\label{eq:openfoam-shear-profile}
\begin{aligned}
 f(\zeta)&=1+0.3\cos(\pi\zeta)+0.2\cos(2\pi\zeta),\\
 (\widetilde u_0,\widetilde v_0,w_0)&=
 \bigl(u_{m,0}f(\zeta),v_{m,0}f(\zeta),0\bigr).
\end{aligned}
\end{equation}
The function $f$ has unit depth average and satisfies $f'(0)=f'(1)=0$, and the initial water velocity is divergence-free.
The air velocity continues smoothly to zero at the atmospheric top, and pressure is initialized hydrostatically.
The reduced initial data are obtained by Legendre projection and conservative horizontal cell averaging, with moment coefficients
\begin{equation}
\label{eq:openfoam-shear-moments}
 (\alpha_{1,0},\beta_{1,0})=\frac{3.6}{\pi^2}(u_{m,0},v_{m,0}),
 \qquad
 (\alpha_{2,0},\beta_{2,0})=\frac{3}{\pi^2}(u_{m,0},v_{m,0}).
\end{equation}
Each model retains only the coefficients allowed by its order.
For the additional $N=3,4$ tests, the same initial profile gives
\begin{equation}
\label{eq:openfoam-shear-third-moment}
\begin{aligned}
 (\alpha_{3,0},\beta_{3,0})&=\frac{50.4(\pi^2-10)}{\pi^4}(u_{m,0},v_{m,0}),\\
 (\alpha_{4,0},\beta_{4,0})&=\left(\frac{18}{\pi^2}-\frac{189}{\pi^4}\right)(u_{m,0},v_{m,0}).
\end{aligned}
\end{equation}
The physical-profile comparison additionally includes $N=3,4$ on the common $400^2$ horizontal grid; the depth, mean-velocity, and moment-coefficient comparisons retain $N=0,1,2$.
The meshes, time integrators, and comparison times follow Section~\ref{ex:openfoam-swe-limit}.
The moment systems use generalized-minmod reconstruction, a path-conservative local Lax--Friedrichs discretization, and CFL number $0.3$.
For OpenFOAM, both Courant limits are $0.25$ and the maximum time step is $0.01\,\mathrm{s}$.

Water depth and depth-averaged velocities are computed from~\eqref{eq:openfoam-projected-observables}, and the OpenFOAM moment coefficients are
\begin{equation}
\label{eq:openfoam-shear-volume-coordinate}
\begin{aligned}
 \alpha_j&=\frac{2j+1}{h_\alpha}\int_0^{2\,\mathrm{m}}\widehat{\alpha}_{\mathrm{water}}(z)u(z)\phi_j(s(z))\,\md z,\\
 \beta_j&=\frac{2j+1}{h_\alpha}\int_0^{2\,\mathrm{m}}\widehat{\alpha}_{\mathrm{water}}(z)v(z)\phi_j(s(z))\,\md z,
\end{aligned}
\qquad
s(z)=\frac{1}{h_\alpha}\int_0^z\widehat{\alpha}_{\mathrm{water}}(z')\,\md z',
\quad j=1,2.
\end{equation}
Here $s(z)$ is the fraction of the total column water volume below height $z$; it equals $z/h_\alpha$ within a sharp water column extending from the bed.

\begin{figure}[H]
\centering
\includegraphics[width=0.85\linewidth]{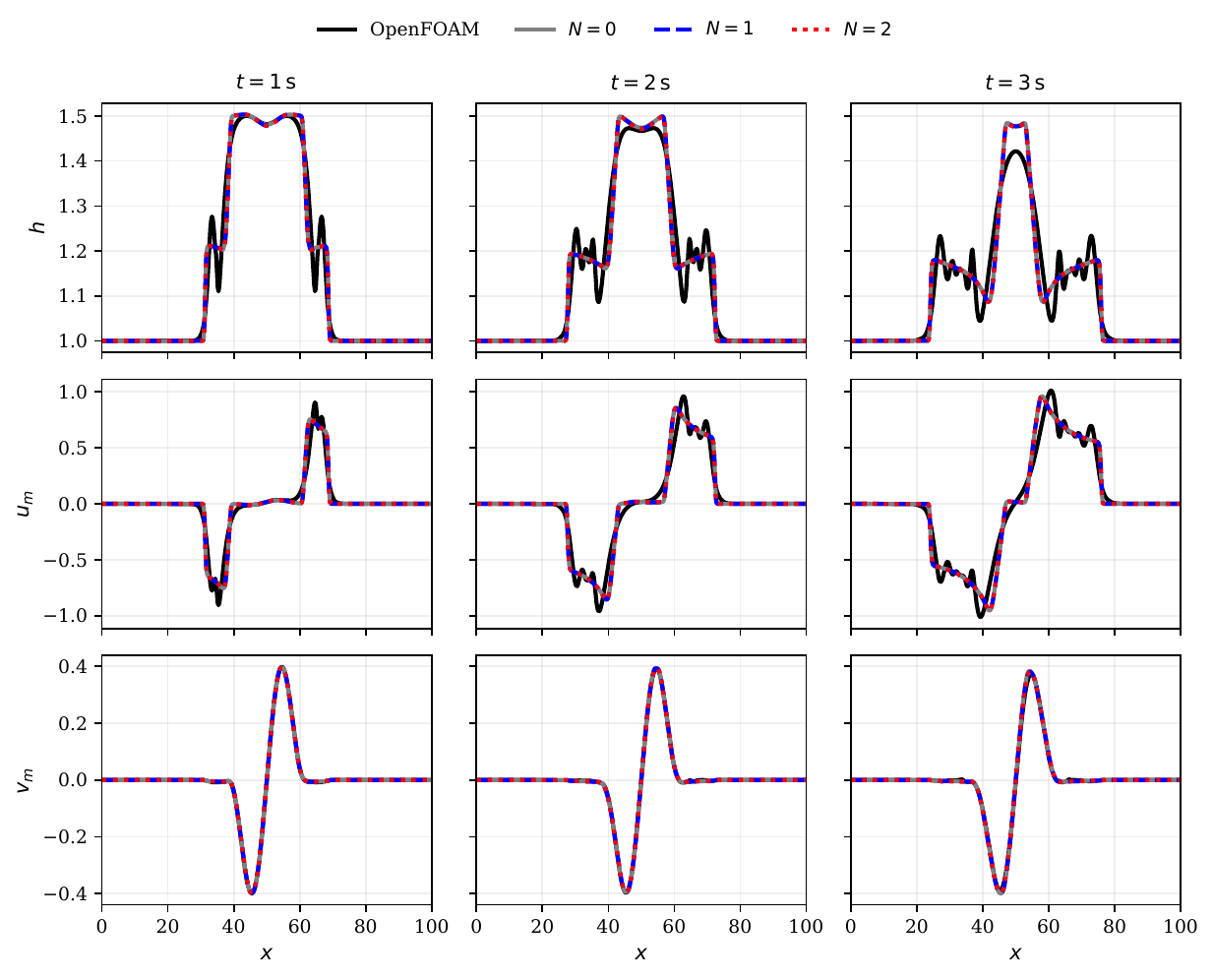}
\caption{Water depth and depth-averaged velocity comparisons on $y=49.875\,\mathrm{m}$ at $t=1,2,3\,\mathrm{s}$.
OpenFOAM uses 800 vertical cells; $N=0,1,2$ in the legend denotes the moment order on the common $400^2$ horizontal grid.}
\label{fig:openfoam-shear-slices}
\end{figure}

Figure~\ref{fig:openfoam-shear-slices} shows close agreement in the water depth $h$ and depth-averaged velocities $u_m$ and $v_m$, with nearly overlapping predictions for $N=0,1,2$.
On the near-center slice $y=49.875\,\mathrm{m}$, $u_m$ and $v_m$ approximately represent radial and azimuthal motion, respectively, up to a sign change across the center.
The initial circular jump in water depth drives radial gravity waves, directly coupling the water depth evolution to $u_m$ and producing their more complex wave structure.
The additional OpenFOAM oscillations are concentrated in these two quantities, consistent with the wave-associated structures and numerical sensitivity discussed for Figure~\ref{fig:openfoam-height-resolution}.
In contrast, $v_m$ primarily reflects the initially smooth, localized rotation, with only a small contribution from radial motion on this slice.
Its simpler profile is reproduced more closely by the reduced models, supporting their representation of the evolving mean rotational flow.
A plot of $v_m$ against $y$ along $x=49.875\,\mathrm{m}$ would exhibit radial-collapse behavior similar to that of $u_m$ against $x$ in Figure~\ref{fig:openfoam-shear-slices} and is therefore omitted.

\begin{figure}[H]
\centering
\includegraphics[width=0.85\linewidth]{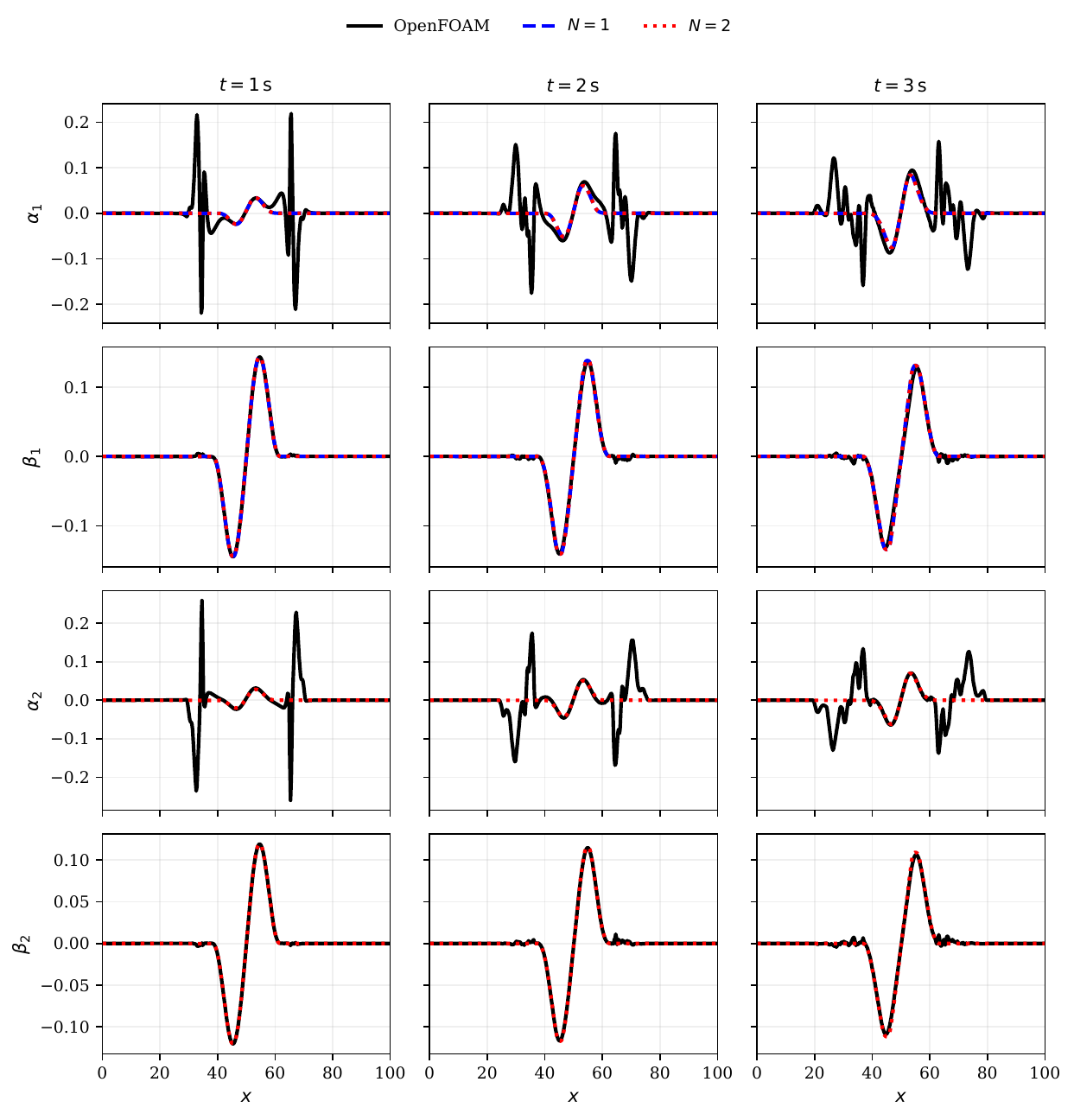}
\caption{First and second moment comparisons on the same horizontal line and at the same times as Figure~\ref{fig:openfoam-shear-slices}.
The OpenFOAM coefficients use the cumulative water-volume projection; a reduced-model curve is shown only when that moment is retained.}
\label{fig:openfoam-shear-moments}
\end{figure}

Figure~\ref{fig:openfoam-shear-moments} shows close agreement in $\alpha_1$ and $\alpha_2$ away from the collapse-wave regions, including the nonzero central profiles and the nearly zero moments in the undisturbed outer regions.
The largest differences occur near the propagating collapse waves, where OpenFOAM exhibits large variations and the reduced coefficients remain small.
On this near-center horizontal slice, $\alpha_j$ primarily describes vertical variations in radial motion, so its more complex profile is associated with the vertical structure of the collapse waves.
In the collapse-wave regions, increasing the moment order from $N=1$ to $N=2$ produces little change in $\alpha_1$ and does not substantially improve its agreement with OpenFOAM.
The remaining differences may reflect contributions from the hydrostatic moment approximation, numerical discretization, and water-volume projection.

The reduced models closely reproduce the amplitudes and shapes of $\beta_1$ and $\beta_2$ over the displayed times.
These coefficients primarily describe vertical variations in azimuthal motion on this slice, and their smoother profiles reflect the evolution of the initially smooth, localized rotational shear.
Their close agreement supports the representation of this rotational shear.
Plots of $\beta_j$ against $y$ along $x=49.875\,\mathrm{m}$ would exhibit behavior similar to that of $\alpha_j$ against $x$ and are therefore omitted.

\begin{figure}[H]
\centering
\includegraphics[width=0.96\linewidth]{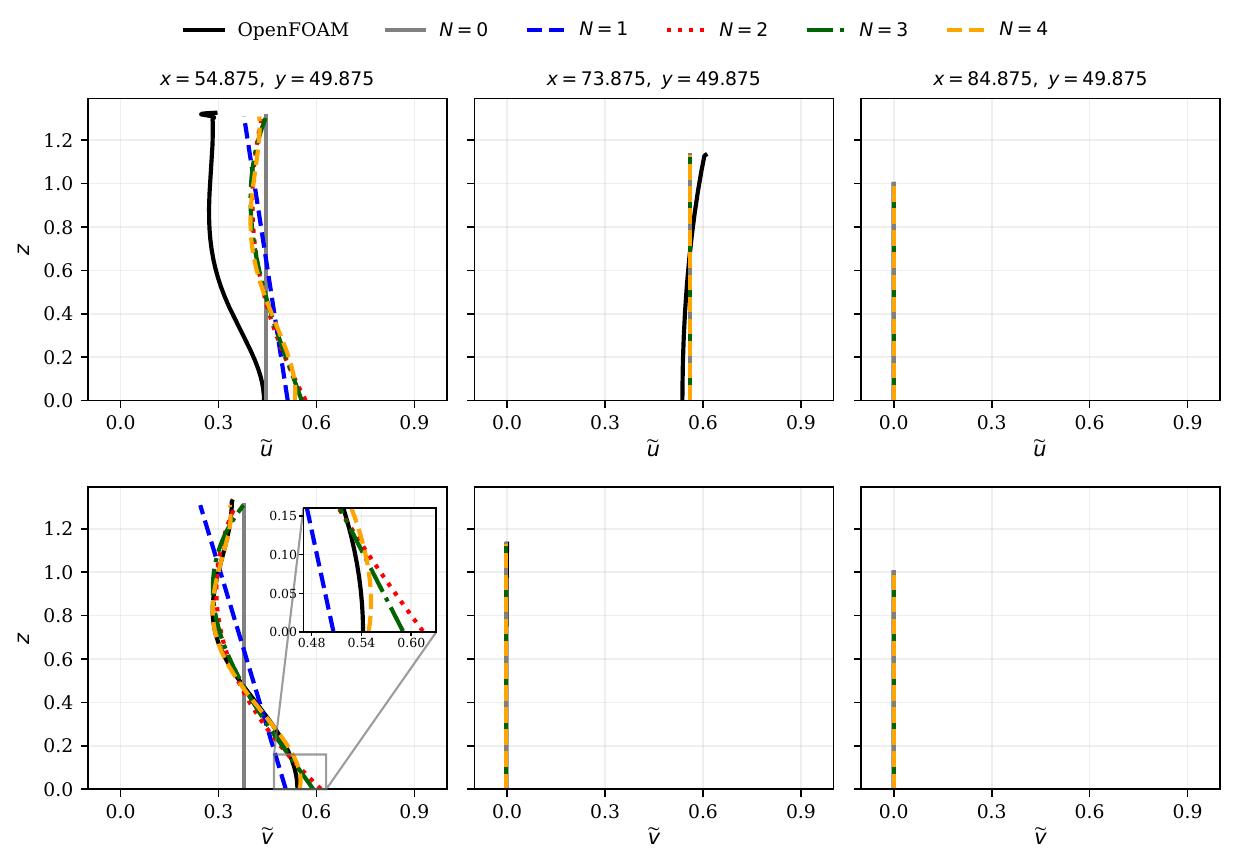}
\caption{Physical vertical profiles of both horizontal velocity components at $t=3\,\mathrm{s}$ and the indicated cell-center locations.
From left to right, the columns correspond to the central rotating region, the outward collapse wave, and the nearly undisturbed region ahead of it, respectively.
The reduced models use $N=0,1,2,3,4$ on the common $400^2$ horizontal grid, and the inset enlarges the near-bed $\widetilde v$ profiles in the left column.
OpenFOAM profiles retain consecutive cells from the bed up to, excluding, the first cell with $\alpha_{\mathrm{water}}<0.95$.
All reduced profiles are reconstructed at the retained OpenFOAM physical heights within their own water columns, evaluating the Legendre expansion at $\zeta=z/h$.
The lowest OpenFOAM sample is at $z=0.00125\,\mathrm{m}$; no extrapolation to the bed is used.}
\label{fig:openfoam-shear-profiles}
\end{figure}

To evaluate the effect of moment order on velocity profile reconstruction, we compare the $G\text{-}\mathrm{SWLME}$ at orders up to $N=4$ in Figure~\ref{fig:openfoam-shear-profiles}.
In the central rotating region (left column), the higher-order reconstructions capture the overall shapes of both velocity profiles, with a persistent offset in $\util$ relative to OpenFOAM and closer agreement in both shape and magnitude for $\vtil$.
The constant $(N=0)$ and linear $(N=1)$ reconstructions cannot represent the observed curvature. Additional basis functions improve the reconstruction, particularly the near bottom topography values of $\vtil$, as highlighted by the inset.
Within the outward collapse wave (middle column), the profiles for $N=0,1,2,3,4$ nearly coincide and capture the magnitude of $\widetilde u$, but not its modest vertical variation over the retained interval; $\widetilde v$ is nearly zero.
Ahead of the wave (right column), both horizontal velocities are nearly zero, providing an undisturbed reference.

Increasing the number of basis functions improves the representation of complex vertical profile, but does not replace an appropriate physical boundary condition. The ability to accommodate the slip is therefore an important part of the model formulation. In flows with thin boundary layers, sharp near-wall variations may require substantially more modes than the smooth interior profile.

\section{Conclusions}
We established rotational invariance of the two-dimensional SWLME, identified its loss of global hyperbolicity, and constructed the globally hyperbolic modification $G\text{-}\mathrm{SWLME}$.
For quasi-two-dimensional flows, we developed first- and second-order path-conservative finite-volume schemes that exactly preserve a prescribed discrete moving equilibrium, including Navier-slip friction and non-flat bottom topography.

Numerical tests confirm equilibrium preservation to roundoff, second-order approximation of stationary branches, and near-second-order accuracy for the tested perturbations using the second-order scheme.
The OpenFOAM comparisons illustrate the invariant SWE limit and improved rotational velocity profiles with nonzero moments, while differences remain in radial motion and full reference convergence remains unverified.
Future work will address general two-dimensional moving equilibria and further validation of reconstructed velocity profiles.

\section*{Data and code availability}
The data-generation source code, including the numerical solvers and OpenFOAM case configurations, is publicly available at \url{https://github.com/zhou-shiping/swe-well-balanced-2026} and archived on Zenodo as release v1.0~\cite{Zhou2026Code}.
The repository includes software dependencies and instructions for running the simulations locally.

\section*{Acknowledgements}
The authors acknowledge support from ONR grant N00014-24-1-2242, NSF grant DMS-2618114, AFOSR grants FA9550-24-1-0254, DOE grant DE-SC0023164, and Ralph E. Powe Junior Faculty Enhancement Award from Oak Ridge Associated Universities (ORAU).

\appendix
\section{Proof of rotational invariance}
\label{theorem:proof}

We prove Theorem~\ref{thm:rot-SWLME} using the conservative/non-conservative splitting of \cite{Bauerle2025}, adapted to the linearized moment system.

Define the rotated variables
\begin{equation*}
\begin{aligned}
    u_\theta &= \cos\theta\,u_m+\sin\theta\,v_m,
    &v_\theta &= -\sin\theta\,u_m+\cos\theta\,v_m,\\
    \alpha_{j,\theta} &= \cos\theta\,\alpha_j+\sin\theta\,\beta_j,
    &\beta_{j,\theta} &= -\sin\theta\,\alpha_j+\cos\theta\,\beta_j,
    \qquad j=1,\ldots,N.
\end{aligned}
\end{equation*}
Then
\begin{equation*}
    T(\theta)U
    =(h,hu_\theta,hv_\theta,h\alpha_{1,\theta},h\beta_{1,\theta},\ldots,h\alpha_{N,\theta},h\beta_{N,\theta})^T.
\end{equation*}
We use the identities
\begin{align*}
    \cos\theta\,u_\theta-\sin\theta\,v_\theta &= u_m,
    &\sin\theta\,u_\theta+\cos\theta\,v_\theta &= v_m, \\
    \cos\theta\,\alpha_{j,\theta}-\sin\theta\,\beta_{j,\theta} &= \alpha_j,
    &\sin\theta\,\alpha_{j,\theta}+\cos\theta\,\beta_{j,\theta} &= \beta_j.
\end{align*}
Moreover,
\begin{align}
    \cos\theta\,u_\theta^2-\sin\theta\,u_\theta v_\theta
    &=u_m u_\theta, \label{eq:rot-id-u2}\\
    \sin\theta\,u_\theta^2+\cos\theta\,u_\theta v_\theta
    &=v_m u_\theta, \notag\\
    2\cos\theta\,u_\theta\alpha_{i,\theta}
    -\sin\theta\,(u_\theta\beta_{i,\theta}+v_\theta\alpha_{i,\theta})
    &=2\cos\theta\,u_m\alpha_i
    +\sin\theta\,(u_m\beta_i+v_m\alpha_i), \label{eq:rot-id-moment-x}\\
    2\sin\theta\,u_\theta\alpha_{i,\theta}
    +\cos\theta\,(u_\theta\beta_{i,\theta}+v_\theta\alpha_{i,\theta})
    &=\cos\theta\,(u_m\beta_i+v_m\alpha_i)
    +2\sin\theta\,v_m\beta_i. \label{eq:rot-id-moment-y}
\end{align}
The same identities also hold after replacing $(u_m,v_m)$ by $(\alpha_j,\beta_j)$.

\begin{lemma}[rotational invariance of the conservative fluxes]
\label{lem:rot-FL-GL}
The conservative fluxes $F_L(U)$ and $G_L(U)$ of the SWLME satisfy
\begin{equation}\label{eq:rot-FL-GL}
    \cos\theta\,F_L(U)+\sin\theta\,G_L(U)
    =T(\theta)^{-1}F_L(T(\theta)U),
\end{equation}
for all $\theta\in[0,2\pi)$ and all admissible states $U$.
\end{lemma}

\begin{proof}
We verify the identity component by component.
The first component is
\begin{equation*}
    h u_\theta
    =h(\cos\theta\,u_m+\sin\theta\,v_m)
    =\cos\theta\,(hu_m)+\sin\theta\,(hv_m),
\end{equation*}
which is the first component of the left-hand side of \eqref{eq:rot-FL-GL}.

For the momentum components, the second component of $T(\theta)^{-1}F_L(T(\theta)U)$ is
\begin{align*}
&\cos\theta\left[h\left(u_\theta^2+
\sum_{j=1}^N\frac{\alpha_{j,\theta}^2}{2j+1}\right)+\frac12Gh^2\right]
-\sin\theta\left[h\left(u_\theta v_\theta+
\sum_{j=1}^N\frac{\alpha_{j,\theta}\beta_{j,\theta}}{2j+1}\right)\right] \\
&\quad
=\cos\theta\left[h\left(u_m^2+
\sum_{j=1}^N\frac{\alpha_j^2}{2j+1}\right)+\frac12Gh^2\right]
+\sin\theta\left[h\left(u_m v_m+
\sum_{j=1}^N\frac{\alpha_j\beta_j}{2j+1}\right)\right],
\end{align*}
where we used \eqref{eq:rot-id-u2} and the corresponding identity for $(\alpha_j,\beta_j)$.
This equals the second component of $\cos\theta F_L(U)+\sin\theta G_L(U)$.
Similarly, the third component is
\begin{align*}
&\sin\theta\left[h\left(u_\theta^2+
\sum_{j=1}^N\frac{\alpha_{j,\theta}^2}{2j+1}\right)+\frac12Gh^2\right]
+\cos\theta\left[h\left(u_\theta v_\theta+
\sum_{j=1}^N\frac{\alpha_{j,\theta}\beta_{j,\theta}}{2j+1}\right)\right] \\
&\quad
=\cos\theta\left[h\left(u_m v_m+
\sum_{j=1}^N\frac{\alpha_j\beta_j}{2j+1}\right)\right]
+\sin\theta\left[h\left(v_m^2+
\sum_{j=1}^N\frac{\beta_j^2}{2j+1}\right)+\frac12Gh^2\right].
\end{align*}

For each moment pair, the corresponding two components of $F_L(T(\theta)U)$ are
\begin{equation*}
    h(2u_\theta\alpha_{i,\theta}),
    \qquad
    h(u_\theta\beta_{i,\theta}+v_\theta\alpha_{i,\theta}).
\end{equation*}
After multiplication by $T_2(\theta)^{-1}$, the first component of this pair is
\begin{align*}
&h\left[2\cos\theta\,u_\theta\alpha_{i,\theta}
-\sin\theta\,(u_\theta\beta_{i,\theta}+v_\theta\alpha_{i,\theta})\right] \\
&\quad
=\cos\theta\,h(2u_m\alpha_i)
+\sin\theta\,h(u_m\beta_i+v_m\alpha_i),
\end{align*}
by \eqref{eq:rot-id-moment-x}.
The second component is
\begin{align*}
&h\left[2\sin\theta\,u_\theta\alpha_{i,\theta}
+\cos\theta\,(u_\theta\beta_{i,\theta}+v_\theta\alpha_{i,\theta})\right] \\
&\quad
=\cos\theta\,h(u_m\beta_i+v_m\alpha_i)
+\sin\theta\,h(2v_m\beta_i),
\end{align*}
by \eqref{eq:rot-id-moment-y}.
These are exactly the corresponding components of $\cos\theta F_L(U)+\sin\theta G_L(U)$.
Since $i=1,\ldots,N$ is arbitrary, \eqref{eq:rot-FL-GL} follows.
\end{proof}

\begin{lemma}[rotational invariance of the non-conservative part]
\label{lem:rot-P1-Q1}
The non-conservative matrices $P_1(U)$ and $Q_1(U)$ of the SWLME satisfy
\begin{equation}\label{eq:rot-P1-Q1}
    \cos\theta\,P_1(U)+\sin\theta\,Q_1(U)
    =T(\theta)^{-1}P_1(T(\theta)U)T(\theta),
\end{equation}
for all $\theta\in[0,2\pi)$ and all admissible states $U$.
\end{lemma}

\begin{proof}
The matrices $P_1(U)$ and $Q_1(U)$ are block diagonal with the same $2\times2$ block repeated along the moment variables.
Thus, it is enough to verify the identity for the block pair
\begin{equation*}
    p(U)=\begin{pmatrix}-u_m&0\\-v_m&0\end{pmatrix},
    \qquad
    q(U)=\begin{pmatrix}0&-u_m\\0&-v_m\end{pmatrix}.
\end{equation*}
At the rotated state $T(\theta)U$, this block becomes
\begin{equation*}
    p(T(\theta)U)=
    \begin{pmatrix}-u_\theta&0\\-v_\theta&0\end{pmatrix}.
\end{equation*}
A direct calculation gives
\begin{align*}
T_2(\theta)^{-1}p(T(\theta)U)T_2(\theta)
&=
\begin{pmatrix}
    -\cos\theta\,u_m & -\sin\theta\,u_m\\
    -\cos\theta\,v_m & -\sin\theta\,v_m
\end{pmatrix} \\
&=\cos\theta\,p(U)+\sin\theta\,q(U).
\end{align*}
Since the same block relation holds for every moment pair and $T(\theta)$ is block diagonal, assembling the block identities gives \eqref{eq:rot-P1-Q1}.
\end{proof}

\begin{proof}[Proof of Theorem~\ref{thm:rot-SWLME}]
By Lemma~\ref{lem:rot-FL-GL},
\begin{equation*}
    \cos\theta\,F_L(U)+\sin\theta\,G_L(U)
    =T(\theta)^{-1}F_L(T(\theta)U).
\end{equation*}
Taking the derivative with respect to $U$ gives
\begin{equation*}
    \cos\theta\,\partial_UF_L(U)+\sin\theta\,\partial_UG_L(U)
    =T(\theta)^{-1}\partial_UF_L(T(\theta)U)T(\theta),
\end{equation*}
because $T(\theta)$ is independent of $U$.
Adding the non-conservative identity \eqref{eq:rot-P1-Q1} yields
\begin{align*}
&\cos\theta\,[\partial_UF_L(U)+P_1(U)]
+\sin\theta\,[\partial_UG_L(U)+Q_1(U)] \\
&\quad
=T(\theta)^{-1}[\partial_UF_L(T(\theta)U)+P_1(T(\theta)U)]T(\theta).
\end{align*}
Using \eqref{eq:AL-BL-def}, we obtain
\begin{equation*}
    \cos\theta\,A_L(U)+\sin\theta\,B_L(U)
    =T(\theta)^{-1}A_L(T(\theta)U)T(\theta),
\end{equation*}
which proves Theorem~\ref{thm:rot-SWLME}.
\end{proof}

\bibliographystyle{abbrv}
\bibliography{SWE-well-balanced}

@article{Audusse2004,
  title = {A Fast and Stable Well-Balanced Scheme with Hydrostatic Reconstruction for Shallow Water Flows},
  author = {Audusse, Emmanuel and Bouchut, Fran{\c c}ois and Bristeau, Marie-Odile and Klein, Rupert and Perthame, B.},
  year = 2004,
  journal = {SIAM J. Sci. Comput.},
  volume = {25},
  number = {6},
  pages = {2050--2065},
  publisher = {{Society for Industrial and Applied Mathematics}},
  issn = {1064-8275},
  doi = {10.1137/S1064827503431090}
}

@article{Audusse2005Multilayer,
  title = {A Multilayer {Saint-Venant} Model: Derivation and Numerical Validation},
  author = {Audusse, Emmanuel},
  year = {2005},
  journal = {Discrete and Continuous Dynamical Systems - B},
  volume = {5},
  number = {2},
  pages = {189--214},
  doi = {10.3934/dcdsb.2005.5.189}
}

@article{Audusse2011Multilayer,
  title = {A Multilayer {Saint-Venant} System with Mass Exchanges for Shallow Water Flows: Derivation and Numerical Validation},
  author = {Audusse, Emmanuel and Bristeau, Marie-Odile and Perthame, Beno{\^\i}t and Sainte-Marie, Jacques},
  year = {2011},
  journal = {ESAIM: Mathematical Modelling and Numerical Analysis},
  volume = {45},
  number = {1},
  pages = {169--200},
  doi = {10.1051/m2an/2010036}
}

@article{BouchutZeitlin2010,
  title = {A Robust Well-Balanced Scheme for Multi-Layer Shallow Water Equations},
  author = {Bouchut, Fran{\c c}ois and Zeitlin, Vladimir},
  year = {2010},
  journal = {Discrete and Continuous Dynamical Systems - B},
  volume = {13},
  number = {4},
  pages = {739--758},
  doi = {10.3934/dcdsb.2010.13.739}
}

@article{Bauerle2025,
  title = {On the Rotational Invariance and Hyperbolicity of Shallow Water Moment Equations in Two Dimensions},
  author = {Bauerle, Matthew and Christlieb, Andrew J. and Ding, Mingchang and Huang, Juntao},
  year = 2025,
  journal = {SIAM J. Math. Anal.},
  volume = {57},
  number = {1},
  pages = {1039--1085},
  publisher = {{Society for Industrial and Applied Mathematics}},
  issn = {0036-1410},
  doi = {10.1137/23M1579789}
}

@article{Bermudez1994,
  title = {Upwind Methods for Hyperbolic Conservation Laws with Source Terms},
  author = {Bermudez, Alfredo and Vazquez, Ma Elena},
  year = 1994,
  journal = {Computers \& Fluids},
  volume = {23},
  number = {8},
  pages = {1049--1071},
  issn = {0045-7930},
  doi = {10.1016/0045-7930(94)90004-3}
}

@article{Caballero-Cardenas2025,
  title = {A Semi-Implicit Exactly Fully Well-Balanced Relaxation Scheme for the {{Shallow Water Linearized Moment Equations}}},
  author = {{Caballero-C{\'a}rdenas}, C. and {G{\'o}mez-Bueno}, I. and Del Grosso, A. and Koellermeier, J. and {Morales de Luna}, T.},
  year = 2025,
  journal = {Computer Methods in Applied Mechanics and Engineering},
  volume = {437},
  pages = {117788},
  issn = {0045-7825},
  doi = {10.1016/j.cma.2025.117788}
}

@article{Cao2026,
  title = {Flux Globalization Based Well-Balanced Path-Conservative Central-Upwind Schemes for Shallow Water Linearized Moment Equations},
  author = {Cao, Yangyang and Huang, Qian and Koellermeier, Julian and Kurganov, Alexander and Liu, Yongle},
  year = 2026,
  journal = {Computers \& Fluids},
  volume = {311},
  pages = {107036},
  issn = {0045-7930},
  doi = {10.1016/j.compfluid.2026.107036}
}

@incollection{Castro2017,
  title = {Chapter 6 - {{Well-Balanced Schemes}} and {{Path-Conservative Numerical Methods}}},
  booktitle = {Handbook of {{Numerical Analysis}}},
  author = {Castro, M. J. and {Morales de Luna}, T. and Par{\'e}s, C.},
  editor = {Abgrall, R{\'e}mi and Shu, Chi-Wang},
  year = 2017,
  series = {Handbook of {{Numerical Methods}} for {{Hyperbolic Problems}}},
  volume = {18},
  pages = {131--175},
  publisher = {Elsevier},
  doi = {10.1016/bs.hna.2016.10.002}
}

@article{Castro2020,
  title = {Well-{{Balanced High-Order Finite Volume Methods}} for {{Systems}} of {{Balance Laws}}},
  author = {Castro, Manuel J. and Par{\'e}s, Carlos},
  year = 2020,
  journal = {J Sci Comput},
  volume = {82},
  number = {2},
  pages = {48},
  issn = {1573-7691},
  doi = {10.1007/s10915-020-01149-5}
}

@article{Ciallella2026,
  title = {High Order Global Flux Schemes for General Steady State Preservation of Shallow Water Moment Equations with Non-Conservative Products},
  author = {Ciallella, Mirco and Koellermeier, Julian},
  year = 2026,
  journal = {Computers \& Fluids},
  volume = {305},
  pages = {106887},
  issn = {0045-7930},
  doi = {10.1016/j.compfluid.2025.106887}
}

@article{Fan2026,
  title = {Well-{{Balanced Path-Conservative Discontinuous Galerkin Methods}} with {{Equilibrium Preserving Space}} for {{Shallow Water Linearized Moment Equations}}},
  author = {Fan, Ruilin and Koellermeier, Julian and Xia, Yinhua and Xu, Yan and Zhang, Jiahui},
  year = 2026,
  journal = {J. Sci. Comput.},
  volume = {108},
  number = {1},
  pages = {2},
  issn = {1573-7691},
  doi = {10.1007/s10915-026-03322-8}
}

@article{Garres-Diaz2023,
  title = {A General Vertical Decomposition of {{Euler}} Equations: {{Multilayer-moment}} Models},
  shorttitle = {A General Vertical Decomposition of {{Euler}} Equations},
  author = {{Garres-D{\'i}az}, J. and Escalante, C. and {Morales de Luna}, T. and Castro D{\'i}az, M. J.},
  year = 2023,
  journal = {Applied Numerical Mathematics},
  volume = {183},
  pages = {236--262},
  issn = {0168-9274},
  doi = {10.1016/j.apnum.2022.09.004}
}

@article{Gomez-Bueno2021a,
  title = {Collocation {{Methods}} for {{High-Order Well-Balanced Methods}} for {{Systems}} of {{Balance Laws}}},
  author = {{G{\'o}mez-Bueno}, Irene and D{\'i}az, Manuel Jes{\'u}s Castro and Par{\'e}s, Carlos and Russo, Giovanni},
  year = 2021,
  journal = {Mathematics},
  volume = {9},
  number = {15},
  pages = {1799},
  publisher = {Multidisciplinary Digital Publishing Institute},
  issn = {2227-7390},
  doi = {10.3390/math9151799},
  copyright = {http://creativecommons.org/licenses/by/3.0/}
}

@article{Greenberg1996,
  title = {A {{Well-Balanced Scheme}} for the {{Numerical Processing}} of {{Source Terms}} in {{Hyperbolic Equations}}},
  author = {Greenberg, J. M. and Leroux, A. Y.},
  year = 1996,
  journal = {SIAM J. Numer. Anal.},
  volume = {33},
  number = {1},
  pages = {1--16},
  publisher = {{Society for Industrial and Applied Mathematics}},
  issn = {0036-1429},
  doi = {10.1137/0733001}
}

@article{Koellermeier2020,
  title = {Analysis and Numerical Simulation of Hyperbolic Shallow Water Moment Equations},
  author = {Koellermeier, Julian and Rominger, Marvin},
  year = 2020,
  journal = {Commun. Comput. Phys.},
  volume = {28},
  number = {3},
  pages = {1038--1084},
  issn = {1815-2406, 1991-7120},
  doi = {10.4208/cicp.OA-2019-0065}
}

@article{Koellermeier2022,
  title = {Steady States and Well-Balanced Schemes for Shallow Water Moment Equations with Topography},
  author = {Koellermeier, J. and {Pimentel-Garc{\'i}a}, E.},
  year = 2022,
  journal = {Appl. Math. Comput.},
  volume = {427},
  pages = {127166},
  issn = {0096-3003},
  doi = {10.1016/j.amc.2022.127166}
}

@article{Klingenberg2019,
  title = {Arbitrary Order Finite Volume Well-Balanced Schemes for the Euler Equations with Gravity},
  author = {Klingenberg, C. and Puppo, G. and Semplice, M.},
  year = 2019,
  journal = {SIAM J. Sci. Comput.},
  volume = {41},
  number = {2},
  pages = {A695--A721},
  issn = {1064-8275, 1095-7197},
  doi = {10.1137/18M1196704}
}

@article{Koellermeier2025,
  author = {Koellermeier, Julian},
  title = {Primitive Variable Regularization to Derive Novel Hyperbolic Shallow Water Moment Equations},
  journal = {Communications in Computational Physics},
  year = {2026},
  doi = {10.4208/cicp.OA-2025-0255},
  url = {https://arxiv.org/abs/2505.17216v2}
}

@article{Kowalski2019,
  title = {Moment Approximations and Model Cascades for Shallow Flow},
  author = {Kowalski, Julia and Torrilhon, Manuel},
  year = 2019,
  journal = {Commun. Comput. Phys.},
  volume = {25},
  number = {3},
  pages = {669--702},
  issn = {18152406},
  doi = {10.4208/cicp.OA-2017-0263}
}

@article{LeVeque1998,
  title = {Balancing {{Source Terms}} and {{Flux Gradients}} in {{High-Resolution Godunov Methods}}: {{The Quasi-Steady Wave-Propagation Algorithm}}},
  shorttitle = {Balancing {{Source Terms}} and {{Flux Gradients}} in {{High-Resolution Godunov Methods}}},
  author = {LeVeque, Randall J.},
  year = 1998,
  journal = {Journal of Computational Physics},
  volume = {146},
  number = {1},
  pages = {346--365},
  issn = {0021-9991},
  doi = {10.1006/jcph.1998.6058}
}

@article{Noelle2007,
  title = {High-Order Well-Balanced Finite Volume {{WENO}} Schemes for Shallow Water Equation with Moving Water},
  author = {Noelle, Sebastian and Xing, Yulong and Shu, Chi-Wang},
  year = 2007,
  journal = {Journal of Computational Physics},
  volume = {226},
  number = {1},
  pages = {29--58},
  issn = {0021-9991},
  doi = {10.1016/j.jcp.2007.03.031}
}

@article{Pares2006b,
  title = {Numerical Methods for Nonconservative Hyperbolic Systems: A Theoretical Framework.},
  shorttitle = {Numerical Methods for Nonconservative Hyperbolic Systems},
  author = {Par{\'e}s, Carlos},
  year = 2006,
  journal = {SIAM J. Numer. Anal.},
  volume = {44},
  number = {1},
  pages = {300--321},
  publisher = {{Society for Industrial and Applied Mathematics}},
  issn = {0036-1429},
  doi = {10.1137/050628052}
}

@inproceedings{Pimentel-Garcia2024,
  title = {Fully Well-Balanced Methods for Shallow Water Linearized Moment Model with Friction},
  booktitle = {Hyperbolic Problems: {{Theory}}, Numerics, Applications. {{Volume II}}},
  author = {{Pimentel-Garc{\'i}a}, Ernesto},
  editor = {Par{\'e}s, C. and Castro, M. J. and {Morales de Luna}, T. and {Mu{\~n}oz-Ruiz}, M.},
  year = 2024,
  pages = {195--208},
  publisher = {Springer Nature Switzerland},
  address = {Cham},
  doi = {10.1007/978-3-031-55264-9_17},
  isbn = {978-3-031-55264-9}
}

@article{Steldermann2023,
  title = {Shallow {{moments}} to {{capture vertical structure}} in {{open curved shallow flow}}},
  author = {Steldermann, I and Torrilhon, M and Kowalski, J},
  year = {2023},
  journal = {J. Comput. Theor. Transp.},
  volume = {52},
  number = {7},
  pages = {475--505},
  doi = {10.1080/23324309.2023.2284202}
}

@article{Vazquez-Cendon1999,
  title = {Improved {{Treatment}} of {{Source Terms}} in {{Upwind Schemes}} for the {{Shallow Water Equations}} in {{Channels}} with {{Irregular Geometry}}},
  author = {{V{\'a}zquez-Cend{\'o}n}, Mar{\'\i}a Elena},
  year = 1999,
  journal = {Journal of Computational Physics},
  volume = {148},
  number = {2},
  pages = {497--526},
  issn = {0021-9991},
  doi = {10.1006/jcph.1998.6127}
}

@article{Xing2014,
  title = {Exactly Well-Balanced Discontinuous {{Galerkin}} Methods for the Shallow Water Equations with Moving Water Equilibrium},
  author = {Xing, Yulong},
  year = 2014,
  journal = {Journal of Computational Physics},
  volume = {257},
  pages = {536--553},
  issn = {0021-9991},
  doi = {10.1016/j.jcp.2013.10.010}
}

@misc{Zhou2025a,
  title = {Moment-Enhanced Shallow-Water Equations with an Effective Wall Closure for No-Slip Bottoms},
  author = {Zhou, Shiping and Huang, Juntao and Christlieb, Andrew J.},
  year = 2025,
  number = {arXiv:2506.14785},
  eprint = {2506.14785},
  primaryclass = {math.NA},
  publisher = {arXiv},
  note = {arXiv:2506.14785v3, revised 2 September 2026},
  doi = {10.48550/arXiv.2506.14785},
  archiveprefix = {arXiv}
}

@book{French1985,
  author    = {French, Richard H.},
  title     = {Open-Channel Hydraulics},
  publisher = {McGraw-Hill},
  year      = {1985}
}

@article{Christen2010,
  author  = {Christen, Marc and Kowalski, Julia and Bartelt, Perry},
  title   = {{RAMMS}: Numerical simulation of dense snow avalanches in three-dimensional terrain},
  journal = {Cold Regions Science and Technology},
  volume  = {63},
  number  = {1--2},
  pages   = {1--14},
  year    = {2010}
}

@article{Cantero-Chinchilla2020,
  author  = {{Cantero-Chinchilla}, Francisco N. and Bergillos, Rafael J. and Gamero, Pedro and {Castro-Orgaz}, Oscar and Cea, Luis and Hager, Willi H.},
  title   = {Vertically Averaged and Moment Equations for Dam-Break Wave Modeling: Shallow Water Hypotheses},
  journal = {Water},
  volume  = {12},
  number  = {11},
  pages   = {3232},
  year    = {2020},
  doi     = {10.3390/w12113232}
}

@article{Ferro2022,
  author  = {Ferro, Paulin and Landel, Paul and Pescheux, Marc and Guillot, Simon},
  title   = {Development of a Free Surface Flow Solver Using the Ghost Fluid Method on {OpenFOAM}},
  journal = {Ocean Engineering},
  volume  = {253},
  pages   = {111236},
  year    = {2022},
  doi     = {10.1016/j.oceaneng.2022.111236}
}

@article{Kim2011,
  author  = {Kim, Dae-Hong and Lynett, Patrick J.},
  title   = {Dispersive and Nonhydrostatic Pressure Effects at the Front of Surge},
  journal = {Journal of Hydraulic Engineering},
  volume  = {137},
  number  = {7},
  pages   = {754--765},
  year    = {2011},
  doi     = {10.1061/(ASCE)HY.1943-7900.0000345}
}

@article{Larsen2019,
  author  = {Larsen, Bjarke Eltard and Fuhrman, David R. and Roenby, Johan},
  title   = {Performance of {interFoam} on the Simulation of Progressive Waves},
  journal = {Coastal Engineering Journal},
  volume  = {61},
  number  = {3},
  pages   = {380--400},
  year    = {2019},
  doi     = {10.1080/21664250.2019.1609713}
}

@article{Berberich2021,
  author = {Berberich, Jonas P. and Chandrashekar, Praveen and Klingenberg, Christian},
  title = {High order well-balanced finite volume methods for multi-dimensional systems of hyperbolic balance laws},
  journal = {Computers \& Fluids},
  volume = {219},
  pages = {104858},
  year = {2021},
  doi = {10.1016/j.compfluid.2021.104858},
  url = {https://arxiv.org/abs/1903.05154}
}

@misc{Careaga2026,
  author = {Careaga, Julio and Ersing, Patrick and Koellermeier, Julian and Winters, Andrew R.},
  title = {Entropy analysis and entropy stable {DG} methods for the {1D} shallow water moment equations},
  year = {2026},
  note = {Preprint, arXiv:2602.06513v2, revised 23 June 2026},
  doi = {10.48550/arXiv.2602.06513},
  url = {https://arxiv.org/abs/2602.06513v2}
}

@article{DalMaso1995,
  author = {{Dal Maso}, G. and LeFloch, P. G. and Murat, F.},
  title = {Definition and weak stability of nonconservative products},
  journal = {Journal de Math{\'e}matiques Pures et Appliqu{\'e}es},
  volume = {74},
  pages = {483--548},
  year = {1995},
  url = {https://hdl.handle.net/20.500.11767/16373}
}

@article{Garres-Diaz2021,
  author = {{Garres-D{\'i}az}, Jos{\'e} and {Castro D{\'i}az}, Manuel J. and Koellermeier, Julian and {Morales de Luna}, Tom{\'a}s},
  title = {Shallow Water Moment Models for Bedload Transport Problems},
  journal = {Communications in Computational Physics},
  volume = {30},
  number = {3},
  pages = {903--941},
  year = {2021},
  doi = {10.4208/cicp.OA-2020-0152},
  url = {https://arxiv.org/abs/2008.08449}
}

@article{GomezBuenoControl2021,
  author = {{G{\'o}mez-Bueno}, Irene and {Castro D{\'i}az}, Manuel Jes{\'u}s and Par{\'e}s, Carlos},
  title = {High-order well-balanced methods for systems of balance laws: a control-based approach},
  journal = {Applied Mathematics and Computation},
  volume = {394},
  pages = {125820},
  year = {2021},
  doi = {10.1016/j.amc.2020.125820}
}

@article{XingShuNoelle2011,
  author = {Xing, Yulong and Shu, Chi-Wang and Noelle, Sebastian},
  title = {On the Advantage of Well-Balanced Schemes for Moving-Water Equilibria of the Shallow Water Equations},
  journal = {Journal of Scientific Computing},
  volume = {48},
  pages = {339--349},
  year = {2011},
  doi = {10.1007/s10915-010-9377-y},
  url = {https://arxiv.org/abs/1502.00800}
}

@misc{Zhou2026Code,
  author = {Zhou, Shiping and Huang, Juntao and Christlieb, Andrew J.},
  title = {{Two-Dimensional Shallow Water Linearized Moment Equations: Hyperbolicity and Well-Balanced Schemes}},
  year = {2026},
  howpublished = {Zenodo [software]},
  version = {v1.0},
  doi = {10.5281/zenodo.22925453},
  url = {https://doi.org/10.5281/zenodo.22925453},
  note = {Version v1.0. \url{https://doi.org/10.5281/zenodo.22925453}}
}

\end{document}